\documentclass{amsart}

\usepackage[letterpaper,margin=1.25in]{geometry}
\usepackage{amssymb}
\usepackage{mathtools}
\usepackage{enumitem}
\usepackage{booktabs}
\usepackage{xcolor}
\usepackage{hyperref}

\hypersetup{
  hidelinks,
  colorlinks=false,
  linkcolor=black,
  citecolor=black,
  urlcolor=black
}

\newcommand{\C}{\mathbb{C}}
\newcommand{\Z}{\mathbb{Z}}

\newcommand{\Ocal}{\mathcal{O}}
\newcommand{\Ecal}{\mathcal{E}}

\newcommand{\eps}{\epsilon}
\newcommand{\id}{\mathrm{id}}
\DeclareMathOperator{\Ram}{Ram}
\DeclareMathOperator{\Res}{Res}
\DeclareMathOperator{\Log}{Log}

\theoremstyle{plain}
\newtheorem{thm}{Theorem}[section]
\newtheorem{lem}[thm]{Lemma}
\newtheorem{prop}[thm]{Proposition}
\newtheorem{cor}[thm]{Corollary}

\theoremstyle{definition}

\theoremstyle{remark}
\newtheorem{rem}[thm]{Remark}

\numberwithin{equation}{section}

\begin{document}

\title{Conifold Gap Theorem for Topological Recursion}

\author{Bohan Fang}
\address{Bohan Fang, Beijing International Center for Mathematical Research, Peking University, 5 Yiheyuan Road, Beijing 100871, China}
\email{bohanfang@gmail.com}

\author{Juping Chen}
\address{Juping Chen, Department of Mathematical Sciences, Tsinghua University, Haidian District, Beijing 100084, China}
\email{cjp24@mails.tsinghua.edu.cn}

\author{Linji Chen}
\address{Linji Chen, Beijing International Center for Mathematical Research, Peking University, 5 Yiheyuan Road, Beijing 100871, China}
\email{aco@stu.pku.edu.cn}

\author{Zhengyu Zong}
\address{Zhengyu Zong, Department of Mathematical Sciences, Tsinghua University, Haidian District, Beijing 100084, China}
\email{zyzong@mail.tsinghua.edu.cn}

\date{}

\begin{abstract}
We prove a conifold gap theorem for the topological recursion of toric mirror curves: for every local analytic family acquiring a generic one-node degeneration and every fixed genus at least two, the conifold-polarized free energy has one universal polar term, and the remainder is jointly holomorphic in the transverse and spectator parameters. The result covers separating and nonseparating nodes and allows general family deformations beyond the pure filling-fraction case.
\end{abstract}

\maketitle

\setcounter{tocdepth}{1}
\tableofcontents

\section{Introduction}\label{sec:introduction}

\subsection{Background}

Topological string theory assigns to a Calabi--Yau threefold a free energy at every genus. Near a conifold point of the moduli space, where a cycle collapses and the threefold acquires an ordinary double point, physics predicts one universal local model for all of these free energies at once: a single BPS state, realized by a brane wrapping the vanishing cycle, becomes massless there \cite{Str95}, and its contribution is that of the $c=1$ string at the self-dual radius \cite{GhV95,GoV98}.

Huang and Klemm \cite{HK07} discovered, and Huang, Klemm, and Quackenbush \cite{HKQ09} developed systematically, a sharper prediction, the \emph{conifold gap condition}: in the suitably normalized flat coordinate $t$ given by the period of the vanishing cycle, the genus-$g$ free energy satisfies
\[
F_g=\frac{B_{2g}}{2g(2g-2)}\,t^{2-2g}+O(1),
\qquad g\geq2,
\]
with every intermediate negative power and every logarithm absent. The leading coefficient is classical: by Harer and Zagier \cite{HZ86}, $B_{2g}/(2g(2g-2))$ is the orbifold Euler characteristic of the moduli space of genus-$g$ curves. The gap supplies $2g-2$ boundary conditions at each conifold point of the moduli space and, as one of several boundary conditions, contributes to fixing the holomorphic ambiguity of the holomorphic anomaly equations; within this physical framework, the available boundary data determine the quintic B-model amplitudes through genus $51$ \cite{HKQ09}. As a theorem, the gap has remained rare: for the resolved conifold it is an elementary consequence of the $(-1,-1)$-curve multiple-cover formula \cite[Theorem~3]{FP00} and is made explicit by the analytic-continuation formulas of Pasquetti and Schiappa \cite[eqs.~(4.17)--(4.18)]{PS10} and Alim \cite[Theorem~4.6]{Ali25}; it was proved recently for local $\mathbb P^2$ by Brini \cite{Bri25}. On the side of Chekhov--Eynard--Orantin topological recursion \cite{CEO06,EO07}, the leading double-scaling behavior of its invariants in rational singular limits is classical; Iwaki and Kidwai \cite{IK22} proved exact all-genus formulas for the nine genus-zero spectral curves of hypergeometric type, with the higher-genus free energies $F_g$, $g\geq2$, expressed as Bernoulli-weighted sums over BPS states, building on the quantum-curve computations of Iwaki, Koike, and Takei \cite{IKoT}; and Iorgov, Iwaki, Lisovyy, and Zhuravlov \cite{IILZ25} proved the gap for the rescaled Painlev\'e I elliptic family by degeneration to the Weber curve.

\subsection{Main result}

Roughly speaking, our result is the gap condition for the topological recursion of an arbitrary toric mirror curve: for every local analytic family of toric mirror curves acquiring a generic one-node degeneration, and for every fixed genus at least two, the conifold-polarized free energy has exactly one universal polar term, and the remainder is jointly holomorphic in all parameters. Three features make this class demanding. The defining equation of a general toric mirror curve is not explicit enough for direct computation near the conifold point; the compactified curve may acquire either a separating or a nonseparating node; and, writing the conifold divisor locally as $\{q=0\}$, the deformation in the transverse $q$-direction is allowed to be general, so the theorem covers deformations beyond the pure filling-fraction case.

Because the initial form of the recursion is logarithmic, the framework that applies a priori is logarithmic topological recursion; our framing hypothesis removes its vital punctures, in the sense of \cite[Definitions~2.4 and~2.5 and Remark~2.6]{ABDKS23} and \cite[Definition~2.3 and Remark~2.3]{HMO26}, so ordinary Eynard--Orantin recursion \cite{EO07} applies throughout. We now fix the notation needed to state the result.

Let $\mathfrak X$ be a smooth toric Calabi--Yau threefold, and consider a local analytic family
\[
H(X,Y;q,s)=0
\]
of its reduced compactified mirror curves with fixed two-dimensional Newton polygon $\Delta$. Here $\{q=0\}$ is the smooth conifold divisor, the $q$-direction is transverse to it, and $s$ denotes the spectator moduli (the parameters along the divisor). We assume that the family has a generic one-node degeneration at $(q,s)=(0,s_0)$: informally, the central compactified curve has a single ordinary node in the dense torus, the $q$-direction smooths it transversely, and no other singularity or toric-boundary degeneration interferes. The precise hypotheses are stated in Proposition~\ref{thm:generic-framing} and restated in Theorem~\ref{thm:main-gap}.

Choose a good integer framing $f$ in the sense of Proposition~\ref{thm:generic-framing}; only finitely many integer framings are excluded. Put
\[
x_f=\log X+f\log Y,\qquad
y=\log Y,\qquad
\omega_{0,1}=y\,dx_f.
\]
With the conifold-adapted Torelli marking fixed before Proposition~\ref{prop:regime-atlas}, let $\gamma$ be an oriented vanishing cycle, and let $\omega_{g,1;t,s}^c$ denote the conifold-polarized correlators. We define
\[
t=\frac{1}{2\pi i}\int_\gamma\omega_{0,1},
\qquad
F_g^c(s,t)=
\frac{1}{2-2g}
\sum_{a\in\Ram(x_f)}
\operatorname*{Res}_{p=a}
\Phi_a(p)\,\omega_{g,1;t,s}^c(p),
\qquad
d\Phi_a=\omega_{0,1;t,s}.
\]
We use the Bernoulli-number convention
\[
\frac{x}{e^x-1}=\sum_{m\geq0}B_m\frac{x^m}{m!}.
\]
With this notation, our main result is the following.

\begin{thm}[{Conifold gap; $=$ Theorem~\textup{\ref{thm:main-gap}} $+$ Corollary~\textup{\ref{cor:main-calibrated}}}]\label{thm:main-gap-intro}
In the setting above, $(t,s)$ is a holomorphic coordinate system near $(0,s_0)$ in which the discriminant is $\{t=0\}$. For every fixed integer $g\geq2$, after shrinking to a product $U_{\rm sp}\times D_t$, where $U_{\rm sp}$ may be a point, there is a unique jointly holomorphic function $H_g^{\rm gap}\in\Ocal(U_{\rm sp}\times D_t)$ such that, for $t\neq0$,
\begin{equation}\label{eq:intro-gap}
F_g^c(s,t)=\frac{B_{2g}}{2g(2g-2)}\,t^{2-2g}+H_g^{\rm gap}(s,t).
\end{equation}
In particular, the Laurent expansion of $F_g^c$ in $t$ has no logarithmic term and no negative power other than $t^{2-2g}$. The shrinking and the holomorphic extension are uniform for $s$ in compact subsets of the smooth conifold divisor, and the conclusion covers both nonseparating and separating nodes.
\end{thm}

For semi-projective toric Calabi--Yau threefolds, the remodeling conjecture of Bouchard, Klemm, Mari\~no, and Pasquetti \cite{BKMP09,BKMP10,Marino08} identifies, after the mirror map and up to the standard convention-dependent signs, the Gromov--Witten potentials with the Eynard--Orantin invariants of the mirror curve. Eynard and Orantin \cite{EO15} provided a proof in the smooth case, while Fang, Liu, and Zong \cite{FLZ20} proved the general orbifold case. Theorem~\ref{thm:main-gap-intro} therefore addresses the conifold gap expectation on the mirror side for the whole toric class; the transfer to Gromov--Witten generating functions along the mirror map is deferred to future work. Compared with \cite{Bri25}, which proves the gap for local $\mathbb P^2$ by passing through a statistical-mechanical model, our proof stays within the recursion. Within the recursion, the results closest to ours are those of Iwaki and Kidwai \cite{IK22}, whose exact BPS sums exhibit the conifold-type behavior of the nine genus-zero hypergeometric-type curves at a vanishing central charge, and of Iorgov, Iwaki, Lisovyy, and Zhuravlov \cite{IILZ25}, who prove the gap for the rescaled Painlev\'e I elliptic family through series expansions on the degenerate curve, with contours chosen to avoid the colliding ramification points, together with a reduction to the Weber model. Without relying on closed formulae, our theorem treats logarithmic spectral data on toric mirror curves of arbitrary genus by analytic estimates uniform in the moving family. It allows general transverse one-node deformations with spectator moduli and covers both separating and nonseparating nodes.

\subsection{Method}

The proof has three structural inputs. First, the local node is put in a normal form compatible with the framed projection. The vanishing period itself, rather than the leading coefficient of an auxiliary square-root parameter, is used as the transverse coordinate. This gives a single coordinate system for the compact subsets away from the node, the plumbing circles, and the fixed scaled-neck charts.

Second, the normalized bidifferential is estimated in the moving family. The exterior complex structure is allowed to vary transversely. Capping records this variation in one parameter, while an independent parameter performs the plumbing. On scaled-neck charts, a simultaneous two-seam construction retains the Gaussian double pole and eliminates the annular zero mode. A nested-collar induction then estimates the residues at the two colliding ramification points (we call them \emph{cold}, and everything staying away from the node \emph{hot}) collectively rather than term by term.

Third, special geometry is applied to the family deformation itself, with the hot residual bounded:
\[
\partial_tF_g^c=I_g^{\rm neck}+R_g,\qquad R_g=O(1).
\]
The cold part of the free-energy residue sum gives
\[
(2-2g)F_g^c=tI_g^{\rm neck}+O(1).
\]
The neck terms cancel exactly, and what remains is the \emph{Euler defect}
\[
\mathcal E_g=(2-2g)F_g^c-t\partial_tF_g^c=O(1),
\]
the failure of $F_g^c$ to be homogeneous of degree $2-2g$ in the vanishing period. Single-valuedness on the full punctured $t$-disc and the Gaussian leading limit then force the gap by a Laurent-coefficient argument: boundedness of the defect removes every negative Laurent coefficient except the one at exponent $2-2g$.

We fix the toric one-node families, the good framings, and the recursion conventions in Section~\ref{sec:setup}, where the node is also put in a projection-compatible normal form and the vanishing period $t$, the plumbing data, and the Gaussian comparison curve are constructed. In Section~\ref{sec:estimates} we control the bidifferential in the moving family, bound the stable correlators uniformly, and establish the exact-period Gaussian leading limit. In Section~\ref{sec:defect} we prove special geometry for the family deformation with the hot residual bounded, descend the conifold-polarized correlators and free energies to the punctured $t$-disc, and bound the Euler defect. We assemble the gap and evaluate the universal constant against the Gaussian matrix model in Section~\ref{sec:assembly}.

\subsection{Statement on the usage of artificial intelligence}

The mathematics and the writing of this paper were assisted by a modified Danus--Rethlas system \cite{LGS26,JGJ26}, and later by custom agents built on Claude and Codex. The authors steered the system, both mathematically and operationally, and verified all generated content.

\subsection*{Acknowledgements}

The first author would like to thank Chiu-Chu Melissa Liu and Song Yu for valuable mathematical discussions on Gromov--Witten theory and topological recursion. He is grateful to Bin Dong, Jihao Liu, and the AI4Math project at Peking University for guidance on AI--math interaction. The first author also thanks Yin Wu for discussions on AI usage and agent harnesses. The fourth author would like to thank Chiu-Chu Melissa Liu and Song Yu for useful discussions.

The work of the first author is partially supported by the National Key R\&D Program of China 2023YFA1009803, NSFC 12125101, and NSFC 11890661. The work of the fourth author is partially supported by NSFC grant No. 11701315 and the Natural Science Foundation of Beijing, China grant No. 1252008.

\section{Setup and the one-neck normal form}\label{sec:setup}

This section fixes the toric one-node families with their good framings and recursion conventions, and puts the node in a projection-compatible normal form carrying the exact period and the plumbing data.

\subsection{Toric families and generic framing}

Let $\Delta\subset\mathbb R^2$ be a fixed two-dimensional lattice polygon. The curves in this
paper are the mirror curves of a smooth toric Calabi--Yau threefold $\mathfrak X$; all we use
from this provenance is the resulting class of equations, as follows. A curve of the family is
cut out of the algebraic torus $(\C^*)^2$ by a Laurent polynomial with Newton polygon $\Delta$,
\[
H(X,Y;q,s)
=
\sum_{(a,b)\in\Delta\cap\Z^2}
c_{a,b}(q,s)\,X^aY^b,
\]
whose coefficients $c_{a,b}$ are holomorphic in local coordinates $(q,s)$ on the parameter
space. The smooth conifold divisor is locally $\{q=0\}$; the $q$-direction is transverse to it,
while $s$ collects the remaining coordinates along the divisor. These play a passive role
throughout, and we call them the \emph{spectator moduli}.
Let $S$ be the toric surface determined by $\Delta$. The closure $C_{q,s}\subset S$ of the
affine curve is the \emph{compactified mirror curve}. It meets the toric boundary of $S$ in
finitely many points, the \emph{punctures}, each lying on the boundary divisor $D_e$ attached
to an edge $e$ of $\Delta$. The coefficients $c_{a,b}$ at the vertices of $\Delta$ are the
\emph{vertex coefficients}; since they are nonzero, the punctures avoid the torus-fixed points
of $S$, and each puncture lies on the smooth locus of one $D_e$. Every fiber is
\emph{reduced}: the defining Laurent polynomial has no repeated irreducible factor, so no
component of $C_{q,s}$ occurs with multiplicity. Throughout, \emph{locally uniformly in $s$}
means: uniformly for $s$ in each compact subset of the spectator neighborhood under
consideration.

On a smooth fiber, an integer \emph{framing} $f$ determines the logarithmic spectral data
\[
x_f=\log X+f\log Y,\qquad y=\log Y,\qquad
\omega_{0,1}=y\,dx_f.
\]
Because $y=\log Y$, the initial form $\omega_{0,1}$ is logarithmic, and the recursion framework
that applies a priori is logarithmic topological recursion \cite[Definitions~2.4 and~2.5 and
Remark~2.6]{ABDKS23}, \cite[Definition~2.3 and Remark~2.3]{HMO26}. We use it only through one
criterion: a \emph{vital} puncture is a simple pole of $dy$ at which $dx_f$ is regular, and if
there is no vital puncture, logarithmic and ordinary topological recursion coincide. Our choice
of framing forces $dx_f$ to have a pole with nonzero residue at every puncture, so no vital
puncture exists and ordinary Eynard--Orantin recursion applies throughout the paper.

Two finite families of framings must be avoided. First, $dx_f=d\log X+f\,d\log Y$ could have
zero residue at some puncture; each edge of $\Delta$ excludes at most one integer, and these
framings form the finite set $B_{\rm res}$. Second, the framed projection could degenerate on
the central fiber: a limiting ramification point could fail to be simple, or a nodal tangent
could become vertical for the projection; these framings form the finite set $B_{\rm proj}$.
We call an integer framing \emph{good} if it lies outside $B_{\rm res}\cup B_{\rm proj}$, and
we fix a good framing from now on. The next proposition records what this choice guarantees.
Near the nodal degeneration the ramification points split into two classes: the two \emph{cold}
ramification points converge to the node, while every other ramification point, and every
puncture, is \emph{hot} and stays a positive distance away.

\begin{prop}[Generic logarithmic framing and cold--hot separation]\label{thm:generic-framing}
Consider a local analytic family of reduced toric curves $C_{q,s}$ with fixed two-dimensional Newton polygon $\Delta$ as above, where $s$ ranges in a complex-manifold neighborhood of $s_0$ and $q$ is transverse to the smooth \emph{conifold divisor} $q=0$, the discriminant locus along which the fiber acquires its node. Assume that $C_{0,s_0}$ has exactly one ordinary node at a point of the dense torus, that the $q$-direction smooths it transversely, that the central compactified curve is otherwise smooth, that the vertex coefficients are nonzero, and that the puncture sections are pairwise disjoint. Equip the smooth fibers with a \emph{conifold-adapted Torelli marking} and the corresponding $A$-normalized fundamental bidifferential; the conifold-adapted choice is specified in Proposition~\ref{prop:regime-atlas}.

There are two finite sets of framings,
\[
B_{\rm res}\subset\Z,\qquad B_{\rm proj}\subset\Z,
\]
the first depending only on the edge normals of $\Delta$, the second determined by the two nodal tangent branches and by the critical values of the meromorphic ratio $-d\log X/d\log Y$ on the central normalization; both are constructed in the proof. For every good framing $f\notin B_{\rm res}\cup B_{\rm proj}$, fix
\[
x_f=\log X+f\log Y,\qquad y=\log Y.
\]
Then $dx_f$ has a pole with nonzero residue at every puncture and is nonzero on both normalized tangent branches at the node. Moreover there exist a neighborhood $U_f$ of $s_0$, a number $\delta_f>0$, a fixed \emph{node tube} $N_f$ whose closure is contained in the torus, and a fixed hot region $H_f$ in the compactified family, disjoint from a smaller node tube, such that for $s\in U_f$ and $0<|q|<\delta_f$:

\begin{enumerate}[label=\textup{(\arabic*)}]
\item the logarithmic spectral data, together with the conifold-adapted marking and normalized bidifferential, satisfy the admissibility hypotheses of the vital-point criterion recalled above: $dx_f$ and $dy$ are meromorphic, every ramification point is simple, and $dy$ is regular and nonzero at each one; since $dx_f$ has a pole at every puncture, there is no vital puncture, and ordinary topological recursion and logarithmic topological recursion coincide;
\item exactly two ramification points lie in $N_f$, they are simple, and they are the only ramification points converging to the node; on the double cover on which the transverse nodal coordinate is $\lambda=\eps^2$, projection-compatible coordinates (Proposition~\ref{thm:normal-form}) put the curve in the form
\[
v^2=u^2-\lambda,
\]
and the two cold points are exactly $u=\eps$ and $u=-\eps$;
\item every other ramification point and every puncture lies in $H_f$, so none enters the neck.
\end{enumerate}

The finite-set assertions concern all but finitely many integer $f$ at the fixed base point $s_0$. After one such $f$ is chosen, the neighborhood, parameter radius, and hot and cold regions may depend on $f$.

All conclusions are locally uniform in $s$. For every compact $K\subset U_f$, the radius may be chosen uniformly for $s\in K$. In fixed local coordinates and a smooth family of Hermitian metrics, the nonzero puncture residues, the nodal branch derivatives of $x_f$, the derivatives witnessing simplicity of the hot zeros, the normal-form coefficients $\lambda_q$ and $\chi_u$, and the cold--hot separation are bounded away from zero, while the relevant holomorphic coordinate coefficients are bounded above.
\end{prop}

\begin{proof}
Consider a puncture $p$. Because the vertex coefficients are nonzero, $p$
lies over the smooth locus of a boundary divisor $D_e$ of the toric surface,
away from the torus-fixed points. Let $(m_p,n_p)$ be the primitive edge
normal, normalized so that $X$ and $Y$ have orders $m_p$ and $n_p$ along
$D_e$. In a local chart with $D_e=\{z=0\}$, write $X=z^{m_p}U_X$ and
$Y=z^{n_p}U_Y$ with holomorphic units $U_X,U_Y$; on the normalization,
$z=w^{k_p}U$ with a holomorphic unit $U$, where $k_p\geq1$ is the contact
multiplicity of the branch with $D_e$. Hence
\[
\Res_p(d\log X)=k_pm_p,
\qquad
\Res_p(d\log Y)=k_pn_p,
\]
and the residue of
\[
dx_f=d\log X+f\,d\log Y
\]
at $p$ equals $k_p(m_p+fn_p)$. If $n_p\neq0$, this residue vanishes only for
$f=-m_p/n_p$, so each edge excludes at most one integer. If $n_p=0$, then
$dy$ is regular at $p$, while primitivity forces $m_p\neq0$, so the residue
of $dx_f$ is nonzero for every $f$. Let $B_{\rm res}$ be the set of integers
$f$ with $m_p+fn_p=0$ and $n_p\neq0$ at some puncture. Since $(m_p,n_p)$ is
the normal of the edge under $p$ and $\Delta$ has finitely many edges, this
set is finite and depends only on the edge normals of $\Delta$; it does not
change along the family, because the edge normals are fixed and the positive
factor $k_p$ multiplies both residues. Thus, for $f\notin B_{\rm res}$,
$dx_f$ has a pole with nonzero residue at every puncture, and no puncture is
vital.

Write $\alpha=d\log X$ and $\beta=d\log Y$ for the pullbacks to the
normalization. The ambient forms $d\log X$ and $d\log Y$ form a holomorphic
coframe on the torus, and the normalization is immersive on the smooth torus
preimage, including at the node preimages $o_\pm$. Hence $\alpha$ and $\beta$
do not vanish simultaneously there. Work one normalization component at a
time. On a component with $\beta\not\equiv0$, set
\[
r=-\frac{\alpha}{\beta}.
\]
If $r$ is nonconstant, its set of finite critical values is finite. At a zero
of $\alpha+f\beta$ in the smooth torus preimage one has $\beta\neq0$ and
\[
\alpha+f\beta=\beta(f-r),
\]
so a multiple zero can occur only when $f$ is a finite critical value of
$r$. If $r\equiv c$, then $\alpha+f\beta=(f-c)\beta$; the coframe argument
shows that $\beta$ has no zero in the smooth torus preimage, so only $f=c$ is
excluded. If $\beta\equiv0$ on the component, then $\alpha$ is nowhere zero
on the smooth torus preimage and the component contributes no exceptional
value. At each of $o_+$ and $o_-$, the affine expression $\alpha+f\beta$
vanishes for at most one finite value of $f$. Taking the union of the
componentwise values and these at most two node-branch values, and
intersecting with $\Z$, gives the finite set $B_{\rm proj}$. In particular,
for $f\notin B_{\rm proj}$, the form $\alpha+f\beta$ is nonzero at both node
preimages $o_\pm$; that is, $dx_f$ is nonzero on both normalized tangent
branches at the node.

Fix $f\notin B_{\rm res}\cup B_{\rm proj}$. Near the interior node, single-valued holomorphic branches of $\log X$ and $\log Y$ exist, and $(x_f,y)$ are holomorphic ambient coordinates because
\[
dx_f\wedge dy=d\log X\wedge d\log Y
\]
is nowhere zero in the torus. An ordinary plane node has nondegenerate Hessian. Nonvanishing of $dx_f$ on both nodal branches says that both tangent lines are nonvertical for the framed projection $x_f$. Transverse smoothing gives a nonzero transverse derivative of a local defining equation. Proposition~\ref{thm:normal-form} therefore supplies coordinates
\[
v^2=u^2-\lambda,
\qquad
\lambda_q(0,s_0)\neq0,
\qquad
\chi_u\neq0.
\]
On $\lambda=\eps^2\neq0$, the only local ramification points are $u=\eps$ and $u=-\eps$, and both are simple.

On the central normalization, away from the two node preimages and the punctures, $dx_f$ has finitely many zeros, all simple by the choice of $f$. Choose pairwise disjoint coordinate discs around them. The parameter-dependent holomorphic implicit-function theorem continues each zero uniquely to a simple holomorphic ramification section after shrinking the base. Choose puncture tubes disjoint from those discs and from the node chart. In a puncture coordinate $z$,
\[
z\,dx_f=r_p\,dz+O(z)\,dz,
\qquad r_p\neq0.
\]
After shrinking, no zero of $dx_f$ lies in a puncture tube.

Remove the node chart, the hot-zero discs, and the puncture tubes from the central normalization. The remainder is compact and $dx_f$ is nowhere zero there. Away from the node, the total curve family is a holomorphic submersion. A finite collection of relative coordinate charts identifies a neighborhood of this compact remainder with nearby fibers. The coefficient of $dx_f$ varies holomorphically in those charts, so its positive minimum on the central compact remainder persists after shrinking. Thus no additional zero appears. Exactly two zeros occur in the node tube, and every other zero is one of the persistent hot sections.

Every ramification point lies in the torus because $dx_f$ has a pole at every puncture. If $T$ is a nonzero tangent vector at a ramification point, then $dx_f(T)=0$. Since $(dx_f,dy)$ is an ambient coframe, $dy(T)\neq0$. Thus $dy$ is regular and nonzero at every zero of $dx_f$. The vital-point criteria in \cite[Definitions~2.4 and~2.5 and Remark~2.6]{ABDKS23} and \cite[Definition~2.3 and Remark~2.3]{HMO26} now show that the vital set is empty and logarithmic topological recursion equals ordinary topological recursion.

Choose nested node tubes with closures inside the torus. The normal form puts both cold points in the smaller tube for sufficiently small $q$. The compactness argument puts every other zero outside the larger tube, while all punctures lie on the toric boundary. The complement of the smaller tube is the required fixed hot region.

Finally, let $K\subset U_f$ be compact. All residue functions, branch derivatives, hot-zero simplicity derivatives, $\lambda_q$, $\chi_u$, coordinate changes, and ramification sections depend continuously, and holomorphically where appropriate, on $s$. Their relevant central values are nonzero. A finite-cover compactness argument gives a common parameter radius, positive lower bounds for their absolute values, upper bounds for coordinate coefficients, and a positive distance between the chosen node tubes and the hot sections. The construction begins only after $f$ is fixed; every subsequent choice may depend on it.
\end{proof}

We now fix the recursion conventions of the paper once and for all. Let $B$ be the fundamental bidifferential (the Bergman kernel of \cite{EO07}), normalized by the $A$-cycles of the conifold-adapted marking. At every simple ramification point $a$ of $x_f$, let $\sigma_a$ be the local deck involution. The ordinary recursion kernel is
\begin{equation}\label{eq:EO-kernel}
K_a(p_0,q)=
\frac{\frac12\int_{\sigma_a(q)}^qB(p_0,\cdot)}
{\omega_{0,1}(q)-\omega_{0,1}(\sigma_a(q))}.
\end{equation}
We call its numerator the \emph{cross-sheet primitive} $J_a(p_0,q)$.
The correlators are initialized by $\omega_{0,1}$ and $\omega_{0,2}=B$ and, for $2g-2+n>0$, satisfy
\begin{align}
\omega_{g,n}(p_0,\mathbf p)
={}&
\sum_{a\in\Ram(x_f)}
\operatorname*{Res}_{q=a}K_a(p_0,q)
\bigg[
\mathbf 1_{g\geq1}\,\omega_{g-1,n+1}(q,\sigma_a(q),\mathbf p)
\nonumber\\
&\hspace{25mm}
+\sum_{\substack{g_1+g_2=g\\I\sqcup J=\mathbf p}}^{\prime}
\omega_{g_1,|I|+1}(q,I)
\omega_{g_2,|J|+1}(\sigma_a(q),J)
\bigg].
\label{eq:EO-recursion}
\end{align}
Here $\mathbf p=(p_2,\ldots,p_n)$ collects the external arguments other than the distinguished one, the prime omits factors of type $(0,1)$, and the indicator $\mathbf 1_{g\geq1}$ removes the first term for $g=0$. We write $\mathcal S_{g,n}(q,\sigma_a(q);\mathbf p)$ for the bracket in \eqref{eq:EO-recursion} and call it the \emph{recursion source}; its first term is the \emph{genus-reduction term} and the remaining ones are the \emph{splitting terms}.

\subsection{The one-neck normal form and exact period}\label{sec:normal-form}

The first step is a normal form for the node which is compatible with the framed projection, the map $(x,y)\mapsto x$; in the toric application this is the framed projection $x_f$. The coordinate $u$ below depends on $x$ alone, so this projection and its deck involution are normalized simultaneously with the curve; this compatibility is what the recursion analysis of the later sections uses. Throughout this subsection, all data are holomorphic on a neighborhood of the indicated base point; after a finite shrinking, all constructions below take place on the resulting neighborhoods. For a constant $y_0$, put $\omega=(y_0+y)\,dx$.

\begin{prop}[Projection-compatible analytic normal form]\label{thm:normal-form}
Let $\mathcal S$ be a complex manifold with base point $s_0$, let $q$ be a complex parameter, and let $F(x,y;q,s)$ be holomorphic near
\[
(x,y;q,s)=(0,0;0,s_0).
\]
Assume:

\begin{enumerate}[label=\textup{(\roman*)}]
\item $F=F_x=F_y=0$ at the base point, and the $2\times2$ Hessian matrix of $F$ in $(x,y)$ is invertible there;
\item the two distinct tangent lines of $F(x,y;0,s_0)=0$ are both nonvertical for the projection $(x,y)\mapsto x$;
\item $F_q(0,0;0,s_0)\neq0$.
\end{enumerate}

After shrinking the neighborhoods:

\begin{enumerate}[label=\textup{(\arabic*)}]
\item There are a holomorphic function $\lambda(q,s)$, holomorphic coordinates
\[
u=u(x;q,s),\qquad v=v(x,y;q,s),
\]
and a holomorphic inverse $x=\chi(u;q,s)$ such that
\[
\lambda(0,s_0)=0,\qquad
\lambda_q(0,s_0)\neq0,\qquad
\chi_u\neq0,
\]
and the curve $F=0$ is exactly
\[
v^2=u^2-\lambda.
\]
\item On the cover $\lambda=\eps^2$, the deck involution is $\sigma(u,v)=(u,-v)$, and for every $\eps\neq0$ the projection $x$ has precisely the two simple local ramification points $u=\eps$ and $u=-\eps$.
\item The deck-odd part $\omega_{\rm odd}=\frac{\omega-\sigma^*\omega}{2}$ has the form
\[
\omega_{\rm odd}=M(u,\lambda,s)v\,du,
\qquad
M(u,\lambda,s)=\chi_u(u;q(\lambda,s),s),
\qquad
M(0,0,s_0)\neq0,
\]
where $q(\lambda,s)$ is the local holomorphic inverse of $q\mapsto\lambda(q,s)$.
\end{enumerate}
\end{prop}

\begin{proof}
Let $Q(x,y)$ be the homogeneous quadratic term in the Taylor expansion of $F(x,y;0,s_0)$. Assumption (i) says that $Q$ is a nondegenerate binary quadratic form, so its zero set is the union of two distinct lines. The vertical tangent vector is $(0,1)$, and
\[
Q(0,1)=\frac12F_{yy}(0,0;0,s_0).
\]
Thus assumption (ii) is equivalent to $F_{yy}(0,0;0,s_0)\neq0$.

The Weierstrass preparation theorem, applied to $F$ in the variable $y$, gives a nowhere-vanishing holomorphic function $U(x,y;q,s)$ and holomorphic functions $a(x;q,s)$ and $b(x;q,s)$, all defined near the base point, such that
\[
F(x,y;q,s)
=
U(x,y;q,s)\bigl[y^2+a(x;q,s)y+b(x;q,s)\bigr].
\]
At the base point, the equalities $F=F_y=F_x=0$ and $U\neq0$ imply
\[
b=0,\qquad a=0,\qquad b_x=0.
\]
Define
\[
v=y+\frac{a(x;q,s)}2,
\qquad
D(x;q,s)=\frac{a(x;q,s)^2}{4}-b(x;q,s).
\]
Then $F=0$ is equivalent to $v^2=D$. At the base point,
\[
D=0,\qquad D_x=0.
\]
The tangent cone of $v^2-D(x;0,s_0)=0$ is
\[
v^2-\frac12D_{xx}(0;0,s_0)x^2.
\]
It has two distinct lines because the original singularity is an ordinary node and the change from $y$ to $v$ is invertible. Hence
\[
D_{xx}(0;0,s_0)\neq0.
\]

The holomorphic implicit-function theorem applied to $D_x$ gives a unique holomorphic $c(q,s)$, with $c(0,s_0)=0$, such that
\[
D_x(c(q,s);q,s)=0.
\]
Taylor's formula with integral remainder gives
\[
D(x;q,s)-D(c(q,s);q,s)
=
(x-c(q,s))^2A(x;q,s),
\]
where
\[
A(x;q,s)
=
\int_0^1(1-r)
D_{xx}\bigl(c(q,s)+r(x-c(q,s));q,s\bigr)\,dr.
\]
This is holomorphic near the base point and
\[
A(0;0,s_0)=\frac12D_{xx}(0;0,s_0)\neq0.
\]
After shrinking, $A$ is nowhere zero and has a holomorphic square root $\alpha$. Set
\[
u=(x-c(q,s))\alpha(x;q,s),
\qquad
\lambda(q,s)=-D(c(q,s);q,s).
\]
Since $u_x=\alpha$ at $x=c$ and $\alpha\neq0$, the parameter-dependent inverse-function theorem gives a holomorphic inverse $x=\chi(u;q,s)$ with $\chi_u$ nowhere zero. The equation becomes
\[
v^2=u^2-\lambda.
\]

At the base point $a=b=0$. Differentiating
\[
F=U(y^2+ay+b)
\]
in $q$ at $x=y=0$ gives
\[
F_q=U b_q.
\]
Also $D_q=-b_q$ there. Because $D_x=0$, differentiating
\[
\lambda=-D(c(q,s);q,s)
\]
gives
\[
\lambda_q=-D_q=b_q=\frac{F_q}{U}.
\]
Assumption (iii) implies $\lambda_q(0,s_0)\neq0$. The equation
$v^2=u^2-\lambda$ is singular exactly at $u=v=\lambda=0$, so $\lambda=0$ is its local discriminant.

Fix $\lambda=\eps^2\neq0$. The curve is a double cover of the $u$-disc with ramification locus $v=0$, that is, with ramification points $(u,v)=(\pm\eps,0)$ over the branch values $u=\pm\eps$. At either ramification point the curve is smooth because the derivative of
\[
v^2-u^2+\eps^2
\]
with respect to $u$ is $-2u\neq0$. Using $v$ as a local coordinate,
\[
u=\pm\eps+\frac{v^2}{2(\pm\eps)}+O(v^4).
\]
Thus $u$, and therefore $x=\chi(u;q,s)$, has a simple ramification point. There are no other local ramification points because $\chi_u\neq0$ and the only zeros of $du$ on the double cover occur at $v=0$.

On the curve,
\[
y=v-\frac12a(\chi(u;q,s);q,s)=e(u;q,s)+v,
\]
where $e$ is deck-even and holomorphic. Since
\[
x=\chi(u;q,s),\qquad dx=\chi_u\,du,
\]
the involution fixes $u$ and negates $v$. Therefore
\begin{align*}
\frac{\omega-\sigma^*\omega}{2}
&=
\frac{(y_0+e+v)-(y_0+e-v)}{2}\chi_u\,du\\
&=\chi_u v\,du.
\end{align*}
Consequently $M(u,\lambda,s)=\chi_u(u;q(\lambda,s),s)$ is holomorphic and
$M(0,0,s_0)\neq0$.

\end{proof}

At a point $(X_0,Y_0)$ of the dense torus, choose single-valued holomorphic branches of $\log X$ and $\log Y$ on a small bidisc. The coordinates
\[
x=(\log X-\log X_0)+f(\log Y-\log Y_0),
\qquad
y=\log Y-\log Y_0
\]
have Jacobian determinant one with respect to
$(\log X-\log X_0,\log Y-\log Y_0)$. Hence the proposition applies to the logarithmic toric spectral data whenever its three hypotheses hold.

Let
\[
G=\{w^2=z^2-1\}
\]
be the fixed Gaussian spectral curve, uniformized by the Zhukovsky coordinate $\zeta$ through
\[
z=\frac{\zeta+\zeta^{-1}}2,\qquad
w=\frac{\zeta-\zeta^{-1}}2,
\]
and let $\gamma_G=\{|\zeta|=1\}$ be the unit circle, oriented counterclockwise. Put
\[
\kappa=\int_{\gamma_G}w\,dz.
\]
Up to normalization of the initial one-form, $G$ is the spectral curve of the Gaussian Hermitian one-matrix model; the identification is used for the calibration in Section~\ref{sec:calibration}.

In the setting of Proposition~\ref{thm:normal-form}, orient the local vanishing cycle $\gamma$ and define
\[
P(\lambda,s)=\int_\gamma\omega.
\]
Fix $c_{\rm per}\in\C^*$ and put $t=c_{\rm per}P$. For $\eps^2=\lambda$, let
\[
\Phi_\eps\colon(z,w)\longmapsto(u,v)=(\eps z,\eps w);
\]
on every fixed compact subset of $G$ disjoint from $\zeta=0$ and $\zeta=\infty$, the map $\Phi_\eps$ takes values in the normal-form neighborhood for all sufficiently small $\eps$.

\begin{prop}[Holomorphic transverse period and exact Gaussian decomposition]\label{thm:period-split}
\begin{enumerate}[label=\textup{(\arabic*)}]
\item $P$ is a single-valued holomorphic function of $(\lambda,s)$ near $(0,s_0)$, including $\lambda=0$, and
\[
P(\lambda,s)
=
\kappa M(0,0,s)\lambda+O(\lambda^2),
\]
with $\kappa\neq0$. For the counterclockwise orientation of $\gamma_G$ fixed above, $\kappa=-\pi i$; reversing the orientation changes its sign.
\item The map $(q,s)\mapsto(t,s)$ is locally biholomorphic at $(0,s_0)$, so $t$ is a single-valued holomorphic coordinate transverse to the nodal discriminant.
\item With $\lambda=\eps^2$, there is a decomposition
\[
\omega_{\rm odd}
=
\omega_G(t)+R_\omega,
\]
in which $\omega_G(t)$ is the product of $v\,du$ with a holomorphic function of $(\lambda,s)$ and satisfies
\[
\int_\gamma\omega_G(t)=P,
\qquad
\Phi_\eps^*\omega_G(t)=\frac{t}{c_{\rm per}\kappa}w\,dz,
\]
and in which the residual $R_\omega$ has zero $\gamma$-period and satisfies, on every fixed compact subset of $G$ disjoint from $\zeta=0$ and $\zeta=\infty$,
\[
\Phi_\eps^*R_\omega=\eps^3\Delta_\omega,
\]
with $\Delta_\omega$ holomorphic jointly in the point of the compact subset, $\eps$, and $s$, and with
\[
\int_{\gamma_G}\Delta_\omega=0.
\]
\end{enumerate}
\end{prop}

\begin{proof}
Proposition~\ref{thm:normal-form} gives
\[
v^2=u^2-\lambda,\qquad
\omega_{\rm odd}=M(u,\lambda,s)v\,du.
\]
The map $\Phi_\eps$ fixed above carries $G$ into the curve $\{v^2=u^2-\lambda\}$. Replacing $\eps$ by $-\eps$ composes $\Phi_\eps$ with
\[
\tau(z,w)=(-z,-w).
\]
Under the displayed uniformization, $\tau$ is $\zeta\mapsto-\zeta$, which preserves the counterclockwise orientation. Thus the oriented image
\[
\gamma=\Phi_\eps(\gamma_G)
\]
is independent of the sign of $\eps$.

The deck involution corresponds on $G$ to
\[
(z,w)\mapsto(z,-w),
\]
or $\zeta\mapsto\zeta^{-1}$ on the unit circle. Hence
\[
\sigma_*\gamma=-\gamma.
\]
Let
\[
\eta=\frac{\omega+\sigma^*\omega}{2}.
\]
Since $\sigma^*\eta=\eta$,
\[
\int_\gamma\eta
=
\int_\gamma\sigma^*\eta
=
\int_{\sigma_*\gamma}\eta
=
-\int_\gamma\eta.
\]
Therefore $\int_\gamma\eta=0$ and
\[
P=\int_\gamma\omega_{\rm odd}.
\]

Pulling back the odd part gives the lifted period
\[
\widetilde P(\eps,s)
=
\eps^2\int_{\gamma_G}
M(\eps z,\eps^2,s)w\,dz.
\]
The contour is fixed and compact, so this expression is holomorphic in $(\eps,s)$. Under $\eps\mapsto-\eps$, make the orientation-preserving change
\[
(z,w)\mapsto(-z,-w).
\]
The differential $w\,dz$ is invariant, and
$M((-\eps)z,\eps^2,s)$ becomes $M(\eps z,\eps^2,s)$. Hence
\[
\widetilde P(-\eps,s)=\widetilde P(\eps,s).
\]
Its convergent Taylor series contains only even powers, so it defines a unique single-valued holomorphic function $P(\lambda,s)$ near $(0,s_0)$ with $\widetilde P(\eps,s)=P(\eps^2,s)$.

The uniformization gives
\[
w\,dz
=
\frac14(\zeta-2\zeta^{-1}+\zeta^{-3})\,d\zeta.
\]
Its residue at $\zeta=0$ is $-1/2$, so
\[
\kappa=2\pi i\left(-\frac12\right)=-\pi i.
\]
Expanding $M$ at $\eps=0$ gives
\[
\widetilde P(\eps,s)
=
\eps^2\kappa M(0,0,s)+O(\eps^4),
\qquad
P(\lambda,s)
=
\kappa M(0,0,s)\lambda+O(\lambda^2).
\]
Since $M(0,0,s_0)$, $\kappa$, $\lambda_q(0,s_0)$, and $c_{\rm per}$ are nonzero,
\[
\partial_qt(0,s_0)
=
c_{\rm per}\kappa M(0,0,s_0)\lambda_q(0,s_0)
\neq0.
\]
The holomorphic inverse-function theorem applied to
$(q,s)\mapsto(t,s)$ proves local biholomorphy and transversality. The evenness argument shows that $t$ descends from the $\eps$-cover to a single-valued holomorphic function on the $\lambda$-disc.

Since $P$ is holomorphic and divisible by $\lambda$, define
\[
A_{\rm per}(\lambda,s)
=
\frac{P(\lambda,s)}{\kappa\lambda}
\]
for $\lambda\neq0$ and set
\[
A_{\rm per}(0,s)=M(0,0,s).
\]
The expansion above proves that this extension is holomorphic. Define on the unscaled neck
\[
\omega_G(t)
=
A_{\rm per}(\lambda,s)v\,du
=
\frac{t}{c_{\rm per}\kappa\lambda}v\,du.
\]
The apparent quotient $t/\lambda$ is holomorphic. Its $\gamma$-period is
\[
\int_\gamma\omega_G(t)=A_{\rm per}\kappa\lambda=P.
\]
Its pullback to $G$ is
\[
\Phi_\eps^*\omega_G(t)
=
A_{\rm per}\eps^2w\,dz
=
\frac{P}{\kappa}w\,dz
=
\frac{t}{c_{\rm per}\kappa}w\,dz.
\]

Let
\[
R_\omega=\omega_{\rm odd}-\omega_G(t).
\]
Its $\gamma$-period is zero. On $G$,
\[
\Phi_\eps^*R_\omega
=
\eps^2\bigl[
M(\eps z,\eps^2,s)-A_{\rm per}(\eps^2,s)
\bigr]w\,dz.
\]
The bracket is holomorphic in $(z,\eps,s)$ and vanishes at $\eps=0$. It is therefore $\eps N(z,\eps,s)$ for a holomorphic $N$. Put
\[
\Delta_\omega=Nw\,dz.
\]
Then
\[
\Phi_\eps^*R_\omega=\eps^3\Delta_\omega.
\]
The differential $w\,dz$ is holomorphic at the ramification points $z=\pm1$: using $w$ as a local coordinate and
\[
2w\,dw=2z\,dz
\]
gives
\[
w\,dz=\frac{w^2}{z}\,dw.
\]
Thus $\Delta_\omega$ is holomorphic on every fixed compact subset of $G$ disjoint from $\zeta=0$ and $\zeta=\infty$, including at the ramification points. For $\eps\neq0$, the zero period of $R_\omega$ gives
\[
\int_{\gamma_G}\Delta_\omega=0
\]
after division by $\eps^3$, and holomorphic dependence extends this equality to $\eps=0$.

The local decomposition isolates the deck-odd initial form; the global argument also includes the deck-even part and bounded exterior contributions.
\end{proof}

We now fix the plumbing coordinates, contours, and conventions of the neck, used throughout the paper. Work in the one-node toric family, fixed good framing, and cold--hot separation of Proposition~\ref{thm:generic-framing}, and in the projection-compatible exact-period setting of Proposition~\ref{thm:period-split}. Define plumbing coordinates
\[
z_+=u+v,\qquad z_-=u-v.
\]
Fix $R>0$ so that both plumbing circles defined below lie inside the fixed node tube of Proposition~\ref{thm:generic-framing} but outside its smaller node tube, and shrink the parameter neighborhood so that $0<|\eps|<R$ on every smooth fiber under consideration. The cold plumbing annulus is
\[
A_{\rm cold}
=
\left\{
\frac{|\eps|^2}{R}\leq|z_+|\leq R
\right\}.
\]
Its outer circle is
\[
C_{\rm out}=\{|z_+|=R\}
\]
in the $z_+$-coordinate, and its inner circle is
\[
C_{\rm in}=\{|z_-|=R\}
\]
in the $z_-$-coordinate.

Orient $C_{\rm out}$ counterclockwise in $z_+$ and $C_{\rm in}$ counterclockwise in $z_-$; we call these the plumbing orientations. Let $\delta_{\rm neck}$ be an oriented radial cross-cut from $C_{\rm out}$ to $C_{\rm in}$, disjoint from the ramification points.

The fixed scaled-neck regime is obtained from
\[
u=\eps z,\qquad v=\eps w
\]
on compact subsets of $G$, in the Zhukovsky coordinate $\zeta$ fixed before Proposition~\ref{thm:period-split}; the image under $\Phi_\eps$ of the counterclockwise circle $\gamma_G=\{|\zeta|=1\}$ is the vanishing cycle $\gamma$.

A compact subset of the normalization of the central nodal curve is called \emph{hot} if it stays a positive distance from the two node preimages and all limiting ramification points.

Finally, cut $A_{\rm cold}$ along $\delta_{\rm neck}$. On the cut annulus, choose a primitive $\Phi_{\rm cut}$ of $\omega$ and a branch $\Log_{\rm cut}$ of the logarithm. Here $\omega=(y_0+y)\,dx$ is the initial form $\omega_{0,1}$ written in the normal-form coordinates of Proposition~\ref{thm:normal-form}.

We impose the following familywise condition on the conifold-adapted marking: on every simply connected chart in the punctured parameter neighborhood of Proposition~\ref{thm:generic-framing}, if $\gamma$ is nonseparating, it is the first $A$-cycle and its symplectic dual is represented by $\delta_{\rm neck}$ completed by a family of paths in a fixed compact subset of the hot region, disjoint from the limiting ramification points and punctures; if $\gamma$ is separating, the componentwise limiting $A$-markings are chosen familywise on each such chart, and $\delta_{\rm neck}$ is regarded as a relative generalized cycle.

\begin{prop}[Exact-period one-neck plumbing data]\label{prop:regime-atlas}
In the setting just fixed, the assertions below hold for $0<|\eps|<R$, with their central-fiber counterparts at $\eps=0$ understood as limits; the holomorphic identity in \textup{(3)} holds at $\eps=0$ as well.

\begin{enumerate}[label=\textup{(\arabic*)}]
\item $z_+z_-=\eps^2$, and the deck involution interchanges $z_+$ and $z_-$; in the Zhukovsky coordinate, $z_+=\eps\zeta$, $z_-=\eps\zeta^{-1}$, and the deck involution is $\zeta\mapsto\zeta^{-1}$.
\item $C_{\rm out}$ and $C_{\rm in}$ are fixed circles in the two normalization charts, and both are hot compact subsets, whereas a fixed scaled-neck compact subset is a distinct $\eps$-dependent regime.
\item Locally uniformly in $s$,
\[
t=\eps^2c_{\rm neck}(\eps^2,s)
\]
for a holomorphic function $c_{\rm neck}$ bounded away from zero; hence $|t|$ is comparable to $|\eps|^2$.
\item Every Gaussian comparison differential is parameterized by the exact period $t$; recall from Proposition~\ref{thm:period-split} that the residual $R_\omega$ has zero $\gamma$-period and the scaled correction $\Delta_\omega$ has zero $\gamma_G$-period.
\item With the plumbing orientations, the annulus boundary is
\[
C_{\rm out}+C_{\rm in},
\]
and, with the complex orientation of the annulus,
\[
\gamma\cdot\delta_{\rm neck}=+1.
\]
\item Analytic continuation of the marking once around $q=0$ fixes $\gamma$ and, in the nonseparating case, changes the completed dual by the Picard--Lefschetz monodromy, which adds $\gamma$ or $-\gamma$, according to the loop-orientation convention, so the chosen $A$-polarization is preserved.
\item Positive continuation around $\gamma$ changes $\Phi_{\rm cut}$ by
\[
P=\frac{t}{c_{\rm per}}.
\]
Equivalently, the $dz_+/z_+$ term in the Laurent expansion of $\omega$ is
\[
\frac{t}{2\pi i c_{\rm per}}\frac{dz_+}{z_+},
\]
so its primitive contributes
\[
\frac{t}{2\pi i c_{\rm per}}\Log_{\rm cut}(z_+).
\]
\end{enumerate}
\end{prop}

\begin{proof}
The product identity follows from
\[
z_+z_-=u^2-v^2=\eps^2.
\]
The deck involution fixes $u$ and negates $v$, so it interchanges $z_+$ and $z_-$. Since
\[
z+w=\zeta,\qquad z-w=\zeta^{-1},
\]
scaling by $\eps$ gives
\[
z_+=\eps\zeta,\qquad z_-=\eps\zeta^{-1}.
\]
The involution $w\mapsto-w$ is therefore $\zeta\mapsto\zeta^{-1}$. The ramification points are $\zeta=1$ and $\zeta=-1$, and $\gamma_G$ is the counterclockwise unit circle.

An $R$ as fixed above exists: take it smaller than the radii of the normalization coordinate discs, small enough that both plumbing circles lie inside the fixed node tube of Proposition~\ref{thm:generic-framing}, and large enough that both lie outside its smaller node tube; both circles then lie in the fixed hot region. If $0<|\eps|<R$, then, since the two ramification points sit at $z_+=\pm\eps$, so that $|z_+|=|\eps|$ there, the inequalities
\[
\frac{|\eps|^2}{R}<|\eps|<R
\]
put both ramification points in the interior of $A_{\rm cold}$. The equations
\[
|z_+|=R,\qquad |z_-|=R
\]
are the two boundary equations because $z_+z_-=\eps^2$. Each uses a fixed coordinate and fixed radius around one node preimage, so each converges to a fixed circle away from that preimage; since the node tube stays a positive distance from every other limiting ramification point by Proposition~\ref{thm:generic-framing}, each is a hot compact subset. In contrast, a fixed compact subset of $G$ has unscaled diameter tending to zero under
$(z,w)\mapsto(\eps z,\eps w)$.

Proposition~\ref{thm:period-split} gives
\[
P(\eps^2,s)
=
\kappa M(0,0,s)\eps^2+O(\eps^4).
\]
Therefore
\[
c_{\rm neck}(\eps^2,s)
=
\frac{t}{\eps^2}
=
c_{\rm per}\frac{P(\eps^2,s)}{\eps^2}
\]
extends holomorphically and has nonzero value
$c_{\rm per}\kappa M(0,0,s)$ at $\eps=0$. Compactness gives positive upper and lower bounds for its absolute value, locally uniformly in $s$.

In the complex $z_+$-coordinate, the positive orientation of the annulus is the outward radial direction followed by the counterclockwise angular direction; the tangent to $\gamma$ is counterclockwise angular and $\delta_{\rm neck}$ points radially inward, so the ordered pair of their tangents has positive determinant. Hence
\[
\gamma\cdot\delta_{\rm neck}=+1.
\]

The induced boundary orientation is counterclockwise on the outer circle and clockwise on the inner circle when both are viewed in $z_+$. Since
\[
z_-=\frac{\eps^2}{z_+},
\]
inversion reverses angular orientation. A counterclockwise traversal in the $z_-$-coordinate is therefore clockwise in $z_+$ and already has the induced inner-boundary orientation. Thus the boundary is $C_{\rm out}+C_{\rm in}$.

The familywise condition can always be arranged: away from the node the family is a proper holomorphic submersion, so over each simply connected chart of the punctured parameter neighborhood the smooth fibers form a topologically locally trivial family, and marking data chosen at one base parameter transport uniquely over the chart. On each simply connected parameter chart in the nonseparating case, the completed $\delta_{\rm neck}$ has intersection $+1$ with $\gamma$ by the preceding computation, and is therefore the chosen symplectic dual, and its hot completion is a family of paths in the fixed hot region by hypothesis. Analytic continuation once around $q=0$ fixes $\gamma$ and changes the completed dual by the Picard--Lefschetz monodromy, which adds $\gamma$ or $-\gamma$, so the chosen $A$-polarization is preserved. If $\gamma$ is separating, it vanishes in absolute homology and cannot be an $A$-cycle; the condition instead supplies the familywise componentwise limiting $A$-markings, with $\delta_{\rm neck}$ regarded as a relative generalized cycle.

On the radially cut annulus, $\omega$ has a single-valued primitive. Analytic continuation of a primitive around a closed loop changes it by the integral of its differential around that loop. Hence continuation around $\gamma$ changes $\Phi_{\rm cut}$ by
\[
\int_\gamma\omega=P=\frac{t}{c_{\rm per}}.
\]
A form $a\,dz_+/z_+$ has integral $2\pi ia$ over $\gamma$. Matching this with $P$ gives
\[
a=\frac{P}{2\pi i}
=\frac{t}{2\pi i c_{\rm per}}.
\]
Its primitive is $a\Log_{\rm cut}(z_+)$, proving the Laurent-coefficient formula.
\end{proof}

\section{Uniform estimates and the Gaussian leading limit}\label{sec:estimates}

This section carries the analytic estimates of the paper: the bidifferential of the degenerating family is controlled with transverse variation of the outer family, the stable correlators are bounded uniformly, and the exact-period Gaussian leading limit is established on the quadratic cover.

\subsection{The degenerating bidifferential}\label{sec:bergman}

The sewing formulas for normalized bidifferentials are taken at the separating and nonseparating levels supplied by
\cite[Theorems~2, 4, and~6]{Yam80}; the expansion of the additional normalized one-form of the nonseparating degeneration is \cite[Corollary~5]{Yam80}. We apply them when the outer family varies in the transverse direction. The proofs use the fixed-component Cauchy-kernel jump construction introduced by Grushevsky, Krichever, and Norton \cite{GKN19}, in the holomorphic $A$-normalized form developed by Hu and Norton \cite[Section~3]{HN20}.

\begin{thm}[Sewing with transverse variation of the outer family]\label{thm:actual-yamada}
Let $C_{q,s}$ be a proper local holomorphic family of compact curves over a $q$-disc times a spectator polydisc, smooth for $q\neq0$, such that $C_{0,s}$ has one ordinary node varying holomorphically with $s$ and is otherwise smooth. Assume that fixed relative node charts give
\[
z_+z_-=q,
\]
that there are no other nearby singularities, and that the conifold-adapted marking is locally constant outside the node charts.

After shrinking the base and choosing fixed nested node radii, there is a proper holomorphic capped outer family
\[
C^{\rm cap}_{r,s},
\]
smooth also at $r=0$, with two marked cap centers and product cap coordinates. It is connected in the nonseparating case and has two components in the separating case. Its central fiber is the normalization of $C_{0,s}$.

If its cap centers are plumbed with an independent parameter $p$, the resulting family
\[
C^{\rm sew}_{p,r,s}
\]
is biholomorphic, with marking and plumbing coordinates preserved, to the original $C_{q,s}$ on the diagonal $p=r=q$.

Let $B^{\rm cap}(r,s)$ be the componentwise $A$-normalized bidifferential of the capped family, let
\[
B_0(s)=B^{\rm cap}(0,s),
\qquad
D^{\rm cap}(s)=
\left.\partial_rB^{\rm cap}(r,s)\right|_{r=0},
\]
and let $b_{\pm,0}(x)=\bigl(B_0(x,z_\pm)/dz_\pm\bigr)\big|_{z_\pm=0}$ be the node-evaluation one-forms of $B_0$, taken in the fixed cap coordinates. On every fixed product of hot compact subsets, including products containing either fixed plumbing circle, the normalized bidifferential $B_q$, written $B_q=B_{q,s}$ with the spectator moduli suppressed, has the following expansions, after subtracting the common diagonal principal part when necessary.

In the nonseparating case,
\begin{equation}\label{eq:yamada-nonsep}
B_q
=
B_0
+
q\bigl[
D^{\rm cap}
-b_{+,0}\otimes b_{-,0}
-b_{-,0}\otimes b_{+,0}
\bigr]
+
q^2R_{\rm ns}.
\end{equation}

In the separating case, write $B_{j,0}$ for the normalized bidifferential of central component $j$, $D_j^{\rm cap}$ for the transverse derivative of its capped bidifferential, and $b_{j,0}$ for its node-evaluation one-form, defined in the same way in the cap coordinate at its cap center. On a product of component $j$ with itself,
\begin{equation}\label{eq:yamada-sep-same}
B_q=B_{j,0}+qD_j^{\rm cap}+q^2R_j.
\end{equation}
On a cross-component product,
\begin{equation}\label{eq:yamada-sep-cross}
B_q=-q\,b_{1,0}\otimes b_{2,0}+q^2R_{12},
\end{equation}
with the factors interchanged when the arguments are interchanged.

Every displayed remainder extends jointly holomorphically to $q=0$ in fixed fiber coordinates and the spectator moduli. Derivatives of every fixed order in the fiber coordinates, $s$, and $q$ are uniformly bounded on smaller compact products. In particular, $B_q$ is $O(1)$ on all such products of hot compact subsets and plumbing circles, and has no logarithmic term.

\end{thm}

\begin{proof}
Choose radii
\[
0<\rho_{\rm in}<\rho_{\rm out}
\]
inside the two fixed relative node charts. Outside the smaller node tube, the total space of $C_{q,s}$ is a holomorphic submersion because the node is the only singularity. Near each boundary collar retain the plumbing coordinate $z_+$ or $z_-$. Glue in a product disc over the full parameter base by the identity in that plumbing coordinate on
\[
\rho_{\rm in}<|z_\pm|<\rho_{\rm out}.
\]
The transition maps are holomorphic in the fiber coordinate and in $(q,s)$. The resulting fibers are compact, and a fixed finite atlas shows that the glued total space is a proper holomorphic submersion. Rename its transverse parameter $r$. At $r=0$, replacing the nodal neck by two cap discs gives the normalization. Cutting a nonseparating neck leaves one connected outer piece, whereas cutting a separating neck leaves two.

Keep $r$ as the outer-family parameter and introduce an independent plumbing parameter $p$. Remove smaller discs around the two cap centers and identify their annuli by
\[
z_+z_-=p.
\]
This defines $C^{\rm sew}_{p,r,s}$. When $p=r=q$, its exterior charts are the exterior charts of $C_{q,s}$, while its neck transition is the original relation $z_+z_-=q$. The identity maps on the exterior and neck agree on both overlaps. They assemble to a biholomorphism
\[
C^{\rm sew}_{q,q,s}\cong C_{q,s}.
\]
All inherited cycles lie in the exterior and the plumbing coordinates are unchanged, so the conifold-adapted marking is preserved.

On the proper smooth capped family, choose a holomorphic frame of the relative Hodge bundle and a holomorphic family of relative bidifferentials of the second kind. $A$-normalization is a finite-dimensional correction obtained by inverting the matrix of $A$-periods. This matrix is invertible on the central fiber and remains invertible after shrinking. Hence $B^{\rm cap}(r,s)$, its node evaluations, and all their fixed-coordinate derivatives depend jointly holomorphically on $(r,s)$. Cauchy estimates on nested cap discs and compact parameter subsets make the node-germ bounds uniform.

Apply the convergent sewing construction of
\cite[Theorems~2, 4, and~6]{Yam80}, in the fixed-kernel jump-problem form of \cite[Section~3]{HN20}, with $p$ as the plumbing parameter and $(r,s)$ as parameters; write $b_\pm(r,s)$ for the node-evaluation one-forms of $B^{\rm cap}(r,s)$, defined as above with $B^{\rm cap}(r,s)$ in place of $B_0$. In the nonseparating case the construction gives
\[
\begin{aligned}
B^{\rm sew}
={}&
B^{\rm cap}(r,s)\\
&-p\bigl[
b_+(r,s)\otimes b_-(r,s)
+b_-(r,s)\otimes b_+(r,s)
\bigr]
+p^2\widetilde R_{\rm ns},
\end{aligned}
\]
where the remainder is jointly holomorphic and uniformly bounded. In the separating case it gives a same-component remainder divisible by $p^2$ relative to the capped component kernels and the cross-component expansion
\[
B^{\rm sew}
=
-p\,b_1(r,s)\otimes b_2(r,s)
+p^2\widetilde R_{12}.
\]
The convergent construction applies on fixed plumbing circles and to all finite derivatives, with normal convergence on a parameter neighborhood.

Restrict these identities to $p=r=q$ and use the exact biholomorphism above. Holomorphy gives
\[
B^{\rm cap}(q,s)
=
B_0(s)+qD^{\rm cap}(s)+q^2\widetilde R,
\]
and
\[
b_\pm(q,s)=b_{\pm,0}(s)+O(q),
\]
with analogous component formulas. Substitution gives
\eqref{eq:yamada-nonsep}--\eqref{eq:yamada-sep-cross}. In general $D_j^{\rm cap}$ is a genuine order-$q$ term, so the transverse variation of the capped outer family enters the expansion at order $q$.

A holomorphic function of $(p,r,s)$ remains holomorphic after diagonal restriction, and multivariable Cauchy estimates remain uniform on smaller compact subsets. Normal convergence gives only nonnegative integral powers of $q$. The $dz_+/z_+$ Laurent mode is excluded by the core $A$-period in the nonseparating case and by the residue theorem on each separating component. Thus no logarithm occurs, and the plumbing-circle estimate is $O(1)$.

\end{proof}

\begin{rem}\label{rem:weaker-sewing}
The weaker consequences
\[
B_q-B_0=O(q)
\]
in the nonseparating case and
\[
B_q-B_{j,0}=O(q)
\]
on separating same-component products hold for an arbitrary transverse family. A pure-sewing $O(q^2)$ same-component statement is valid after comparison with $B^{\rm cap}(q,s)$; against $B_{j,0}$, the transverse derivative $D_j^{\rm cap}$ can contribute at order $q$.
\end{rem}

We now place the toric one-node family in this setting. Since
$\lambda_q(0,s_0)\neq0$, we replace the original transverse coordinate by
$\lambda(q,s)$ and, from now on, denote this normal-form coordinate again by
$q$. The plumbing relation is then $z_+z_-=q$, and on the quadratic cover
\[
q=\lambda=\eps^2,\qquad
t=q\,c_{\rm neck}(q,s),
\]
with $c_{\rm neck}$ holomorphic and bounded away from zero. In particular
$O(q)=O(t)=O(\eps^2)$ and $O(\eps)=O(|t|^{1/2})$, uniformly.

\begin{cor}[The toric one-node family]\label{cor:toric-sewing}
In the normal-form coordinate $q$, the toric one-node family of
Propositions~\ref{thm:generic-framing}, \ref{thm:normal-form}, and
\ref{prop:regime-atlas} satisfies the hypotheses of
Theorem~\ref{thm:actual-yamada}. The expansions
\eqref{eq:yamada-nonsep}--\eqref{eq:yamada-sep-cross} hold with estimates
locally uniform in $s$, the expansions are holomorphic on a full $t$-disc,
and every $O(q)$ error is $O(t)$, equivalently $O(\eps^2)$.
\end{cor}

\begin{proof}
For the marking hypothesis, complete the marking data of
Proposition~\ref{prop:regime-atlas} ($\gamma$ together with its completed
dual in the nonseparating case, the componentwise limiting $A$-markings in
the separating case) to a full conifold-adapted marking by choosing the
remaining symplectic basis with representatives disjoint from the node
charts; such representatives lie in the exterior, so the marking is locally
constant there. Propositions~\ref{thm:generic-framing} and
\ref{thm:normal-form} give the holomorphic toric family, unique moving
node, and fixed projection-compatible charts, and
Proposition~\ref{prop:regime-atlas} gives the conifold-adapted marking,
plumbing circles, $q=\eps^2$, and $t=q\,c_{\rm neck}(q,s)$. The holomorphic
inverse-function theorem gives $q=q(t,s)$ on a full $t$-disc. Substitution
preserves the uniform bounds and converts $O(q)$ into $O(t)=O(\eps^2)$.
\end{proof}

The scaled estimates below compare $B_q$ with the neck model. Regard the
Gaussian spectral curve $G$ as the rational bubble of the degeneration, with
uniformizing coordinate $\zeta$, and write
\[
B_G(\zeta_1,\zeta_2)
=
\frac{d\zeta_1\,d\zeta_2}{(\zeta_1-\zeta_2)^2}
\]
for its normalized bidifferential. Throughout, $K_{\rm scale}$ denotes a
compact subset of $G$ disjoint from $0$ and $\infty$, and $K_{\rm hot}$ a
fixed hot compact subset.

\begin{thm}[Uniform estimates for the degenerating bidifferential]\label{thm:bergman-package}
Work in the toric family, conifold-adapted marking, plumbing coordinates, and exact period of Corollary~\ref{cor:toric-sewing}. Locally uniformly in $s$, the conifold-adapted $A$-normalized bidifferential $B_q$ has the following two properties.

\begin{enumerate}[label=\textup{(\arabic*)}]
\item \emph{Fixed-scale estimates.} The expansions of
Theorem~\ref{thm:actual-yamada} hold on every fixed product of hot compact
subsets, including every product containing either fixed plumbing circle:
after subtraction of the common diagonal principal part when necessary,
$B_q$ is $O(1)$, its difference from the central componentwise normalized
bidifferential is $O(q)$ on same-component products, and separating
cross-component terms are $O(q)$. Derivatives of every fixed order of the
holomorphic remainders, in the fiber coordinates, $s$, and $q$, are
uniformly bounded on smaller compact subsets, and the expansion is
holomorphic in the single-valued product $q$ with no logarithm.
\item \emph{Scaled estimates.} For either square root $\eps$ of $q$,
\begin{equation}\label{eq:scaled-B}
(\Phi_\eps\times\Phi_\eps)^*B_q
=
B_G+O(\eps)
\end{equation}
on $K_{\rm scale}\times K_{\rm scale}$ after diagonal subtraction, and
\begin{equation}\label{eq:mixed-B}
(\id\times\Phi_\eps)^*B_q=O(\eps)
\end{equation}
on $K_{\rm hot}\times K_{\rm scale}$, together with the symmetric mixed estimate. These estimates hold with point derivatives of every fixed order on smaller compact subsets. The $d\zeta/\zeta$ Laurent mode vanishes in either variable.
\end{enumerate}

Both estimates hold on the full punctured $t$-disc. The capped outer
components may vary transversely with $q$ and $s$.
\end{thm}

\begin{proof}
Part (1) is Corollary~\ref{cor:toric-sewing}.

\emph{Step 1: the seam jump operator.} For the scaled estimates, retain the capped outer components constructed in the proof of Theorem~\ref{thm:actual-yamada}. On the cover $q=\eps^2$, insert the fixed rational bubble and use
\[
z_+\zeta_\infty=\eps,\qquad z_-\zeta=\eps,\qquad
\zeta_\infty=\zeta^{-1}.
\]
Eliminating $\zeta$ recovers $z_+z_-=q$, and the capped outer components retain their transverse variation. The convergent sewing construction of
\cite[Theorems~2 and~6]{Yam80}, in the fixed-kernel jump-problem form of
\cite[Lemma~3.2 and Theorem~3.3]{HN20}, adapted to the present moving
family, gives the operator estimate below.

Index the attachment annuli by oriented half-edges: each seam has two
sides, the coordinate on side $e$ is $z_e$, and $-e$ denotes the opposite
side. Choose attachment-coordinate discs $|z_e|<\rho_{\rm att}$ uniformly in the
capped outer family. Choose a compact neighborhood $K^+_{\rm scale}$ of
$K_{\rm scale}$ inside $G\setminus\{0,\infty\}$, and then fix a
radius
\[
0<2\rho<\rho_{\rm att}
\]
so small that the closed discs $|z_e|\leq2\rho$ are pairwise disjoint and the
two bubble discs $|\zeta|\leq2\rho$ and $|\zeta_\infty|\leq2\rho$ are
disjoint from $K^+_{\rm scale}$; after shrinking $|q|$ and the spectator
neighborhood, these choices hold simultaneously throughout the family.
Cauchy estimates are taken on the $2\rho$-discs, and every seam one-form
germ is measured on the $\rho$-discs. Thus, for
\[
j_e(z_e)=\sum_{\ell\geq0}c_{\ell,e}z_e^\ell\,dz_e,
\qquad
\|j\|_{\rho}
=
\max_e\sum_{\ell\geq0}|c_{\ell,e}|\rho^{\ell},
\]
let $\mathcal J_\rho$ be the corresponding Banach space of seam one-form germs.
Denote the capped outer components by $X_\mu(q,s)$, and adjoin the bubble
to the component index set by putting $X_G=G$; from here on, the index $\mu$
ranges over the capped outer components together with $G$. On each outer
$X_\mu(q,s)$, let $B_\mu(q,s)$ be the restriction to $X_\mu(q,s)$ of the
componentwise $A$-normalized bidifferential $B^{\rm cap}$; on the bubble,
take $B_G$. Choose
a base section $o_\mu$ away from the attachment discs and, on a simply
connected neighborhood of an attachment point in the second variable, set
\[
K_{\mu;q,s}(x,w)
=
\frac{1}{2\pi i}\int_{o_\mu}^{w}B_\mu(q,s)(x,\cdot).
\]
This is a differential in $x$ with vanishing inherited $A$-periods; besides
its Cauchy pole at $x=w$, its only auxiliary pole is at $x=o_\mu$. These
data, with the integration path chosen to vary continuously in $(q,s)$,
depend holomorphically on $(q,s)$, with bounds locally uniform in $s$; for $\mu=G$ the kernel $K_{G;q,s}$ is explicit and
independent of $(q,s)$. For a half-edge $e$, let $\mu(e)$ be its
incident component, let $E_\mu$ be the set of half-edges incident to the
component indexed by $\mu$, and let $\iota_e$ be the plumbing inversion from the $e$-side to the
opposite side,
\[
\iota_e\colon z_e\longmapsto z_{-e}=\frac{\eps}{z_e}.
\]
The regular part of $K_{\mu;q,s}$ at two attachment half-edges
$e,e'\in E_\mu$ has the expansion
\[
K^{\rm reg}_{e,e'}(z_e,w_{e'};q,s)
=
\sum_{m,n\geq0}
k_{m,n;e,e'}(q,s)z_e^mw_{e'}^n\,dz_e.
\]
The capped outer family is holomorphic in $(q,s)$. Cauchy estimates on the $2\rho$-discs therefore bound these coefficients locally uniformly in $s$. Plumbing inversion gives
\[
z_{-e'}=\frac{\eps}{w_{e'}},
\qquad
\iota_{e'}^*j_{-e'}
=
-\sum_{\ell\geq0}
c_{\ell,-e'}\eps^{\ell+1}w_{e'}^{-\ell-2}\,dw_{e'}.
\]
Residue extraction against the regular kernel consequently gives
\[
(T_{\eps,q,s}j)_e
=
-2\pi i
\sum_{\substack{e'\in E_{\mu(e)}\\ m,\ell\geq0}}
k_{m,\ell+1;e,e'}(q,s)c_{\ell,-e'}
\eps^{\ell+1}z_e^m\,dz_e.
\]
Every term contains a factor $\eps$; we call $T_{\eps,q,s}$ the jump
operator. Shrink $|q|$ so that $|\eps|\leq\rho^2$. Cauchy estimates on
the $2\rho$-discs then give one constant $C_J$, uniform for $s$ in the
chosen compact set, such that
\[
T_{\eps,q,s}\colon
\mathcal J_{\rho}\longrightarrow\mathcal J_{\rho},
\qquad
\|T_{\eps,q,s}j\|_{\rho}
\leq C_J|\eps|\,\|j\|_{\rho}.
\]
The space $\mathcal J_\rho$ does not vary with $(\eps,q,s)$. After
shrinking so that $C_J|\eps|<1/2$, we may therefore invert
$I-T_{\eps,q,s}$ by the Neumann series
\[
(I-T_{\eps,q,s})^{-1}
=
\sum_{k\geq0}T_{\eps,q,s}^k,
\qquad
\|(I-T_{\eps,q,s})^{-1}\|\leq2.
\]
The series converges normally on the $\rho$-discs, jointly in the moving
point, $\eps$, and $s$.

\emph{Step 2: construction and characterization.} Begin with the initial differential $S_G^a=B_G(a,\cdot)$ on the bubble and $S_\mu^a=0$ on every outer component. Let $j_e^{(0)}$ be its germ on the side indexed by $e$; since the closed bubble discs avoid $K^+_{\rm scale}$, these initial germs lie in $\mathcal J_\rho$ with norms uniformly bounded for $a\in K_{\rm scale}$. Write $\Gamma_e=\{|w_e|=\rho\}$, positively oriented in the coordinate $w_e$. Define $j^{(k)}=T_{\eps,q,s}j^{(k-1)}$, and define the $k$-th correction and the full correction of the component indexed by $\mu$ by
\[
C_\mu^{(k)}(x)
=
\sum_{e\in E_\mu}
\int_{\Gamma_e}
K_{\mu;q,s}(x,w_e)\,
\iota_e^*j_{-e}^{(k-1)}(w_e),
\qquad
C_\mu=\sum_{k\geq1}C_\mu^{(k)}.
\]
The auxiliary base-point pole of $K_{\mu;q,s}$ enters with coefficient
proportional to $\int_{\Gamma_e}\iota_e^*j_{-e}^{(k-1)}$; substituting
$z_{-e}=\eps/w_e$ turns this into an integral of the holomorphic germ
$j_{-e}^{(k-1)}$ over the circle $|z_{-e}|=|\eps|/\rho$, which vanishes. Hence every $C_\mu^{(k)}$ is holomorphic on its capped component away from the seams.

On the $e$-side of a seam, for points in the retained annulus outside $\Gamma_e$, the Cauchy integral formula on the seam circle and the definition of the regular part give
\[
\begin{aligned}
C_{\mu(e)}^{(k)}
&=\iota_e^*j_{-e}^{(k-1)}+j_e^{(k)},\\
C_{\mu(e)}^{(k)}-\iota_e^*C_{\mu(-e)}^{(k)}
&=
\bigl(j_e^{(k)}-\iota_e^*j_{-e}^{(k)}\bigr)
-
\bigl(j_e^{(k-1)}-\iota_e^*j_{-e}^{(k-1)}\bigr).
\end{aligned}
\]
Summing in $k$ telescopes:
\[
C_{\mu(e)}-\iota_e^*C_{\mu(-e)}
=
-\bigl(j_e^{(0)}-\iota_e^*j_{-e}^{(0)}\bigr).
\]
Thus the componentwise forms
$\Omega_{\eps,a}|_{X_\mu}=S_\mu^a+C_\mu$ have exactly zero seam mismatch and glue to a global meromorphic differential.

All correction integrals are supported on seams, and every seam lies a fixed positive distance from $K^+_{\rm scale}$, hence from the moving point $a\in K_{\rm scale}$. Consequently $C_G$ is holomorphic at the moving point, and $\Omega_{\eps,a}$ retains exactly the double-pole principal part $B_G(a,\cdot)$ there, with zero residue. It has no other pole.

Each component Cauchy kernel has zero period in its first variable over every inherited component $A$-cycle, so every correction has zero inherited $A$-period. Near a seam, the Cauchy integral formula expresses each correction as a sum of holomorphic germs. Its period around the seam circle therefore vanishes, and the initial differential contributes no seam period. In the nonseparating case, the inherited component $A$-cycles together with the neck core span the $A$-polarization of the conifold-adapted marking. In the separating case, the inherited componentwise $A$-cycles already span it. Hence $\Omega_{\eps,a}$ is normalized in precisely the $A$-polarization of the conifold-adapted marking of Proposition~\ref{prop:regime-atlas}.

Two global meromorphic differentials with this principal part and these vanishing $A$-periods differ by a holomorphic differential with every $A$-period zero. Expansion in an $A$-normalized basis forces that difference to vanish. Since $(\Phi_\eps\times\id)^*B_q(a,\cdot)$ has the same principal part and the same vanishing adapted $A$-periods, uniqueness gives
\[
\Omega_{\eps,a}
=
(\Phi_\eps\times\id)^*B_q(a,\cdot).
\]

\emph{Step 3: estimates and equivariance.} The constant seam term of a component Cauchy kernel integrates to zero against a holomorphic seam germ. The remaining residue terms contain at least one factor $\eps$, and the Neumann-series estimate therefore gives $C_G=O(\eps)$ on fixed scaled-neck compact subsets after diagonal subtraction. A correction reaching a fixed hot compact subset from the bubble acquires at least one factor of $\eps$ from plumbing inversion and is $O(\eps)$. Pulling the second variable back by $\Phi_\eps$, this proves \eqref{eq:scaled-B} and \eqref{eq:mixed-B}; symmetry gives the other mixed order. Cauchy estimates give the same bounds for point derivatives of every fixed order on smaller compact subsets. All constants come from the component bidifferentials of the capped outer family and from $B_G$, and the bounds above keep them locally uniform in $s$.

The only possible scale-independent Laurent mode is a multiple of $d\zeta/\zeta$. In the nonseparating case its coefficient is the period over the core $A$-cycle and vanishes by normalization. In the separating case it would have a nonzero residue at a cap center, and the residue theorem forces its coefficient to vanish. Hence the coefficient of $d\zeta/\zeta$ vanishes in either variable.

Changing $(\eps,\zeta)$ to $(-\eps,-\zeta)$, and with it the moving point $a$ to $-a$, preserves $u=\eps z$, $v=\eps w$, the seams, and the family. It exchanges the labels of the cold points and leaves the normalized solution invariant. The exact-period conversions follow from Proposition~\ref{prop:regime-atlas}.
\end{proof}

The recursion kernel \eqref{eq:EO-kernel} depends on the bidifferential
through its primitive along deck-conjugate arcs. The scaled estimates
integrate to the following comparison.

\begin{cor}[Scaled comparison of Bergman primitives]\label{cor:primitive-comparison}
For $\zeta$ in a compact subset of $G\setminus\{0,\infty\}$ covered by
simply connected neighborhoods of the two ramification points $\zeta=\pm1$,
choose a path $\lambda_\zeta$ from $\zeta^{-1}$ to $\zeta$, varying
continuously within such a neighborhood and contained in one fixed compact
subset of $G\setminus\{0,\infty\}$, and put
$J_q(p,\zeta)=\int_{\Phi_\eps(\lambda_\zeta)}B_q(p,\cdot)$. Then,
locally uniformly in $s$ and on the full punctured $t$-disc,
\[
\Phi_\eps^*J_q(a,\zeta)
=
\int_{\zeta^{-1}}^{\zeta}B_G(a,\cdot)+O(\eps)
\]
for $a$ in any fixed compact subset of $G\setminus\{0,\infty\}$ disjoint
from every $\lambda_\zeta$, while $J_q(p,\zeta)=O(\eps)$ for $p$ in a
fixed hot compact subset. The same bounds hold for point derivatives of
every fixed order on smaller compact subsets. The primitive is unchanged under
homotopies of $\lambda_\zeta$ within the chosen neighborhood and when a
path picks up the core cycle $\gamma$; it is odd under
$\zeta\mapsto\zeta^{-1}$ and vanishes at $\zeta=\pm1$.
\end{cor}

\begin{proof}
Integrate \eqref{eq:scaled-B} and \eqref{eq:mixed-B} in the second
variable along $\Phi_\eps(\lambda_\zeta)$; the paths have uniformly
bounded length in one fixed compact set, and the moving point stays a fixed
positive distance from them, so the estimates and their derivative bounds
pass through the integral. Direct evaluation gives
\[
\int_{\zeta^{-1}}^{\zeta}B_G(a,\cdot)
=
da\left(\frac{1}{a-\zeta}-\frac{1}{a-\zeta^{-1}}\right),
\]
which is odd under $\zeta\mapsto\zeta^{-1}$ and vanishes at
$\zeta=\pm1$. Two homotopic paths give the same integral because $B_q$ is
holomorphic in the integration variable away from its double pole and has
zero residue there. If a permissible deformation adds the core cycle
$\gamma$, the difference is the $\gamma$-period of $B_q$ in its second
variable, which vanishes by $A$-normalization in the nonseparating case and
by the residue theorem on a separating component. Oddness and the exact
vanishing at $\zeta=\pm1$ follow by reversing the path. Equivariance under
$(\eps,\zeta,a)\mapsto(-\eps,-\zeta,-a)$ carries the estimate to the
full punctured $t$-disc.
\end{proof}

From Section~\ref{sec:boundedness} on, we write $B_t=B_{q(t,s),s}$ for the same bidifferential, regarded as a function of the exact period $t$.

\subsection{Uniform boundedness of stable correlators}\label{sec:boundedness}

We prove that every stable correlator is bounded on hot compact subsets and on
the two plumbing circles, by an induction that peels one nested contour per
recursion step and estimates the two cold residues as a pair on an annular
boundary. Throughout, we call
\[
k=2g-2+n\geq1
\]
the \emph{complexity} of the stable type $(g,n)$. The induction is stated for
an abstract family; for its data we display the spectator dependence and write
$B_{t,s}$, so that in the toric application $B_{t,s}=B_t$ in the notation of
Section~\ref{sec:bergman}.

Consider a holomorphic family
\[
(C_{t,s},x_{t,s},\omega_{0,1;t,s},B_{t,s})
\]
of regular spectral curves for $0<|t|<\delta$, with $s$ in a compact subset of
a parameter neighborhood, whose ramification points are all simple and consist
of a fixed finite collection of hot ramification sections together with
exactly two cold points
\[
a_+(t,s),\qquad a_-(t,s)
\]
in a one-neck plumbing annulus. On the cover
\[
t=\eps^2c_{\rm neck}(\eps^2,s),
\]
use the plumbing coordinates of Proposition~\ref{prop:regime-atlas},
\[
z_+z_-=\eps^2,
\]
and the deck involution $\sigma$ interchanging the two coordinates. For a
plumbing radius $r>0$ write
\[
C_{\rm out}(r)=\{|z_+|=r\},
\qquad
C_{\rm in}(r)=\{|z_-|=r\},
\]
both counterclockwise in their plumbing coordinates, so the boundary of the
annulus between them is $C_{\rm out}(r)+C_{\rm in}(r)$. On the cold plumbing
annulus put
\[
u=\frac{z_++z_-}{2},
\qquad
v=\frac{z_+-z_-}{2},
\]
inverting the plumbing coordinates $z_\pm=u\pm v$, and write
\[
D_{\rm odd}
=
\omega_{0,1;t,s}-\sigma^*\omega_{0,1;t,s}
\]
for the deck-odd part of the initial one-form. The cross-sheet primitive for the deck involution $\sigma$, in the sense of
\eqref{eq:EO-kernel}, and the cold kernel are
\[
J(p,q)
=
\frac12\int_{\sigma(q)}^qB_{t,s}(p,\cdot),
\qquad
K_{\rm cold}(p,q)=\frac{J(p,q)}{D_{\rm odd}(q)}.
\]

Fix an integer $D\geq1$ and radii
\[
R_0>R_1>\cdots>R_D>0
\]
inside the fixed plumbing charts. Let $H_{\rm base}$ be a finite union of
hot compact subsets for the external arguments, disjoint from the ramification
points, punctures, cold plumbing annulus, and unstable bidifferential
diagonals. Around each hot ramification section $a$, choose a coordinate
$\xi_a$ with $\sigma_a(\xi_a)=-\xi_a$ and nested fixed circles
\[
\Gamma_{a,j}=\{|\xi_a|=r_{a,j}\},
\qquad
r_{a,0}>r_{a,1}>\cdots>r_{a,D}>0.
\]
For distinct hot sections, the outermost closed coordinate discs are pairwise
disjoint. Each outermost disc is disjoint from $H_{\rm base}$ and from the
closed cold plumbing annulus, and for each fixed hot section the discs bounded
by the inner circles are nested inside the outermost disc. Put
\[
H_j
=
H_{\rm base}\cup
\bigcup_a\bigcup_{i=0}^{j}\Gamma_{a,i},
\qquad
E_{\ell,j}
=
H_j\cup
\bigcup_{i=0}^{\ell}
\bigl(C_{\rm out}(R_i)\cup C_{\rm in}(R_i)\bigr).
\]
The circles are chosen so that each closed coordinate disc bounded by one of
the $\Gamma_{a,j}$ contains no ramification point other than its center and no
puncture, the local odd denominator
$\omega_{0,1;t,s}-\sigma_a^*\omega_{0,1;t,s}$ has no zero on it except at
$q=a$, and, for every hot section $a$ and all $\ell,j\geq0$ with $\ell+j<D$,
the closed disc bounded by $\Gamma_{a,j+1}$ is disjoint from $E_{\ell,j}$.

\begin{thm}[Nested-collar induction]\label{thm:nested-collar}
With the family and contour data fixed above, assume that on the entire cold
plumbing annulus
\[
D_{\rm odd}
=
2M(u,\eps^2,s)v\,du,
\]
where $M$ is holomorphic and nowhere zero after shrinking, and that the
spectral projection $x_{t,s}$ is compatible with the plumbing in the normal
form of Proposition~\ref{thm:normal-form}: on the annulus it is a function
$\chi(u,\eps^2,s)$ of $u$ alone with $\chi_u$ nowhere zero, its only
ramification points there are the cold points $a_\pm$ at $(u,v)=(\pm\eps,0)$,
and $\sigma$ is its deck involution. Assume moreover the following bounds,
locally uniformly in $s$ as $t\to0$, each measured in the fixed coordinate of
its region, $\xi_a$ near a hot ramification section and $z_\pm$ on the
plumbing circles:

\begin{enumerate}[label=\textup{(\arabic*)}]
\item $B_{t,s}=O(1)$ on every separated compact product, its factors a
positive distance apart uniformly in the parameters, made from $H_D$ and the
finitely many plumbing circles.
\item For $\ell,j\geq0$ with $\ell+j<D$, $p\in E_{\ell,j}$, and
$q\in\Gamma_{a,j+1}$, every local recursion kernel $K_a(p,q)$ of
\eqref{eq:EO-kernel}, formed with $B_{t,s}$ and $\omega_{0,1;t,s}$, is $O(1)$.
\item For $\ell,j\geq0$ with $\ell+j<D$, $p\in E_{\ell,j}$, and
\[
q\in C_{\rm out}(R_{\ell+1})\cup C_{\rm in}(R_{\ell+1}),
\]
the cross-sheet primitive $J(p,q)$ is single-valued in $q$, its analytic
continuation around either circle returning to itself, and is $O(1)$, while
the coefficient of $D_{\rm odd}(q)$ with respect to $dz_+$ on
$C_{\rm out}(R_{\ell+1})$, and with respect to $dz_-$ on
$C_{\rm in}(R_{\ell+1})$, is bounded away from zero.
\end{enumerate}

Then, for every $(g,n)$ with $n\geq1$ and complexity $k\geq1$ and every
$\ell,j\geq0$ satisfying $k+\ell+j\leq D$, the correlator $\omega_{g,n;t,s}$
is $O(1)$ on every compact configuration made from $E_{\ell,j}$ which stays
off the unstable bidifferential diagonals.

In particular, for a fixed stable type with $n\geq1$ and $2g-2+n=D$, the
correlator $\omega_{g,n;t,s}$ is $O(1)$ for every assignment of its arguments
to a hot compact subset, $C_{\rm out}(R_0)$, or $C_{\rm in}(R_0)$. The
constants depend on the stable type $(g,n)$. Since the correlators and
plumbing circles are invariant under $\eps\mapsto-\eps$, the conclusion holds
on the full punctured $t$-disc.
\end{thm}

\begin{proof}
The projection-compatibility hypothesis makes $a_\pm$ fixed points of
$\sigma$, and $\sigma$ agrees near $a_\pm$ with the local involutions
$\sigma_{a_\pm}$. We induct on the complexity $k$, proving at each stage the assertion for every pair $(\ell,j)$ with
\[
k+\ell+j\leq D.
\]

\emph{Step 1: the base case.} At complexity zero there is nothing beyond the initialization:
$\omega_{0,2}=B_{t,s}$, which assumption (1) bounds on every separated configuration.

Fix $k\geq1$ and assume the assertion for smaller complexities at every permitted pair of levels $(\ell,j)$. Write, in the notation of \eqref{eq:EO-recursion},
\[
\omega_{g,n}(p_0,\mathbf p)
=
\sum_{a\in\Ram(x_{t,s})}
\operatorname*{Res}_{q=a}
K_a(p_0,q)\,
\mathcal S_{g,n}(q,\sigma_a(q);\mathbf p),
\]
where $\mathcal S_{g,n}$ is the recursion source of \eqref{eq:EO-recursion}. Every stable factor has complexity at most $k-1$, and every unstable factor is $B_{t,s}$.

\emph{Step 2: hot residues.} Suppose first that $a$ is hot and the external arguments lie in
$E_{\ell,j}$. Since $k\geq1$, the next contour
$\Gamma_{a,j+1}$ exists. We compute the residue on that contour. Every external hot argument is either outside the corresponding ramification disc or lies on an earlier circle of strictly larger radius. Thus neither $q$ nor
$\sigma_a(q)$ can meet an external argument, and no factor
$B_{t,s}(q,p_i)$ or $B_{t,s}(\sigma_a(q),p_i)$ reaches its diagonal.
The source is evaluated on $E_{\ell,j+1}$ and has complexity at most
$k-1$, while
\[
(k-1)+\ell+(j+1)=k+\ell+j\leq D.
\]
The induction hypothesis, assumptions (1) and (2), and the fixed contour length therefore bound every hot residue by $O(1)$.

\emph{Step 3: the cold pair.} We treat the two cold residues together. Suppose the external arguments lie in $E_{\ell,j}$, where $k+\ell+j\leq D$. Since $k\geq1$, $\ell<D$. Let $A_{\ell+1}$ be the closed plumbing annulus bounded by
\[
C_{\rm out}(R_{\ell+1})+C_{\rm in}(R_{\ell+1}).
\]
Both cold ramification points lie in its interior for sufficiently small $t$. Let $\gamma$ denote the class of either plumbing circle in $A_{\ell+1}$, oriented counterclockwise in $z_+$; in the toric application it is the vanishing cycle.

The full cold integrand
\[
K_{\rm cold}(p_0,q)
\mathcal S_{g,n}(q,\sigma(q);\mathbf p)
\]
is a single-valued meromorphic one-form in $q$ on $A_{\ell+1}$; we verify
this now. Assumption (3) gives single-valuedness of $J(p_0,\cdot)$ only along
the plumbing circles, and we upgrade it to all of $A_{\ell+1}$. Put
$\beta=B_{t,s}(p_0,\cdot)$ and let
$c_q$ be a continuously varying relative integration chain from
$\sigma(q)$ to $q$. Since $p_0$ lies outside $A_{\ell+1}$, by the case analysis of the diagonal
inventory below, the form $\beta$ is holomorphic there. When $q$ makes one positive circuit of
$\gamma$, the point $\sigma(q)$, whose plumbing coordinate is $\eps^2/z_+$, makes the opposite circuit, so
\[
[c_q]_{\rm final}-[c_q]_{\rm initial}
=
\gamma-(-\gamma)
=
2\gamma.
\]
The prefactor $\frac12$ therefore makes the monodromy of $J$ equal to one
$\gamma$-period:
\[
\Delta_\gamma J
=
\frac12\int_{2\gamma}\beta
=
\int_\gamma\beta.
\]
The two plumbing circles in assumption~\textup{(3)} are homotopic to $\gamma$
in $A_{\ell+1}$, and the assumed single-valuedness of $J$ on them forces this
$\gamma$-period to vanish. Since every closed loop in the annulus is homologous
to a multiple of $\gamma$, $J$ is single-valued on the whole annulus. All remaining factors of the integrand are defined intrinsically in $q$,
hence single-valued on $A_{\ell+1}$. Stable correlators have poles only at
ramification points in each argument \cite{EO07}. The remaining possible
poles arise from diagonal factors
\[
B_{t,s}(q,p_i),\qquad
B_{t,s}(\sigma(q),p_i),\qquad
B_{t,s}(q,\sigma(q))
\]
in an unstable source; the last has poles in $A_{\ell+1}$ only at the cold
points, where $q=\sigma(q)$.

It remains to locate the zeros of the odd denominator. In the $z_+$-coordinate $z$,
\[
v=\frac{z^2-\eps^2}{2z},
\qquad
du=\frac{z^2-\eps^2}{2z^2}\,dz,
\]
and therefore
\[
D_{\rm odd}
=
\frac{M(u,\eps^2,s)(z^2-\eps^2)^2}{2z^3}\,dz.
\]
The annulus does not contain $z=0$, and $M$ is nowhere zero. Thus the only zeros of $D_{\rm odd}$ in the annulus are
\[
z=\eps,\qquad z=-\eps,
\]
the two cold points. Near the cold points $\sigma$ agrees with the local involutions
$\sigma_{a_\pm}$, so the cross-sheet primitive $J$ agrees there with the local
kernel primitives, and the monodromy computation above leaves it no period
ambiguity.

No diagonal lies in $A_{\ell+1}$. If $p_i$ is hot, it lies outside the neck.
If $p_i$ lies on an outer circle of radius $R_m$ with $m\leq\ell$, then
\[
R_m>R_{\ell+1}
\]
and it lies outside the outer boundary. If $p_i$ lies on the corresponding inner plumbing circle, its $z_+$-modulus is
\[
\frac{|\eps|^2}{R_m}
<
\frac{|\eps|^2}{R_{\ell+1}},
\]
so it lies inside the inner boundary. Hence the only poles of the cold integrand in $A_{\ell+1}$ are the two cold ramification points, including the diagonal $q=\sigma(q)$ poles there.

The residue theorem with the plumbing boundary orientation gives
\begin{align}
&\sum_{\alpha\in\{a_+,a_-\}}
\operatorname*{Res}_{q=\alpha}
K_{\rm cold}(p_0,q)\mathcal S_{g,n}
\nonumber\\
&\qquad=
\frac{1}{2\pi i}
\left(
\int_{C_{\rm out}(R_{\ell+1})}
+
\int_{C_{\rm in}(R_{\ell+1})}
\right)
K_{\rm cold}(p_0,q)\mathcal S_{g,n}.
\label{eq:cold-pair-boundary}
\end{align}
On either boundary, $q$ and $\sigma(q)$ lie on the two opposite plumbing circles at level $\ell+1$. Every stable source factor has complexity at most $k-1$, and
\[
(k-1)+(\ell+1)+j=k+\ell+j\leq D.
\]
Each such factor is therefore $O(1)$ by induction. Assumption (1) bounds the bidifferential factors, and assumption (3) bounds the cold kernel. Both plumbing circles have fixed length in their plumbing coordinates. The boundary integrals, and hence the complete pair of cold residues, are $O(1)$.

Adding the hot and cold contributions closes the induction. Taking
$k=D$ and $\ell=j=0$ gives the bound on $H_{\rm base}$ and the two outermost plumbing circles.

\emph{Step 4: uniformity and the full disc.} All constants arise from finitely many contours and source terms for the fixed stable type, so for each fixed stable type they are locally uniform in $s$. Changing $\eps$ to $-\eps$ exchanges the cold points and changes only the plumbing presentation. The correlators, plumbing circles, and paired cold residue sum are intrinsic. Since $t$ is single-valued in $\eps^2$, the estimate holds on the full punctured $t$-disc.
\end{proof}

We now verify these hypotheses in the toric family. Work in the toric
one-node family, good framing, conifold-adapted marking, exact period, and
plumbing coordinates of Propositions~\ref{thm:generic-framing},
\ref{thm:normal-form}, \ref{prop:regime-atlas} and
Corollary~\ref{cor:toric-sewing}, and set
\[
C_{t,s}=C_{q(t,s),s},
\qquad
\omega_{0,1;t,s}=\omega_{0,1;q(t,s),s},
\qquad
B_{t,s}=B_{q(t,s),s}=B_t,
\]
so that the whole spectral family is parametrized by the exact period. Fix a
hot compact subset $K_{\rm hot}$ away from every puncture, ramification point,
node preimage, and unstable bidifferential diagonal, and a sufficiently small
fixed plumbing radius $R$ below a fixed buffer radius whose closed plumbing
discs are disjoint from $K_{\rm hot}$.

\begin{thm}[Exterior bound for the toric family]\label{thm:exterior-bound}
In this toric family, for every fixed stable type with $n\geq1$ and
\[
D=2g-2+n\geq1
\]
and every assignment of its $n$ arguments to $K_{\rm hot}$, $C_{\rm out}(R)$,
or $C_{\rm in}(R)$, one has
\[
\omega_{g,n;t,s}=O(1)
\]
on compact configurations avoiding the unstable bidifferential diagonals.

The bound is locally uniform in $s$, holds on the full punctured $t$-disc, has no logarithmic term, and
its constant depends on the stable type.
\end{thm}

\begin{proof}
\emph{Step 1: the setting.} Proposition~\ref{thm:generic-framing} gives the good framing, ordinary recursion, fixed hot region, uniformly simple hot ramification points, and exactly two cold points. Propositions~\ref{thm:normal-form} and \ref{prop:regime-atlas} put the neck in the projection-compatible form required by Theorem~\ref{thm:nested-collar}: the projection is a function $\chi(u,\eps^2,s)$ of $u$ alone with $\chi_u$ nowhere zero, $\sigma(u,v)=(u,-v)$ is its deck involution, the cold points are $(u,v)=(\pm\eps,0)$, and
\[
\omega_{0,1}-\sigma^*\omega_{0,1}
=
2M(u,\eps^2,s)v\,du,
\]
with $M$ holomorphic and bounded away from zero, locally uniformly in $s$.

\emph{Step 2: the nested contours.} Fix $D$. Set $H_{\rm base}=K_{\rm hot}$ and choose, around every hot
ramification section, the finite nested circles fixed before
Theorem~\ref{thm:nested-collar}. Fix a buffer radius $R_{\rm buf}$ whose closed discs $|z_\pm|\leq R_{\rm buf}$
are disjoint from $H_D$ and from the closed outermost hot discs; the plumbing
radius $R$ fixed above, shrunk if necessary below $R_{\rm buf}$, qualifies. Choose radii
\[
R_0=R>R_1>\cdots>R_D>0,
\qquad
R_j\geq8R_{j+1}.
\]

\emph{Step 3: assumption (1).} The fixed-scale estimates of Theorem~\ref{thm:bergman-package} give an $O(1)$ bound for $B_t$ on every separated product of $H_D$ and the plumbing circles. On a product meeting the bidifferential diagonal, the bound of Theorem~\ref{thm:bergman-package} is stated after subtraction of the common principal part; the separated products of assumption (1) stay off the diagonal, and the stable configuration sets avoid the remaining unstable bidifferential diagonals. This verifies assumption (1) of Theorem~\ref{thm:nested-collar}.

\emph{Step 4: assumption (2).} Let $a(t,s)$ be a hot ramification section. Uniform simplicity of the hot
sections gives a fixed local coordinate $\xi$, centered at $a(t,s)$, in which
\[
x_f-x_f(a)=\frac{\xi^2}{2},
\qquad
\sigma_a(\xi)=-\xi.
\]
In this coordinate the local odd denominator is
\[
\omega_{0,1}-\sigma_a^*\omega_{0,1}
=
2\xi^2h_{a;t,s}(\xi)\,d\xi,
\qquad
h_{a;t,s}(0)=\partial_\xi y(a(t,s))\neq0.
\]
The uniform simplicity and nonvanishing of $dy$ supplied by
Proposition~\ref{thm:generic-framing} give a uniform positive lower bound for
$|h_{a;t,s}(0)|$. Equicontinuity of the holomorphic family $h_{a;t,s}$ then gives a single
radius, uniform in $(t,s)$, below which $h_{a;t,s}$ is zero-free on every
closed disc bounded by a circle $\Gamma_{a,j}$. In particular, the coefficient
of the odd denominator is uniformly bounded away from zero on every circle
$\Gamma_{a,j}$.
The numerator
\[
\frac12\int_{\sigma_a(q)}^qB_t(p,\cdot)
\]
is integrated along the straight segment from $-\xi$ to $\xi$, of length
$2r_{a,j+1}$, inside the closed disc bounded by $\Gamma_{a,j+1}$, which is
disjoint from every allowed $p$.
The fixed-scale estimates of Theorem~\ref{thm:bergman-package} bound $B_t$ on
all separated products of an earlier circle with the next circle. The paths stay off the bidifferential diagonal, so integrating along them
makes every hot recursion kernel $O(1)$ and verifies assumption (2).

\emph{Step 5: assumption (3).} Fix $0\leq\ell<D$ and put
\[
R_{\rm cross}=R_{\ell+1},
\qquad
R_{\rm control}=\frac{R_\ell}{2}.
\]
Then $R_{\rm control}\geq4R_{\rm cross}$: the first is the radius at which
the cross-sheet primitive is estimated, the second the radius of the annulus
on which $B_t(p,\cdot)$ is expanded. Let $A_{\rm ctrl}$ be the plumbing
annulus
\[
\frac{|\eps|^2}{R_{\rm control}}\leq|z_+|\leq R_{\rm control}.
\]
Every allowed $p$ lies outside $A_{\rm ctrl}$: a point of $H_D$ avoids the
neck, a point on an earlier outer plumbing circle has
$|z_+|\geq R_\ell>R_{\rm control}$, and a point on an earlier inner circle
has $|z_+|\leq|\eps|^2/R_\ell<|\eps|^2/R_{\rm control}$, inside the inner
boundary.

On $A_{\rm ctrl}$ use the outer coordinate $z=z_+$. Since every allowed $p$
lies outside $A_{\rm ctrl}$, the one-form $B_t(p,\cdot)$ is
holomorphic there and has a normally convergent Laurent expansion
\[
B_t(p,z)=\sum_{n\in\Z}a_nz^n\,dz.
\]
The fixed-scale estimates of
Theorem~\ref{thm:bergman-package} give the required uniform bounds on the two boundary circles of
$A_{\rm ctrl}$, each a fixed circle in a plumbing coordinate and hence a hot
compact subset in the sense of Proposition~\ref{prop:regime-atlas}. The bound
on the outer boundary circle controls the tail $n\geq0$, while the bound on
the inner one, in $z_-=\eps^2/z$, controls the tail $n\leq-2$. Moreover,
\[
2\pi i\,a_{-1}
=
\int_\gamma B_t(p,\cdot)
=
0,
\]
by the $A$-normalization of the conifold-adapted marking in the nonseparating case and by the residue
theorem in the separating case. Termwise integration of the two geometric tails, along any path from
$\sigma(q)$ to $q$ inside $A_{\rm ctrl}$, therefore gives
\[
\int_{\sigma(q)}^qB_t(p,\cdot)=O(1)
\]
for $q$ on either radius-$R_{\rm cross}$ plumbing circle. There is no
logarithmic term, and the vanishing $\gamma$-period makes the primitive
single-valued. The boundary bounds and normal convergence make the estimate
uniform around each complete circle and locally uniform in $s$.

In the outer coordinate $z=z_+$,
\[
v=\frac{z^2-\eps^2}{2z},
\qquad
du=\frac{z^2-\eps^2}{2z^2}\,dz,
\]
so
\[
\omega_{0,1}-\sigma^*\omega_{0,1}
=
\frac{M(u,\eps^2,s)(z^2-\eps^2)^2}{2z^3}\,dz.
\]
On $|z|=R_{\rm cross}$, once
\[
|\eps|^2\leq\frac{R_{\rm cross}^2}{2},
\]
the absolute value of the coefficient is at least
\[
\frac{mR_{\rm cross}}8,
\]
where $m$ is a positive lower bound for $|M|$, locally uniform in $s$. In the inner coordinate the coefficient differs only by a sign, so the same lower bound holds. Dividing the bounded cross-sheet primitive by this fixed positive lower bound verifies assumption (3).

\emph{Step 6: conclusion and descent.} Theorem~\ref{thm:nested-collar} now applies and gives $O(1)$ for the target
correlator on every assignment to $H_{\rm base}=K_{\rm hot}$, $C_{\rm out}(R)$,
and $C_{\rm in}(R)$.

Only finitely many nested circles occur for a fixed stable type. The Bergman estimates, the hot-simplicity lower bounds, and the denominator lower bounds are locally uniform in $s$. Every plumbing assertion is formulated in the single-valued transverse
parameter of Corollary~\ref{cor:toric-sewing}, which satisfies $q=q(t,s)$ on a
full $t$-disc. Hence the $O(1)$ bound holds on the full punctured $t$-disc.
\end{proof}

\subsection{The Gaussian leading limit}\label{sec:gaussian-limit}

We record a local parity consequence of the recursion which will also be used in Section~\ref{sec:defect}.

\begin{lem}[Local parity at simple ramification]\label{lem:local-parity}
Let $\omega_{g,n}$ be a stable correlator of a regular ordinary
topological-recursion spectral curve, with $n\geq1$. Fix all arguments except
the first away from the ramification points, and write
\[
\alpha(p)=\omega_{g,n}(p,p_2,\ldots,p_n).
\]
If $\xi$ is a local coordinate at a simple ramification point $a$, with local
deck involution $\sigma(\xi)=-\xi$, then $\alpha+\sigma^*\alpha$ is holomorphic there.
Consequently, if $\varphi$ is holomorphic and invariant under $\sigma$, then
\[
\operatorname*{Res}_{\xi=0}\varphi(\xi)\alpha(\xi)=0.
\]
In particular, every stable correlator has zero residue at a ramification
point in each external variable.
\end{lem}

\begin{proof}
We use the first argument as the distinguished variable in the recursion.
The dependence on the distinguished variable enters through the cross-sheet primitive
\[
J_b(p,q)=\frac12\int_{\sigma_b(q)}^qB(p,\cdot)
\]
of \eqref{eq:EO-kernel} at each ramification point $b$. If $b\neq a$, the integration path stays away
from $\xi=0$, and the symmetrization of $J_b$ is holomorphic in $p$.

It remains to consider the residue at $a$. On the coordinate disc write
$\xi$, $\xi'$, and $\eta$ for the local coordinates of $p$, of the second
argument of $B$, and of the residue variable $q$. Decompose the
bidifferential as
\[
B
=
\frac{d\xi\,d\xi'}{(\xi-\xi')^2}
+B_{\rm hol}(\xi,\xi')\,d\xi\,d\xi',
\]
where $B_{\rm hol}$ is holomorphic. The diagonal part of $J_b$ at $b=a$ is
\[
J^{\rm diag}(p,q)
=
\frac12\left(\frac1{\xi-\eta}-\frac1{\xi+\eta}\right)d\xi
=
\frac{\eta}{\xi^2-\eta^2}\,d\xi.
\]
Thus $\sigma^*J^{\rm diag}=-J^{\rm diag}$, the sign coming from
$\sigma^*d\xi=-d\xi$. For the holomorphic remainder, put
\[
F(\xi,\eta)=\frac12\int_{-\eta}^{\eta}B_{\rm hol}(\xi,\xi')\,d\xi',
\]
with the integration path inside $|\xi'|<|\xi|$, off the diagonal. The
symmetrization satisfies
\[
F\,d\xi+\sigma^*(F\,d\xi)=\bigl(F(\xi,\eta)-F(-\xi,\eta)\bigr)d\xi;
\]
since $F$ is odd in $\eta$ and the difference vanishes at $\xi=0$, this is
divisible by $\xi\eta$ in the ring of germs holomorphic at
$(\xi,\eta)=(0,0)$. The local odd denominator vanishes to order two in
$\eta$, so after division the symmetrized kernel contribution has at most a
simple pole in $\eta$, and taking the residue in $\eta$ leaves a quantity
divisible by $\xi$, hence holomorphic in $\xi$.
The recursion source is independent of the distinguished variable, so the
dependence of each residue on $p$ lies entirely in the numerators treated
above; summing the ramification residues proves that $\alpha+\sigma^*\alpha$
is holomorphic.

Thus the principal part $\alpha_{\rm pr}$ of $\alpha$ at $\xi=0$ is
anti-invariant. Since residues are unchanged by the local involution and $\varphi$
is invariant,
\[
\operatorname*{Res}_{\xi=0}\varphi\alpha
=
\operatorname*{Res}_{\xi=0}\varphi\alpha_{\rm pr}
=
-\operatorname*{Res}_{\xi=0}\varphi\alpha_{\rm pr}
=0.
\]
Taking $\varphi=1$ proves the residue assertion in the distinguished variable, and
symmetry of the correlators gives the assertion in every external variable.
\end{proof}

The scaled limit is governed by ordinary recursion on the Gaussian curve.
Recall from Section~\ref{sec:normal-form} the curve $G$, its coordinate
$\zeta$, the circle $\gamma_G$, and the constants
$\kappa=\int_{\gamma_G}w\,dz$ and $c_{\rm per}$, and from
Section~\ref{sec:bergman} the bidifferential $B_G$. Define the \emph{unit
exact-period Gaussian data} on $G$ by
\[
\overline\omega_{0,1}^{\,G}
=
\frac{1}{c_{\rm per}\kappa}w\,dz,
\qquad
\omega_{0,2}^{G}=B_G;
\]
the overline marks this normalization, and
$c_{\rm per}\int_{\gamma_G}\overline\omega_{0,1}^{\,G}=1$, so the exact
period of these data equals $1$. Denote their ordinary topological-recursion
correlators by $\omega_{g,n}^{G}$, and near each ramification point
$\zeta=\pm1$ write $\Phi_G$ for the local primitive of
$\overline\omega_{0,1}^{\,G}$ vanishing there.

Call an argument of a correlator \emph{exterior} when it is placed on a
fixed hot compact subset, including a plumbing circle
(Proposition~\ref{prop:regime-atlas}), and \emph{scaled} when it is placed
on a fixed scaled-neck compact subset through $\Phi_\eps$. Sizes of scaled
arguments are measured in the coordinate $\zeta$ and sizes of exterior
arguments in the fixed coordinate of their region; recall
$|t|\asymp|\eps|^2$. In the displays below the scaled arguments are listed
first.

\begin{prop}[Mixed scaled--exterior recursion bounds]\label{prop:mixed-scaled}
Fix a stable type $(g,n)$ with $n\geq1$ and put $k=2g-2+n\geq1$. Assign $m$
of the $n$ arguments, $0\leq m\leq n$, to scaled positions and the remaining
$e=n-m$ to exterior positions. On compact configurations away from the
ramification points and off the unstable bidifferential diagonals,
\[
(\Phi_\eps^{\times m}\times\id^{\times e})^*\omega_{g,n;t,s}
=
\begin{cases}
O(1),&m=0,\\
O\!\left(|t|^{-k}|\eps|^{e}\right)
=O\!\left(|\eps|^{e-2k}\right),&m\geq1.
\end{cases}
\]
If every argument is scaled, then
\begin{equation}\label{eq:scaled-correlator}
t^k(\Phi_\eps^{\times n})^*\omega_{g,n;t,s}
=
\omega_{g,n}^G+O(\eps).
\end{equation}
The estimates hold with point derivatives of every fixed order on smaller compact subsets, locally uniformly in $s$ and on the full punctured $t$-disc.
\end{prop}

\begin{proof}
\emph{Step 1: the cold residues.} The case $m=0$ is Theorem~\ref{thm:exterior-bound}. For $m\geq1$, we induct on $k$ and choose a scaled argument as the distinguished recursion variable. On either fixed cold residue circle, Theorem~\ref{thm:bergman-package} and the exact Gaussian decomposition of Proposition~\ref{thm:period-split} give, since $D_{\rm odd}=\omega_{0,1}-\sigma^*\omega_{0,1}=2\omega_{\rm odd}$,
\[
\frac12\int_{\sigma(q)}^qB_t(p_0,\cdot)=O(1),
\qquad
\Phi_\eps^*D_{\rm odd}
=
t\left(\frac{2}{c_{\rm per}\kappa}w\,dz+O(\eps)\right).
\]
The coefficient of the parenthesized form with respect to $d\zeta$ is bounded away from zero on the two fixed cold residue circles, since $w\,dz$ vanishes only at $\zeta=\pm1$; hence the cold kernel is $O(t^{-1})$ there. The bidifferential $B_t$ supplies the complexity-zero base of the induction: on the allowed off-diagonal configurations, Theorem~\ref{thm:bergman-package} bounds it by $O(1)$ when both arguments are scaled \eqref{eq:scaled-B}, by $O(\eps)$ when one argument is exterior and one is scaled \eqref{eq:mixed-B}, and by $O(1)$ when both are exterior. The genus-reduction term has complexity $k-1$; the distinguished argument
being scaled, it keeps the same $e$ exterior arguments, and its two recursion
arguments lie on the cold circles and are scaled. Every factor of a splitting term contains one of the two scaled recursion arguments; the factor complexities sum to $k-1$ and their exterior counts sum to $e$. Every term of the cold source therefore falls under one of these cases, so the cold source is
$O(|t|^{-(k-1)}|\eps|^e)$, and multiplication by the kernel bound gives the cold bound.

\emph{Step 2: the hot residues.} At a hot ramification point, the mixed Bergman estimate makes the kernel
$O(\eps)$. If $m=1$, every argument of the source is exterior and the source is $O(1)$; kernel times source is then $O(\eps)$, and its ratio to the claimed bound is $O(|\eps|^{1+2k-e})$, with nonnegative exponent by the display below. If
$m\geq2$, the two hot recursion arguments contribute two further exterior arguments; for the genus-reduction term, and for a splitting term all of whose factors contain a scaled argument, induction then gives
\[
O\!\left(|t|^{-(k-1)}|\eps|^{e+2}\right).
\]
After multiplication by the kernel, dividing by the claimed bound leaves
$O(|t||\eps|^3)$. In the remaining splitting case, let the exterior-only factor have complexity $k_0$ and $e_0$ arguments, the count including its hot recursion argument. The mixed factor and kernel give
\[
O\!\left(
|\eps|\,|t|^{-(k-1-k_0)}|\eps|^{e+2-e_0}
\right).
\]
Dividing by the claimed bound leaves
$O(|t|^{1+k_0}|\eps|^{3-e_0})$. For an unstable bidifferential,
$(k_0,e_0)=(0,2)$, so this is $O(|\eps|^3)$. For a stable factor,
$k_0=2g_0-2+e_0$, and comparison with $|t|\asymp|\eps|^2$ gives the positive exponent
\[
2(1+k_0)+3-e_0=1+4g_0+e_0>0.
\]
Thus every hot contribution obeys the mixed bound. Finally,
\[
2k-e=4g-4+n+m\geq0
\]
for a stable type with $m\geq1$, since $g\geq1$, or $g=0$ and $n\geq3$; hence the exponent $e-2k$ is nonpositive.

\emph{Step 3: Gaussian convergence.} For Gaussian convergence, scaling the Gaussian initial form by $t$ scales a stable correlator of complexity $k$ by $t^{-k}$. Half of the scaled comparison of Corollary~\ref{cor:primitive-comparison}, taken for the moving point in a fixed scaled compact subset disjoint from its integration arcs, together with the exact Gaussian decomposition of Proposition~\ref{thm:period-split}, sharpens the cold estimate to
\[
t\,(\Phi_\eps\times\Phi_\eps)^*K_{\rm cold}=K_G+O(\eps)
\]
on fixed scaled compact subsets, where $K_G$ is the recursion kernel of
$(G,\overline\omega_{0,1}^{\,G},B_G)$.
Induction in \eqref{eq:EO-recursion} therefore sends every rescaled cold source to its Gaussian counterpart with error $O(\eps)$. The rescaled hot residues are $O(\eps)$: for $n=1$ every argument of their sources is exterior and
$t^kK_a=O(\eps^{2k+1})$ at each hot ramification point $a$; for $n\geq2$ the mixed estimates above leave either the factor
$|t||\eps|^3$ or
$|t|^{1+k_0}|\eps|^{3-e_0}$, both $O(\eps)$. Hence in the limit only the two cold residues contribute, and they give the Gaussian recursion, proving \eqref{eq:scaled-correlator}. Cauchy estimates give the derivative bounds. Writing $\zeta_i$ for the scaled coordinates of the arguments, equivariance under
$(\eps,\zeta_i)\mapsto(-\eps,-\zeta_i)$ gives the full-disc assertion.
\end{proof}

Recall the conifold-polarized free energy
\[
F_g^c(s,t)=
\frac{1}{2-2g}
\sum_{a\in\Ram(x_f)}
\operatorname*{Res}_{p=a}
\Phi_a(p)\,\omega_{g,1;t,s}^c(p),
\qquad
d\Phi_a=\omega_{0,1;t,s},
\]
of \eqref{eq:main-free-energy}, where $\Phi_a$ is a local primitive of the initial
form at the ramification point $a$; by Lemma~\ref{lem:local-parity}, the
residues do not depend on the choice of additive constants. The correlators
estimated in this subsection are the conifold-polarized correlators,
$\omega_{g,1;t,s}=\omega_{g,1;t,s}^c$, since the recursion runs on the
conifold-adapted $A$-normalized bidifferential $B_t$. For the unit
exact-period Gaussian data, put
\begin{equation}\label{eq:gaussian-free-energy}
F_g^G
=
\frac{1}{2-2g}
\sum_{\alpha\in\{1,-1\}}
\operatorname*{Res}_{\zeta=\alpha}
\Phi_G(\zeta)\,\omega_{g,1}^G(\zeta),
\end{equation}
with the local primitives $\Phi_G$ fixed above.

\begin{thm}[Exact-period Gaussian leading limit]\label{thm:gaussian-limit}
Fix an integer $g\geq2$. On the quadratic cover
$t=\eps^2c_{\rm neck}(\eps^2,s)$,
\begin{equation}\label{eq:gaussian-leading}
t^{2g-2}F_g^c(s,t)
=
F_g^G+O(\eps)
\end{equation}
as $\eps\to0$, locally uniformly in $s$. In particular
\[
\lim_{\eps\to0}t^{2g-2}F_g^c(s,t)=F_g^G,
\]
and the limit is independent of all spectator moduli.
\end{thm}

\begin{proof}
\emph{Step 1: the hot residues.} The ramification set consists of the fixed hot sections and the two cold points; we treat them separately. Around every hot section choose a fixed small coordinate circle. The initial form has a local primitive, chosen to vanish at the section and holomorphic in the curve point, $t$, and $s$; it is uniformly $O(1)$ on these circles. Theorem~\ref{thm:exterior-bound} gives
\[
\omega_{g,1;t,s}=O(1)
\]
there. Thus the sum of the hot free-energy residues is $O(1)$, and
\begin{equation}\label{eq:hot-free-energy}
t^{2g-2}
\sum_{a\ {\rm hot}}
\operatorname*{Res}_{p=a}
\Phi_a\,\omega_{g,1;t,s}
=
O(|t|^{2g-2})
=
O(\eps),
\end{equation}

\emph{Step 2: the odd part at the cold points.} Near either cold ramification point, decompose the initial form into its
deck-even and deck-odd parts,
\[
\omega_{0,1;t,s}
=
\omega_{\rm even}+\omega_{\rm odd},
\qquad
\omega_{\rm even/odd}
=
\frac{\omega_{0,1;t,s}\pm\sigma^*\omega_{0,1;t,s}}{2},
\]
matching Proposition~\ref{thm:normal-form}.
Choose primitives vanishing at the ramification point. The primitive of the even part is invariant under the local involution. Lemma~\ref{lem:local-parity} gives
\begin{equation}\label{eq:even-primitive-zero}
\operatorname*{Res}_{p=a}
\Phi_{\rm even}\,\omega_{g,1;t,s}(p)=0.
\end{equation}
Hence only the primitive of $\omega_{\rm odd}$ contributes.

\emph{Step 3: the scaled comparison.} We pull back by $\Phi_\eps$ and compute the two cold residues on fixed cold residue circles around $\zeta=1$ and $\zeta=-1$. Proposition~\ref{thm:period-split} and Proposition~\ref{prop:regime-atlas} give
\[
\Phi_\eps^*\omega_{\rm odd}
=
t\,\overline\omega_{0,1}^{\,G}
+
\eps^3\Delta_\omega,
\qquad
t=\eps^2c_{\rm neck}(\eps^2,s),
\]
where $\Delta_\omega$ is holomorphic and uniformly bounded on the two closed discs bounded by these circles, fixed compact subsets of $G$ away from $\zeta=0$ and $\zeta=\infty$. Normalizing the pulled-back primitives to vanish at the corresponding ramification point gives
\begin{equation}\label{eq:odd-primitive}
\Phi_\eps^*\Phi_{\rm odd}
=
t\Phi_G+\eps^3\Delta_\Phi,
\end{equation}
with $\Delta_\Phi$ uniformly bounded.

\emph{Step 4: the cold Gaussian limit.} For type $(g,1)$, the complexity is $2g-1$. Proposition~\ref{prop:mixed-scaled} gives
\begin{equation}\label{eq:scaled-one-point}
t^{2g-1}\Phi_\eps^*\omega_{g,1;t,s}
=
\omega_{g,1}^G+O(\eps)
\end{equation}
on each fixed cold residue circle. Combining
\eqref{eq:odd-primitive} and \eqref{eq:scaled-one-point},
\begin{align*}
&t^{2g-2}
(\Phi_\eps^*\Phi_{\rm odd})
(\Phi_\eps^*\omega_{g,1;t,s})\\
&\qquad=
\left(
\Phi_G+\frac{\eps^3}{t}\Delta_\Phi
\right)
\left(
\omega_{g,1}^G+O(\eps)
\right)
=
\Phi_G\omega_{g,1}^G+O(\eps),
\end{align*}
because
\[
\frac{\eps^3}{t}
=
\frac{\eps}{c_{\rm neck}(\eps^2,s)}
=
O(\eps).
\]
Taking residues at the two cold points gives
\begin{align}
&t^{2g-2}
\sum_{a\ {\rm cold}}
\operatorname*{Res}_{p=a}
\Phi_a\,\omega_{g,1;t,s}
\nonumber\\
&\qquad=
\sum_{\alpha\in\{1,-1\}}
\operatorname*{Res}_{\zeta=\alpha}
\Phi_G\omega_{g,1}^G
+
O(\eps).
\label{eq:cold-free-energy}
\end{align}
Equations \eqref{eq:hot-free-energy}, \eqref{eq:even-primitive-zero}, and
\eqref{eq:cold-free-energy}, divided by $2-2g$, prove
\eqref{eq:gaussian-leading}. The Gaussian data contain no spectator moduli, so the limit is independent of $s$.
\end{proof}

\section{Special geometry and the bounded Euler defect}\label{sec:defect}

This section proves the two exact identities whose difference is the Euler defect, and bounds the defect on the full punctured $t$-disc.

\subsection{Special geometry for the fixed-$x_f$ deformation}\label{sec:special-geometry}

Work in the toric one-node family, good framing, conifold-adapted marking, and
exact period of Propositions~\ref{thm:generic-framing}, \ref{thm:normal-form},
and \ref{prop:regime-atlas}, with $s$ the spectator moduli. The family
deformation is the motion of the exact period $t$ itself, evaluated at fixed
$x_f$, as opposed to a variation of filling fractions; its tangent form is the
\emph{deformation form}
\[
\Omega_t=(\partial_ty)_{x_f}\,dx_f.
\]
Its puncture residues need not be independent of $t$. At a puncture $p$ choose a parameter $z$ and write
\[
X=z^{a_p}U_X(z,t,s),\qquad
Y=z^{b_p}U_Y(z,t,s),\qquad
r_p=a_p+fb_p\neq0,
\]
with $U_X$ and $U_Y$ holomorphic units; here $r_p$ is the residue of $dx_f$
at $p$, nonzero by Proposition~\ref{thm:generic-framing}. After choosing a
holomorphic branch of $(U_XU_Y^f)^{1/r_p}$, set
\[
\xi=z(U_XU_Y^f)^{1/r_p}.
\]
Then
\[
x_f=r_p\log\xi,\qquad
y=b_p\log\xi+\log V,\qquad
V=U_Y(U_XU_Y^f)^{-b_p/r_p},
\]
where the units on the right are expressed in the $\xi$-coordinate. Fixed-$x_f$ differentiation is fixed-$\xi$ differentiation, and hence
\[
\Omega_t
=
r_p\,\partial_t\log V(\xi,t,s)\frac{d\xi}{\xi}.
\]
Thus $\Omega_t$ has at most a simple pole at each puncture, with no
higher-order polar part, and residue
\[
\rho_p
=
a_p\,\partial_t\log U_Y(0,t,s)
-b_p\,\partial_t\log U_X(0,t,s).
\]
The framing does not enter this formula: only $a_p$, $b_p$, and the units
occur. It is also coordinate independent: under $z'=c(t,s)z+O(z^2)$ the
units acquire the factors $c^{-a_p}$ and $c^{-b_p}$ at the puncture, so the
two terms change by $-a_pb_p\,\partial_t\log c$ and
$+a_pb_p\,\partial_t\log c$. The units vary holomorphically through $t=0$;
therefore every $\rho_p$ is holomorphic, and bounded locally uniformly in
$s$. Granted the holomorphy on the torus established below, the residue
theorem gives $\sum_p\rho_p=0$.

In the algebraic torus, write the family locally as
$\Ecal(x_f,y;t,s)=0$. Away from ramification,
\[
\Omega_t=-\frac{\Ecal_t}{\Ecal_y}\,dx_f,
\]
while at a ramification point smoothness gives $\Ecal_{x_f}\neq0$ and the same form is
\[
\Omega_t=\frac{\Ecal_t}{\Ecal_{x_f}}\,dy.
\]
The form $\Omega_t$ is therefore holomorphic at every point of the fiber
lying in the algebraic torus, including every ramification point.

Let a dot denote a Gauss--Manin lift of $\partial_t$. On every local logarithm branch, direct differentiation gives
\[
\dot\omega_{0,1}-\Omega_t
=
\dot y\,dx_f+y\,d\dot x_f
-\left(\dot y-\frac{dy}{dx_f}\dot x_f\right)dx_f
=
d(y\dot x_f).
\]
This identity is branchwise exact: because $y$ is logarithmic, the primitive
$y\dot x_f$ need not be globally single-valued. Radially displace the flatly
transported vanishing cycle $\gamma$ off the two cold ramification points
$\zeta=\pm1$. On the
displaced representative the fixed-$x_f$ lifts are regular and patch by
parameter-independent translations on overlaps. On each arc use the fixed-$x_f$ lift, in which $\dot x_f=0$ and the identity
reads $\dot\omega_{0,1}=\Omega_t$; since the charts differ on overlaps by
parameter-independent translations, the arcwise lifts glue, and summing over
the arcs identifies the $\Omega_t$-period with the $t$-derivative of the
$\omega_{0,1}$-period. Both periods are unchanged
under the radial displacement: for $\Omega_t$ by the holomorphy established
above, and for $\omega_{0,1}$ because the neck lies in a torus chart carrying
single-valued logarithms. The period normalization
$t=c_{\rm per}\int_\gamma\omega_{0,1}$ then gives
\[
\int_\gamma\Omega_t=\frac1{c_{\rm per}}.
\]

The normalization above yields an explicit decomposition of $\Omega_t$ into
generalized cycles; recall that a generalized cycle is a finite formal sum of
cycles and of paths joining punctures, with coefficients, integrated in the
second variable of $B_t$.

\begin{prop}[Generalized-cycle decomposition of the deformation form]\label{prop:gamma-decomposition}
There are generalized cycles $\Gamma_{\rm neck}$ and $\Gamma_{\rm hot}$,
constructed in the proof, with
\begin{equation}\label{eq:physical-tangent}
\Omega_t
=
\int_{\Gamma_{\rm neck}}B_t
+
\int_{\Gamma_{\rm hot}}B_t.
\end{equation}
The supports of $\Gamma_{\rm hot}$ may be chosen away from the neck and the
ramification points, and all coefficients are bounded locally uniformly in
$s$. The generalized cycle $\Gamma_{\rm hot}$ carries the variable third-kind
coefficients and all remaining holomorphic coefficients. Write
$\Gamma_t=\Gamma_{\rm neck}+\Gamma_{\rm hot}$.
\end{prop}

\begin{proof}
In the nonseparating case take $A_1=\gamma$; let $\nu_1,\ldots,\nu_{g_C}$ be
the holomorphic one-forms normalized by $\int_{A_i}\nu_j=\delta_{ij}$, where
$g_C$ is the fiber genus, and let $\beta_i$ be the dual cycle of $A_i$.
Choose a reference puncture $p_{\rm ref}$ and, for each puncture
$p\neq p_{\rm ref}$, a path $\ell_{p_{\rm ref},p}$ from $p_{\rm ref}$ to $p$;
let $dS_{p,p_{\rm ref}}=\int_{p_{\rm ref}}^{p}B_t$. Set
$h_i=\int_{A_i}\Omega_t$. Then
\[
\Omega_t
=
\frac1{c_{\rm per}}\nu_1
+\sum_{i=2}^{g_C}h_i\nu_i
+\sum_{p\neq p_{\rm ref}}\rho_p\,dS_{p,p_{\rm ref}}:
\]
the two sides have the same simple poles with the same residues and the same
$A$-periods, so their difference is a holomorphic form with vanishing
$A$-periods, hence zero. Equivalently,
\[
\Gamma_{\rm neck}
=
\frac{1}{2\pi i c_{\rm per}}\beta_1,
\qquad
\Gamma_{\rm hot}
=
\sum_{i=2}^{g_C}\frac{h_i}{2\pi i}\beta_i
+\sum_{p\neq p_{\rm ref}}\rho_p\,\ell_{p_{\rm ref},p}.
\]
All paths and cycles in $\Gamma_{\rm hot}$ may be chosen away from the neck
and the ramification points; the residue coefficients are bounded by the
puncture calculation, and the coefficients $h_i$ are bounded because
$\Omega_t$ is holomorphic through $t=0$ on fixed hot representatives.

In the separating case, let $C_L$ be the side of which $\gamma$, with its
fixed orientation, is the positively oriented boundary, and let $C_R$ be the
other side. The residue theorem and the period normalization give
\[
\sum_{p\in C_L}\rho_p=\frac{1}{2\pi i c_{\rm per}},
\qquad
\sum_{p\in C_R}\rho_p=-\frac{1}{2\pi i c_{\rm per}}.
\]
Each sum is nonzero, so each side carries a puncture. Choose $p_L\in C_L$ and
$p_R\in C_R$, and a path $\ell_{\rm sep}$ from $p_R$ to $p_L$ traversing the
neck along the cross-cut $\delta_{\rm neck}$ of
Proposition~\ref{prop:regime-atlas}; since $\gamma$ bounds $C_L$,
\[
\int_\gamma\int_{\ell_{\rm sep}}B_t=2\pi i.
\]
Set
\[
\Gamma_{\rm neck}
=
\frac{1}{2\pi i c_{\rm per}}\ell_{\rm sep}.
\]
After subtracting $\int_{\Gamma_{\rm neck}}B_t$ from $\Omega_t$, the adjusted
residues are
\[
\widetilde\rho_{p_L}
=
\rho_{p_L}-\frac{1}{2\pi i c_{\rm per}},
\qquad
\widetilde\rho_{p_R}
=
\rho_{p_R}+\frac{1}{2\pi i c_{\rm per}},
\]
with all other residues unchanged, and their sums vanish separately on the
two sides. Carry them by paths joining punctures within one side, and add
the componentwise dual cycles with coefficients $1/(2\pi i)$ times the
corresponding componentwise $A$-periods of
$\Omega_t-\int_{\Gamma_{\rm neck}}B_t$; this produces $\Gamma_{\rm hot}$.
Its residue coefficients are bounded by the puncture calculation, and its
cycle coefficients are bounded because $\Omega_t$ is holomorphic on fixed
hot representatives.
\end{proof}

Write $\nabla_t^{x_f}$ for the $t$-derivative taken with all external points
transported at fixed $x_f$; the subscript on a generalized-cycle integral
records the variable of integration. We suppress $(t,s)$ from the correlator
subscripts when no confusion can arise.

\begin{thm}[Exact special geometry for the family deformation]\label{thm:physical-special-geometry}
Let $U$ be a simply connected marked punctured neighborhood in the exact
period $t$, and let $\omega_{g,n}^c$ and $F_g^c$ be the conifold-polarized
correlators and free energies. For every stable type $(g,n)$ with $n\geq1$,
\[
\nabla_t^{x_f}\omega_{g,n}^c
=
\int_{\Gamma_{\rm neck}(q)}\omega_{g,n+1}^c(\cdot,q)
+
\int_{\Gamma_{\rm hot}(q)}\omega_{g,n+1}^c(\cdot,q),
\]
the external points being transported at fixed $x_f$. For every $g\geq2$,
\begin{equation}\label{eq:physical-special-geometry}
\partial_tF_g^c
=
I_g^{\rm neck}+R_g,
\qquad
I_g^{\rm neck}
:=
\int_{\Gamma_{\rm neck}}\omega_{g,1}^c,
\qquad
R_g
:=
\int_{\Gamma_{\rm hot}}\omega_{g,1}^c,
\end{equation}
and
\[
R_g=O(1)
\]
locally uniformly in $s$. In the separating case, $I_g^{\rm neck}$ is the
integral of $\omega_{g,1}^c$ along the relative third-kind path crossing the
neck; $R_g$ is a finite sum of integrals of $\omega_{g,1}^c$ over the fixed
hot paths and cycles of Proposition~\ref{prop:gamma-decomposition}, with
locally uniformly bounded coefficients.
\end{thm}

\begin{proof}
\emph{Step 1: the intrinsic deformation form.} Fix the spectator moduli. For any local lift
$\mathcal V$ of $\partial_t$ to the total family, define
$\Omega_{\mathcal V}=\mathcal V(y)\,dx_f-\mathcal V(x_f)\,dy$ on a fiber.
This form is independent of the lift. Indeed, replacing $\mathcal V$ by
$\mathcal V+\mathcal W$ for a vertical vector field changes it by
$\mathcal W(y)\,dx_f-\mathcal W(x_f)\,dy$. Where $dx_f\neq0$, both
differentials are multiples of $dx_f$, and the chain rule makes this
difference zero. Holomorphic continuation makes it zero identically.
Where $x_f$ is a coordinate, the resulting intrinsic form is
$(\partial_ty)_{x_f}\,dx_f=\Omega_t$.

The branchwise identity established above says that
$\mathcal V(y\,dx_f)-\Omega_t=d(y\mathcal V(x_f))$ on every local logarithm
branch. Its primitive need not be globally single-valued, but in a
fixed-$x_f$ lift the correction vanishes. Thus $\Omega_t$ is the intrinsic
deformation form for the moving recursion datum.

Although $x_f$ is logarithmic, $dx_f=d\log X+f\,d\log Y$ is globally
meromorphic. Its period around any Gauss--Manin-flat closed cycle avoiding the punctures
is $2\pi i$ times the integer $w(X)+f\,w(Y)$, where $w(\cdot)$ denotes the
winding number of the corresponding torus coordinate along the cycle; this
integer is locally constant in the family. At a puncture, the coordinate constructed above satisfies
$x_f=r_p\log\xi$ with the fixed nonzero integer $r_p=a_p+fb_p$. Hence fixed-$x_f$
charts have parameter-independent translations on overlaps, and their
fixed-$\xi$ trivializations extend nonsingularly across every puncture disc.
The puncture trivializations therefore contribute zero to the \v{C}ech
representative of the Kodaira--Spencer class; these charts fail only at the
simple zeros of $dx_f$.

\emph{Step 2: the Rauch variational formula.} We next compute the variation
of $B_t$ under fixed-$x_f$ transport. The normalized bidifferential
satisfies the Rauch variational formula
\begin{equation}\label{eq:physical-rauch}
\nabla_t^{x_f}B_t(p,q)
=
-\sum_{a\in\Ram(x_f)}
\operatorname*{Res}_{r=a}
\frac{\Omega_t(r)B_t(r,p)B_t(r,q)}
{dx_f(r)\,dy(r)}.
\end{equation}
The deformation convention in \cite[eq.~(5-1)]{EO07} is
$\delta_{\Omega^{\rm EO}}y|_{x_f}\,dx_f=-\Omega^{\rm EO}$, so the family
deformation has $\Omega^{\rm EO}=-\Omega_t$.
The branch-value formula \cite[eq.~(5-3)]{EO07} has coefficient exactly
$\Omega^{\rm EO}(a)/dy(a)$ in front of the Schiffer variation. Since $B_t$ is
$A$-normalized, the auxiliary polarization matrix vanishes, and
\cite[eq.~(5-4)]{EO07} has the form displayed above.

To identify the local coefficient, choose a parameter-dependent coordinate $\xi_a$ at a
simple ramification section such that
$x_f=x_f(a(t))+\xi_a^2/2$. Fixed-$x_f$ transport has singular part
$(\partial_t\xi_a)_{x_f}=-\partial_t x_f(a(t))/\xi_a$. At the ramification
point, $dx_f=0$ and $dy\neq0$, so intrinsicity gives
$\Omega_t(a(t))=-\partial_t x_f(a(t))\,dy(a(t))$. The coordinate-change
field is therefore
$(\Omega_t(a)/dy(a))\xi_a^{-1}\partial_{\xi_a}$. Since
$dx_f=\xi_a\,d\xi_a$, pairing it against $B_t(r,p)B_t(r,q)$ in the Rauch
variational formula gives the signed residue in
\eqref{eq:physical-rauch}. Regular coordinate-change terms are \v{C}ech
coboundaries. The fixed-$x_f$ charts above trivialize the complement of the
ramification and puncture discs, so these polar fields represent the
deformation.

Choose transported $A$-cycle representatives away from the ramification
discs. Their fixed-$x_f$ transport is homologous to Gauss--Manin transport,
and the right side of \eqref{eq:physical-rauch} has zero $A$-period in either
external variable. Fixed-$x_f$ coordinates also preserve the diagonal
principal part away from the ramification points. The difference between
$\nabla_t^{x_f}B_t$ and the residue expression is consequently a
holomorphic bidifferential with all $A$-periods zero, and hence vanishes.

It remains to confirm the outer sign, which can be done on $\mathbb P^1$,
writing $x$ for the projection in the role of $x_f$; the fixed-$x$
derivative $\dot z\big|_x$ transports the point, so $\nabla_t^xB_t$ is the
transported derivative rather than the derivative at fixed $z$:
\begin{align*}
x(z)&=\frac{z^2}{2}+\lambda(t),
&
y(z)&=z,
&
B_t(z_1,z_2)&=\frac{dz_1\,dz_2}{(z_1-z_2)^2},
\\
\dot z\big|_x&=-\frac{\dot\lambda}{z},
&
\Omega_t&=-\dot\lambda\,dz,
&
\nabla_t^xB_t(z_1,z_2)
&=
\dot\lambda
\frac{dz_1\,dz_2}{z_1^2z_2^2},
\\
\operatorname*{Res}_{z=0}
\frac{\Omega_t(z)B_t(z,z_1)B_t(z,z_2)}
{dx(z)\,dy(z)}
&=
-\dot\lambda
\frac{dz_1\,dz_2}{z_1^2z_2^2}.
\end{align*}
The outer minus sign in \eqref{eq:physical-rauch} therefore gives the direct
fixed-$x$ derivative.

\emph{Step 3: the generalized-cycle projection.} We now connect this
deformation to the generalized-cycle decomposition
\eqref{eq:physical-tangent}. Let $\alpha$ be an $A$-normalized meromorphic
one-form whose poles lie at the ramification points and whose residues
vanish. Every stable correlator has these properties in each external
variable. The cross-sheet primitive $J_a$ gives $A$-normalization in the distinguished
variable, and symmetry gives it in all variables; ordinary recursion and
Lemma~\ref{lem:local-parity} give the pole and residue properties. Let $\Psi_a$ be a local primitive of the deformation form, $d\Psi_a=\Omega_t$
near $a$; the Riemann bilinear relation gives
$\int_{\Gamma_t}\alpha=\sum_{a\in\Ram(x_f)}\operatorname*{Res}_{p=a}\Psi_a(p)\alpha(p)$,
where the subscript on a generalized-cycle integral records the variable of
integration. For a weighted cycle summand $\beta_i$, the convention
$\int_{\beta_i(q)}B_t(p,q)=2\pi i\,\nu_i(p)$ gives exactly the corresponding
weighted $\beta_i$-period of $\alpha$. For a path from $p'$ to $p''$, the
normalized third-kind differential $\int_{p'}^{p''}B_t$ has residues $-1$ at
$p'$ and $+1$ at $p''$, and cutting the surface along the path, where
$\Psi_a$ becomes single-valued, gives $\int_{p'}^{p''}\alpha$ with the same
orientation.
Additive constants in $\Psi_a$ disappear because $\alpha$ has zero residues.
Moreover, a primitive of the local deck-even part of $\Omega_t$ is an
invariant holomorphic function, whose residue against a stable correlator
vanishes by Lemma~\ref{lem:local-parity}; only the local deck parity of
$\Omega_t$ enters.

\emph{Step 4: the $(0,2)$ insertion base.} This projection identity first
gives the insertion identity for the bidifferential.
In a ramification coordinate, write
$x_f=x_f(a)+\xi_a^2/2$,
$y=y(a)+y_1\xi_a+O(\xi_a^2)$, and, for the arguments
$(p_0,p_1,p_2)=(p,q,z)$ of \eqref{eq:physical-omega-zero-three},
$B_t(r,p_j)=(b_j+O(\xi_a))d\xi_a$ for $j=0,1,2$. The kernel $K_a$ of
\eqref{eq:EO-kernel} has leading term $b_0/(2y_1\xi_a\,d\xi_a)$. Evaluating
the second factor of each splitting term at $\sigma_a(r)$, that is, pulling
back by $\xi_a\mapsto-\xi_a$, makes the two source terms sum to
$-2b_1b_2(d\xi_a)^2+O(\xi_a)(d\xi_a)^2$. Hence
\begin{equation}\label{eq:physical-omega-zero-three}
\omega_{0,3;t}^c(p,q,z)
=
-\sum_a\operatorname*{Res}_{r=a}
\frac{B_t(r,p)B_t(r,q)B_t(r,z)}
{dx_f(r)\,dy(r)}.
\end{equation}
Substituting \eqref{eq:physical-tangent}, in the form
$\Omega_t=\int_{\Gamma_t}B_t$, into \eqref{eq:physical-rauch}, then
interchanging the finite generalized-cycle integral with the residues, the
supports being chosen away from the ramification discs as arranged below,
gives
$\nabla_t^{x_f}B_t(p,q)=\int_{\Gamma_t(z)}
\omega_{0,3;t}^c(p,q,z)$. This is the $(0,2)$ insertion base.

\emph{Step 5: from the Eynard--Orantin conventions to the correlator
identity.} It remains to propagate the identity through the recursion. Write
$W_{g,n;t}^{\rm EO}$ for the correlators of \cite{EO07} attached to the same
spectral data in their conventions. In \cite[Definition~4.1 and
footnote~3]{EO07} the kernel numerator is
$dE_r(p)=\frac12\int_r^{\sigma(r)}B_t(p,\cdot)$ and the denominator is
$(y(r)-y(\sigma(r)))\,dx_f(r)$; the numerator's integration path is opposite
to that of $K_a$ in \eqref{eq:EO-kernel}, the normalizations otherwise agree
(the factor $\tfrac12$ matches the cited kernel with the path reversed), and
the cited recursion carries no further sign at the residue. Consequently
$K_a$ is the negative of the \cite{EO07} kernel. If $k=2g-2+n$, in every
genus-reduction or splitting term the complexities of the factors sum to
$k-1$, so induction gives
\begin{equation}\label{eq:physical-eo-sign}
\omega_{g,n;t}^c=(-1)^{2g-2+n}W_{g,n;t}^{\rm EO}.
\end{equation}
At complexity one this agrees with
\eqref{eq:physical-omega-zero-three}.

After shrinking the base, choose all cycle and path representatives away
from the ramification set; the neck path can be perturbed inside the cold
annulus without changing the value of $\int_{\Gamma_{\rm neck}}$ on the forms
considered here. In the
$A$-normalized case, the same oriented cycle carries the representation
$\Omega^{\rm EO}=\int B_t$ and the correlator insertion in
\cite[Lemma~5.1, eqs.~(5-12)--(5-13), and Theorem~5.1, eq.~(5-14)]{EO07}. Applied to
$\Omega^{\rm EO}=\int_{-\Gamma_t}B_t$, these give
$\nabla_t^{x_f}W_{g,n;t}^{\rm EO}
=-\int_{\Gamma_t}W_{g,n+1;t}^{\rm EO}$.
Combining this with \eqref{eq:physical-eo-sign} proves the stated
correlator identity. It holds first for external points off the generalized-cycle supports; both
sides are meromorphic there, and the identity persists under continuation
across the supports. Moving residue points add no term: the residue is computed in a coordinate
centered on the moving section, and is therefore unaffected by its motion,
while the singular discrepancy between that transport and fixed-$x_f$
transport is the coordinate-change field already included in
\eqref{eq:physical-rauch}.

\emph{Step 6: logarithmic recursion and the free energy.} Proposition~\ref{thm:generic-framing} gives an empty vital set.
By \cite[Remark~2.3 and Definition~3.1]{HMO26}, logarithmic recursion then
equals ordinary recursion, and every additional free-energy term, being
indexed by a vital puncture, vanishes. Stable correlators have poles only
at ramification points, so the third-kind integrations converge at their
puncture endpoints. A logarithm-branch change $y\mapsto y+c_y$ adds the locally
deck-even form $c_y\,dx_f$ to $\omega_{0,1}$. It cancels from the recursion denominator, and its
primitive is an invariant function of $x_f$ whose residue against
$\omega_{g,1}^c$ vanishes by Lemma~\ref{lem:local-parity}.

For $g\geq2$, differentiate the residue definition of $F_g^c$, recalled at
the end of Section~\ref{sec:gaussian-limit}, at fixed $x_f$. Let $d\Phi_a=\omega_{0,1}$ and $d\Psi_a=\Omega_t$; then
$\nabla_t^{x_f}\Phi_a-\Psi_a$ is locally constant and contributes no residue.
Using the correlator insertion identity and the generalized-cycle projection
gives
\begin{align*}
(2-2g)\partial_tF_g^c
&=
\int_{\Gamma_t}\omega_{g,1;t}^c
+
\int_{\Gamma_t(q)}
\sum_a\operatorname*{Res}_{p=a}
\Phi_a(p)\omega_{g,2;t}^c(p,q),
\\
\sum_a\operatorname*{Res}_{p=a}
\Phi_a(p)\omega_{g,2;t}^c(p,q)
&=(1-2g)\omega_{g,1;t}^c(q),
\\
(2-2g)\partial_tF_g^c
&=(2-2g)\int_{\Gamma_t}\omega_{g,1;t}^c.
\end{align*}
For the identity in the second line, scaling $\omega_{0,1}$ by $\mu$ scales
the recursion kernel \eqref{eq:EO-kernel} by $\mu^{-1}$, and recursion
induction makes $\omega_{g,1}^c$ homogeneous of weight $1-2g$.
Differentiating the recursion in $\mu$ at $\mu=1$ realizes the scale
derivative as the rooted-residue insertion on the left side, and the
homogeneity evaluates it to the displayed coefficient $1-2g$. Since $g\geq2$, division proves
$\partial_tF_g^c=\int_{\Gamma_t}\omega_{g,1;t}^c$.

\emph{Step 7: assembly and the bound on $R_g$.} Linearity and
\eqref{eq:physical-tangent} now give
\eqref{eq:physical-special-geometry}. The identity is instantaneous: only
the value of $\Gamma_t$ at the given $t$ enters, and its coefficients and
supports are not differentiated. On the marked base the absolute $A$-cycles
are Gauss--Manin flat and $B_t$ is normalized on them, so no polarization
correction arises.
In the nonseparating case the vanishing cycle $\gamma$ is one of those cycles. In the
separating case the absolute normalization cycles remain on the component
sides, while $\Gamma_{\rm neck}$ is the relative path with coefficient
$1/(2\pi i c_{\rm per})$ in the orientation fixed above.

Finally, a stable one-point correlator can have poles only at ramification
points: after the recursion-variable residues are taken, its
distinguished-variable dependence occurs through the cross-sheet primitives
$J_a$, whose poles arise at the residue points. It is therefore holomorphic
on every fixed puncture disc. Theorem~\ref{thm:exterior-bound} bounds its
coefficient, taken with respect to the fixed puncture coordinate, on the
boundary circle of such a disc, and the maximum-modulus principle extends
the same bound over the closed disc, hence through the puncture endpoint.
The theorem also bounds the correlator along all fixed hot paths and cycles.
Since $\Gamma_{\rm hot}$ is a finite sum of these supports with the bounded
coefficients derived above, including the variable third-kind coefficients,
\[
R_g=O(1).
\]
\end{proof}

\subsection{Deck descent and the bounded Euler defect}\label{sec:euler-descent}

The preceding sections control the conifold-polarized correlators and free energies on the quadratic cover $t=\eps^2c_{\rm neck}(\eps^2,s)$. This subsection first descends those results to the $t$-disc itself and then bounds the \emph{Euler defect}, the failure of $F_g^c$ to satisfy Euler's identity for a function homogeneous of degree $2-2g$ in the exact period; boundedness of this defect is what Section~\ref{sec:assembly} integrates into the gap. Theorem~\ref{thm:deck-descent} proves the descent, Corollary~\ref{cor:gaussian-limit-disc} records the resulting Gaussian leading limit on the $t$-disc for the assembly in Section~\ref{sec:assembly}, Lemma~\ref{lem:cold-collar} converts the two cold free-energy residues into the neck integral of special geometry, and Theorem~\ref{thm:euler-defect} combines these into the bound.

Work in the toric one-node family, good framing, conifold-adapted marking, and exact period of Propositions~\ref{thm:generic-framing}, \ref{thm:normal-form}, and \ref{prop:regime-atlas}, with $s$ the spectator moduli, on the quadratic cover
\[
v^2=u^2-\eps^2,
\qquad
t=\eps^2c_{\rm neck}(\eps^2,s),
\]
with $c_{\rm neck}$ holomorphic and nonzero. Let $\mathfrak d$ denote the \emph{$\eps$-deck involution} $\eps\mapsto-\eps$ of this cover; since $t$ depends only on $\eps^2$, the pulled-back fibers over $\eps$ and $-\eps$ coincide, and $\mathfrak d$ acts through this identification. It exchanges the cold ramification sections $a_+$ and $a_-$, and it is distinct from the fiberwise deck involutions $\sigma_a$. Cover the punctured $\eps$-disc by simply connected charts $U_i$ from its universal cover, write $\mathcal C_i$ for the pulled-back local family over $U_i$, and let
\[
h_{ij}\colon\mathcal C_i\longrightarrow\mathcal C_j
\]
be the family-transition map on an overlap; the transition maps satisfy the cocycle condition on triple overlaps.

\begin{thm}[$\eps$-deck descent of conifold-polarized recursion]\label{thm:deck-descent}
Assume the toric one-node setting of Propositions~\ref{thm:generic-framing}, \ref{thm:normal-form}, and~\ref{prop:regime-atlas}, with good framing, conifold-adapted marking with its $A$-normalized bidifferential $B_t$, and exact period
\[
t=c_{\rm per}\int_\gamma\omega_{0,1;t,s},
\qquad c_{\rm per}\neq0,
\]
on the quadratic cover and chart cover fixed above. Then the pulled-back correlators and free energies are single-valued on the full punctured $\eps$-disc with respect to the family-transition cocycle $\{h_{ij}\}$: they agree after pullback by every $h_{ij}$, compatibly on triple overlaps, and, in particular, a full loop
\[
\eps\mapsto e^{2\pi i}\eps
\]
acts trivially. Moreover:
\begin{enumerate}[label=\textup{(\arabic*)}]
\item For every stable type $(g,n)$,
\begin{equation}\label{eq:deck-correlators}
(\mathfrak d^{\times n})^*
\omega_{g,n;-\eps,s}^c
=
\omega_{g,n;\eps,s}^c.
\end{equation}
\item For every stable one-point type $(g,1)$ and every adapted $A$-cycle $A_i$,
\begin{equation}\label{eq:A-period-vanishing}
\int_{A_i}\omega_{g,1;t,s}^c=0.
\end{equation}
\item For every fixed $g\geq2$, if $\widetilde F_g^c(s,\eps)$ is the pullback of the conifold-polarized free energy, then
\begin{equation}\label{eq:deck-free-energy}
\widetilde F_g^c(s,-\eps)
=
\widetilde F_g^c(s,\eps).
\end{equation}
\end{enumerate}
Consequently the correlators descend as relative meromorphic forms, fiberwise meromorphic and varying holomorphically with the parameters, and $F_g^c$ descends as a jointly holomorphic single-valued function of $(s,t)$ on a product of a spectator neighborhood with the full punctured $t$-disc, in both one-neck topologies, locally uniformly in $s$ after shrinking.
\end{thm}

\begin{proof}
\emph{Step 1: the capped family and its monodromy.} The capped-family construction in Theorem~\ref{thm:actual-yamada} extends across the full product coordinate
\[
q=\eps^2.
\]
Its inherited marking extends across the $q$-disc. Reconstructing the original fiber by
\[
z_+z_-=q
\]
shows that a positive loop of $q$ has only the annular Dehn twist $T_\gamma$ as monodromy. A full loop of $\eps$ therefore has monodromy $T_\gamma^2$.

In the nonseparating case, choose a symplectic basis on the capped surface, with representatives disjoint from both cap discs, and add the vanishing cycle $\gamma$ to its $A$-cycles. The twist fixes $\gamma$ and every inherited representative. In the separating case, choose componentwise symplectic bases away from the cap discs. The annular twist is the identity on all these representatives. Thus every adapted $A$-cycle is fixed in both topologies.

\emph{Step 2: invariance of the bidifferential.} The pulled-back fibers over $\eps$ and $-\eps$ coincide, while the marking continued along a half-loop of $\eps$, that is, a full loop of $q$, returns changed by the annular mapping class $T_\gamma$, which fixes every adapted $A$-cycle; the two parameters therefore induce the same adapted $A$-polarization. Pull the continued bidifferential back by $\mathfrak d$. It remains symmetric, has the same diagonal double pole with biresidue one, has no other pole, and has zero period over every adapted $A_i$ in either variable. These properties determine it uniquely. Indeed, the difference is a holomorphic bidifferential; fixing one argument gives a holomorphic one-form with all $A$-periods zero, hence zero. Therefore
\begin{equation}\label{eq:B-invariance}
(\mathfrak d\times\mathfrak d)^*B_{-\eps}=B_\eps.
\end{equation}
The same argument applies to every family-transition map whose annular mapping class is a power of $T_\gamma$.

\emph{Step 3: logarithm branches and the recursion kernel.} On a simply connected parameter chart, choose local logarithms. After continuation, their branches may return as
\[
x_f\longmapsto x_f+c_x,
\qquad
y\longmapsto y+c_y,
\qquad
c_x,c_y\in2\pi i\Z.
\]
The constants lie in $2\pi i\Z$ because the branches of $\log X$ and $\log Y$ return shifted by integer multiples of $2\pi i$ and the framing $f$ is an integer. Thus $dx_f$ is unchanged and
\[
\omega_{0,1}\longmapsto\omega_{0,1}+c_y\,dx_f.
\]
At a simple ramification point,
\[
x_f\circ\sigma_a=x_f,
\qquad
\sigma_a^*dx_f=dx_f.
\]
Hence the added form cancels from the recursion denominator:
\[
(\omega_{0,1}+c_y dx_f)
-
\sigma_a^*(\omega_{0,1}+c_y dx_f)
=
\omega_{0,1}-\sigma_a^*\omega_{0,1}.
\]
Equation \eqref{eq:B-invariance} identifies the cross-sheet primitive $J_a$, and the family identification conjugates each local involution to the continued involution. The recursion kernel is invariant.

\emph{Step 4: $\mathfrak d$-evenness of the correlators.} Induct on $2g-2+n$ in the complete recursion formula \eqref{eq:EO-recursion}:
\begin{align*}
\omega_{g,n}(p_0,\mathbf p)
={}&
\sum_{a\in\Ram(x_f)}
\operatorname*{Res}_{q=a}
K_a(p_0,q)
\bigg[
\mathbf 1_{g\geq1}\,\omega_{g-1,n+1}(q,\sigma_a(q),\mathbf p)\\
&+
\sum_{\substack{g_1+g_2=g\\I\sqcup J=\mathbf p}}^{\prime}
\omega_{g_1,|I|+1}(q,I)
\omega_{g_2,|J|+1}(\sigma_a(q),J)
\bigg].
\end{align*}
The base of the induction is the bidifferential, invariant by \eqref{eq:B-invariance}. The kernel is invariant, and the induction hypothesis identifies every stable source. Continuation permutes the finite ramification set, exchanging the two cold labels under $\mathfrak d$ and possibly permuting hot labels. Residues are invariant under biholomorphic coordinate changes, and their finite sum is invariant under permutation. This proves \eqref{eq:deck-correlators}.

\emph{Step 5: vanishing of the adapted $A$-periods.} To prove \eqref{eq:A-period-vanishing}, choose pairwise disjoint residue discs around the ramification points. Every adapted $A$-cycle has a representative disjoint from them. In the nonseparating case, displace $\gamma$ radially so that it misses the cold residue discs. For each ramification point, choose the path from $\sigma_a(p)$ to $p$ inside its residue disc. It is disjoint from every displaced $A_i$. The two contours are compact and disjoint and $B_t$ is continuous on their product, so the integrations interchange:
\[
\int_{A_i(p_0)}
\left(
\frac12\int_{\sigma_a(p)}^pB_t(p_0,\cdot)
\right)
=
\frac12\int_{\sigma_a(p)}^p
\left(
\int_{A_i(p_0)}B_t(p_0,\cdot)
\right)
=0.
\]
The denominator and recursion source are independent of $p_0$, so integrating the recursion formula over the displaced $A_i$ gives zero. A stable correlator, as a one-form in an external argument, has poles only at ramification points. Lemma~\ref{lem:local-parity} says that each such pole has zero residue. Moving the cycle across a ramification point therefore does not change its period. This proves \eqref{eq:A-period-vanishing} for the original adapted cycle.

\emph{Step 6: free energies, gluing, and descent.} For $g\geq2$,
\[
F_g^c
=
\frac{1}{2-2g}
\sum_{a\in\Ram(x_f)}
\operatorname*{Res}_{p=a}
\Phi_a(p)\omega_{g,1}^c(p),
\qquad
d\Phi_a=\omega_{0,1}.
\]
By \eqref{eq:deck-correlators} the one-point correlators agree, while $\mathfrak d$ permutes the ramification points. If no logarithm branch changes, centered local primitives prove invariance. Under the branch shift above, a local primitive changes by
\[
c_yx_f+c_0,
\]
with $c_0$ locally constant. Both $x_f$ and $1$ are invariant under the local involution. Lemma~\ref{lem:local-parity} gives
\[
\operatorname*{Res}_{p=a}
(c_yx_f+c_0)\omega_{g,1}^c(p)=0.
\]
Thus branch changes alter no summand, proving \eqref{eq:deck-free-energy}.

Consider the chart cover $\{U_i\}$ and the transition maps $h_{ij}$ fixed before the theorem. If two lifts differ by $N_{ij}$ turns, then
\[
[h_{ij}]=T_\gamma^{2N_{ij}},
\]
because the plumbing product is $\eps^2$. The logarithm branches satisfy
\[
h_{ij}^*x_{f,j}
=
x_{f,i}+2\pi i\,m^x_{ij},
\qquad
h_{ij}^*y_j
=
y_i+2\pi i\,m^y_{ij}
\]
for locally constant integers $m^x_{ij}$ and $m^y_{ij}$. The transition maps satisfy the cocycle condition on triple overlaps, and the branch shifts are additive there: $m^x_{ij}+m^x_{jk}=m^x_{ik}$ and $m^y_{ij}+m^y_{jk}=m^y_{ik}$.

Every adapted $A$-cycle is fixed by the transition. Uniqueness of the $A$-normalized bidifferential gives
\[
(h_{ij}\times h_{ij})^*B_j=B_i.
\]
The shift of $\omega_{0,1}$ cancels from the recursion denominator. Induction gives
\[
(h_{ij}^{\times n})^*
\omega_{g,n;j,s}^c
=
\omega_{g,n;i,s}^c.
\]
The local-parity residue argument gives
\[
F_{g,j}^c=F_{g,i}^c.
\]
These identities satisfy the cocycle condition on triple overlaps, so the local objects glue on the full punctured $\eps$-disc.

The $\eps$-deck involution $\mathfrak d$ commutes with the transition cocycle and exchanges $a_+$ and $a_-$. Equations \eqref{eq:deck-correlators} and \eqref{eq:deck-free-energy} say that the glued objects are $\mathfrak d$-invariant, so they descend to the punctured disc with coordinate
\[
q=\eps^2.
\]
Since
\[
t=q\,c_{\rm neck}(q,s),
\qquad
c_{\rm neck}(0,s)\neq0,
\]
the map $(q,s)\mapsto(t,s)$ is a biholomorphism onto its image after shrinking. This proves joint holomorphy and single-valuedness on the full punctured $t$-disc. A square-root branch is excluded by \eqref{eq:deck-free-energy}, and a logarithmic branch by the trivial action of a full loop.
\end{proof}

\begin{cor}[Gaussian leading limit on the $t$-disc]\label{cor:gaussian-limit-disc}
Fix an integer $g\geq2$. Then
\[
t^{2g-2}F_g^c(s,t)
=
F_g^G+O(|t|^{1/2})
\]
as $t\to0$ on the full punctured $t$-disc, locally uniformly in $s$. The left-hand side and $F_g^G$ are single-valued in $t$; $|t|^{1/2}$ records only the size of the remainder.
\end{cor}

\begin{proof}
Theorem~\ref{thm:deck-descent} makes $F_g^c$ a jointly holomorphic single-valued function on the full punctured $t$-disc, and Theorem~\ref{thm:gaussian-limit} gives the expansion on the quadratic cover. Proposition~\ref{prop:regime-atlas} gives $|\eps|=O(|t|^{1/2})$ locally uniformly in $s$. Both lifts of a small nonzero $t$ give the same left side and the same Gaussian residue sum, while the error is bounded by a uniform multiple of $|t|^{1/2}$.
\end{proof}

Work in the cut annulus of Proposition~\ref{prop:regime-atlas}. Fix $g\geq2$, and put
\[
\alpha=\omega_{g,1;t,s}^c,
\qquad
P=\frac{t}{c_{\rm per}}.
\]
Let $\Phi_{\rm cut}$ be the cut-annulus primitive of $\omega_{0,1}$ fixed before that proposition, normalized at a fixed point of the outer boundary, and let $B_g^{\rm cold}$ be the unnormalized sum of the two cold free-energy residues. In the nonseparating case, let $L_{\rm neck}$ be the adapted $B$-cycle, the completion of the cross-cut $\delta_{\rm neck}$ by marked paths in the fixed hot region (Proposition~\ref{prop:regime-atlas}); in the separating case, let $L_{\rm neck}=\ell_{\rm sep}$ be the path of Proposition~\ref{prop:gamma-decomposition}, whose annular segment traverses the neck along the outer-to-inner cross-cut $\delta_{\rm neck}$, consistently with the convention $\gamma\cdot\delta_{\rm neck}=+1$ of Proposition~\ref{prop:regime-atlas}. In both cases set
\[
h_{\rm neck}=L_{\rm neck}-\delta_{\rm neck},
\]
the hot part of $L_{\rm neck}$. Proposition~\ref{prop:gamma-decomposition} gives, in both topologies,
\[
\Gamma_{\rm neck}
=
\frac{1}{2\pi i c_{\rm per}}L_{\rm neck};
\]
hence
\[
I_g^{\rm neck}
=
\frac{1}{2\pi i c_{\rm per}}
\int_{L_{\rm neck}}\alpha.
\]

\begin{lem}[Cold collar identity]\label{lem:cold-collar}
In this setting,
\[
B_g^{\rm cold}
=
tI_g^{\rm neck}+H_g^{\rm collar},
\]
where
\begin{equation}\label{eq:cold-collar-remainder}
H_g^{\rm collar}
=
\frac{1}{2\pi i}
\left[
\int_{C_{\rm out}}\Phi_{\rm cut}\alpha
+
\int_{C_{\rm in}}\Phi_{\rm cut}\alpha
-
P\int_{h_{\rm neck}}\alpha
\right].
\end{equation}
Moreover,
\[
H_g^{\rm collar}=O(1)
\]
on the full punctured $t$-disc, locally uniformly in $s$ and per fixed $g$, in both one-neck topologies.
\end{lem}

\begin{proof}
\emph{Step 1: banks and orientations.} Write
\[
z_+=e^{\rho+i\theta},
\qquad
0<\theta<2\pi
\]
on the cut annulus. Its complex orientation is $d\rho\wedge d\theta$. The bank $\theta=2\pi$ is therefore oriented from the outer boundary to the inner boundary, hence as $\delta_{\rm neck}$, while the bank $\theta=0$ has the opposite orientation. If $\Phi_{\rm top}$ and $\Phi_{\rm bottom}$ are the corresponding boundary values, then, since positive continuation once around $\gamma$ changes $\Phi_{\rm cut}$ by $\int_\gamma\omega_{0,1}=P$ (Proposition~\ref{prop:regime-atlas}),
\[
\Phi_{\rm top}=\Phi_{\rm bottom}+P.
\]
The two bank integrals consequently combine with positive sign to
\[
P\int_{\delta_{\rm neck}}\alpha.
\]

\emph{Step 2: the collar residue identity.} The form $\alpha$ has poles only at ramification points, and exactly the two cold ramification points lie in this annulus. By Lemma~\ref{lem:local-parity} both residues of $\alpha$ vanish, so the two cold residues in $B_g^{\rm cold}$ are unchanged when the primitives $\Phi_a$ are replaced by $\Phi_{\rm cut}$. Applying the residue theorem to $\Phi_{\rm cut}\alpha$ on the cut annulus, whose boundary is $C_{\rm out}+C_{\rm in}$ together with the two banks (Proposition~\ref{prop:regime-atlas}), gives
\begin{equation}\label{eq:cold-collar-residue}
2\pi i B_g^{\rm cold}
=
\int_{C_{\rm out}}\Phi_{\rm cut}\alpha
+
\int_{C_{\rm in}}\Phi_{\rm cut}\alpha
+
P\int_{\delta_{\rm neck}}\alpha.
\end{equation}
\emph{Step 3: the neck integral.} Since $\Gamma_{\rm neck}=(2\pi i c_{\rm per})^{-1}L_{\rm neck}$ in both topologies, $P=t/c_{\rm per}$, and
$L_{\rm neck}=\delta_{\rm neck}+h_{\rm neck}$,
\[
tI_g^{\rm neck}
=
\frac{P}{2\pi i}
\left[
\int_{\delta_{\rm neck}}\alpha
+
\int_{h_{\rm neck}}\alpha
\right].
\]
Subtracting this equality from \eqref{eq:cold-collar-residue} proves
\eqref{eq:cold-collar-remainder} and the exact collar identity.

\emph{Step 4: bounding the remainder.} In the coordinates of Proposition~\ref{thm:normal-form}, the initial form on the neck splits into its deck-even and deck-odd parts,
\[
\omega_{0,1}
=
E(u,\lambda,s)\,du
+
M(u,\lambda,s)v\,du,
\qquad
\lambda=\eps^2,
\]
where the odd part $Mv\,du$ is supplied by that proposition and the even part defines $E$; both coefficients are holomorphic and bounded, locally uniformly in $s$. The even term has a bounded holomorphic primitive. Along a radial cross-cut, with $z=z_+$,
\[
u=\frac{z+\lambda/z}{2},
\qquad
v=\frac{z-\lambda/z}{2},
\qquad
du=\frac{1-\lambda/z^2}{2}\,dz.
\]
Writing $r=|z|$ gives
\[
|v\,du|
\leq
C\left(
r+\frac{2|\lambda|}{r}
+\frac{|\lambda|^2}{r^3}
\right)dr.
\]
Its integral from $|\lambda|/R$ to $R$ is
\[
O\!\left(
1+|\lambda|\lvert\log|\lambda|\rvert
\right)
=
O(1).
\]
A point of either fixed plumbing circle, or of a bank, is reached from the normalization point by radial segments as above and arcs of the two fixed circles, on which $\omega_{0,1}$ is bounded; hence $\Phi_{\rm cut}=O(1)$ on the circles and the banks. The borderline logarithmic contribution $(P/2\pi i)\Log_{\rm cut}(z_+)$ of Proposition~\ref{prop:regime-atlas} is controlled the same way: $|\Log_{\rm cut}(z_+)|=O(\log(1/|\lambda|))$ on the annulus, while $P=O(\lambda)$.

Theorem~\ref{thm:exterior-bound} gives $\alpha=O(1)$ on
$C_{\rm out}$, $C_{\rm in}$, and every fixed hot interior segment of
$h_{\rm neck}$. In the separating case, the completion may end at fixed punctures. The maximum-modulus extension through fixed puncture discs established in Step~7 of the proof of Theorem~\ref{thm:physical-special-geometry} supplies the same bound there. Thus the two circle integrals and
$\int_{h_{\rm neck}}\alpha$ are bounded. Since $P=O(t)$, every term in
\eqref{eq:cold-collar-remainder} is $O(1)$.

A finite collection of radial cut directions covers the punctured parameter disc. On overlaps,
\[
H_g^{\rm collar}
=
B_g^{\rm cold}-tI_g^{\rm neck}
\]
is the difference of intrinsic quantities, so the local expressions patch. All estimates are locally uniform in $s$.
\end{proof}

For $g\geq2$, the free-energy residue sum splits into its cold and hot parts,
\[
(2-2g)F_g^c
=
B_g^{\rm cold}+B_g^{\rm hot},
\]
with $B_g^{\rm cold}$ as above and $B_g^{\rm hot}$ the corresponding sum over the hot ramification points, and the special-geometry decomposition of Theorem~\ref{thm:physical-special-geometry} reads
\[
\partial_tF_g^c=I_g^{\rm neck}+R_g,
\qquad
R_g=O(1).
\]
A function homogeneous of degree $2-2g$ in $t$ satisfies $t\partial_tF=(2-2g)F$. The \emph{Euler defect}
\[
\mathcal E_g
=
(2-2g)F_g^c-t\partial_tF_g^c
\]
measures the failure of this identity for the conifold free energy: it vanishes identically precisely when $F_g^c(s,t)=C(s)\,t^{2-2g}$ with $C$ independent of $t$, and its boundedness on the punctured $t$-disc is what Section~\ref{sec:assembly} converts into the gap.

\begin{thm}[Bounded Euler defect]\label{thm:euler-defect}
Assume the toric one-node setting of Propositions~\ref{thm:generic-framing}, \ref{thm:normal-form}, and~\ref{prop:regime-atlas}, with good framing, conifold-adapted marking, exact period, and projection-compatible neck, and fix $g\geq2$. Then $B_g^{\rm hot}$, $B_g^{\rm cold}$, $I_g^{\rm neck}$, $R_g$, and $F_g^c$ are jointly holomorphic and single-valued on a product of a spectator neighborhood with the full punctured $t$-disc, and there is a jointly holomorphic function $H_g$ there such that
\begin{equation}\label{eq:dilaton-neck}
(2-2g)F_g^c
=
tI_g^{\rm neck}+H_g,
\qquad
H_g=O(1).
\end{equation}
Consequently the Euler defect satisfies
\begin{equation}\label{eq:euler-defect}
\mathcal E_g
=
H_g-tR_g
=
O(1).
\end{equation}
All bounds are locally uniform in $s$ and per fixed $g$, hold in both one-neck topologies, and cover the full family deformation, including the bounded hot residual $R_g$.
\end{thm}

\begin{proof}
\emph{Step 1: the chart-level dilaton identity.} Theorem~\ref{thm:deck-descent} proves that $\alpha$ and $F_g^c$ descend as jointly holomorphic single-valued objects to the full punctured $t$-disc. It also gives
\[
\int_{A_i}\alpha=0
\]
for every adapted $A$-cycle.

Work first on a simply connected marked chart. Normalize each hot primitive by $\Phi_a(a)=0$. Lemma~\ref{lem:local-parity} gives $\alpha$ zero residue at every ramification point, so this normalization does not change the residue sum. The primitive is $O(1)$ on a fixed hot residue circle, and Theorem~\ref{thm:exterior-bound} gives $\alpha=O(1)$ there. Hence
\[
B_g^{\rm hot}=O(1).
\]

For the cold pair, Lemma~\ref{lem:cold-collar} gives the exact identity
\[
B_g^{\rm cold}
=
tI_g^{\rm neck}+H_g^{\rm collar},
\qquad
H_g^{\rm collar}=O(1).
\]
The lemma supplies both the explicit plumbing-circle formula for
$H_g^{\rm collar}$ and its locally uniform bound on the full punctured $t$-disc. Set
\[
H_g=B_g^{\rm hot}+H_g^{\rm collar}.
\]
This proves \eqref{eq:dilaton-neck} on the marked chart.

\emph{Step 2: the defect bound.} Subtracting $t$ times the special-geometry decomposition of Theorem~\ref{thm:physical-special-geometry} from \eqref{eq:dilaton-neck} gives
\[
\mathcal E_g=H_g-tR_g=O(1).
\]
The two occurrences of $tI_g^{\rm neck}$ cancel exactly, so the bound rests on $H_g$ and $R_g$ alone.

\emph{Step 3: descent of the neck integral and residual.} It remains to descend the auxiliary objects. In the nonseparating case, continuation once around $q=0$ changes the adapted dual cycle $\beta_1$ by the Picard--Lefschetz monodromy, adding $\gamma$ or $-\gamma$ according to the loop-orientation convention (Proposition~\ref{prop:regime-atlas}), and fixes $A_1=\gamma$. The change in $I_g^{\rm neck}$ is a multiple of
\[
\int_{A_1}\alpha=0.
\]
In the separating case, continuation can change the relative crossing path only by the vanishing cycle $\gamma$, which bounds a subsurface. The poles of $\alpha$ lie only at ramification points, so $\alpha$ is holomorphic at the punctures, and each pole has zero residue by Lemma~\ref{lem:local-parity}; the residue theorem gives
\[
\int_\gamma\alpha=0.
\]
Thus $I_g^{\rm neck}$ is single-valued in both topologies.

The deformation form represented by $\Gamma_t$ is intrinsic; the representative decomposes as
\[
\Gamma_{\rm neck}+\Gamma_{\rm hot}.
\]
Continuation acts on the marking data through powers of $T_\gamma$, so its supports can change only by adapted $A$-cycles, the vanishing cycle $\gamma$, and small loops around fixed puncture endpoints. The first two pairings with $\alpha$ vanish as above. The one-point correlator is holomorphic at punctures because all its poles are at ramification points, so the endpoint loops also integrate to zero. Therefore
\[
R_g=\int_{\Gamma_{\rm hot}}\alpha
\]
is single-valued.

\emph{Step 4: single-valuedness of the remainder.} By Theorem~\ref{thm:deck-descent}, monodromy permutes the ramification sections and leaves $\alpha$ invariant. A logarithm-branch change sends a local primitive to
\[
\Phi_a+c_yx_f+c_0.
\]
The functions $x_f$ and $1$ are invariant under the local involution, so Lemma~\ref{lem:local-parity} gives
\[
\operatorname*{Res}_{p=a}
[(c_yx_f+c_0)\alpha]=0.
\]
Thus neither branch changes nor permutations alter $B_g^{\rm hot}$ or $B_g^{\rm cold}$. It follows that
\[
H_g^{\rm collar}=B_g^{\rm cold}-tI_g^{\rm neck}
\]
and $H_g$ are jointly holomorphic and single-valued. All constituent estimates hold on the full punctured $t$-disc, locally uniformly in $s$.
\end{proof}

\section{Assembly and calibration of the gap}\label{sec:assembly}

This section assembles the gap from the descent, the bounded Euler defect, and the Gaussian leading limit, and then evaluates the universal constant against the Gaussian matrix model.

\subsection{Assembly of the gap}

We restate the main theorem in its general form, with an arbitrary period normalization; Theorem~\ref{thm:main-gap-intro} follows by specializing to $c_{\rm per}=1/(2\pi i)$ and evaluating the constant through Corollary~\ref{cor:main-calibrated}; the deduction is carried out at the end of Section~\ref{sec:calibration}.

\begin{thm}[Generic one-neck conifold gap]\label{thm:main-gap}
Let $\mathfrak X$ be a smooth toric Calabi--Yau threefold, and let
\[
H(X,Y;q,s)=0
\]
be a local analytic family of its reduced compactified mirror curves with fixed two-dimensional Newton polygon $\Delta$. Here $q$ is transverse to the smooth conifold divisor, the discriminant locus along which the fiber acquires its node, and $s$ denotes the spectator moduli. Assume that at $(q,s)=(0,s_0)$:

\begin{enumerate}[label=\textup{(\arabic*)}]
\item the compactified curve has exactly one ordinary double point, at a point $(X_0,Y_0)$ of the dense torus $(\C^*)^2$;
\item the $q$-direction smooths this node transversely, that is, $\partial H/\partial q\neq0$ at $(X_0,Y_0;0,s_0)$;
\item the central compactified curve is otherwise smooth, so this node is its only singular point;
\item the vertex coefficients are nonzero and the puncture sections are pairwise disjoint.
\end{enumerate}

Choose a good integer framing $f$ and put
\[
x_f=\log X+f\log Y,\qquad
y=\log Y,\qquad
\omega_{0,1}=y\,dx_f.
\]
Here good means that $f\in\Z\setminus(B_{\rm res}\cup B_{\rm proj})$, where $B_{\rm res}$ and $B_{\rm proj}$ are the finite exceptional sets of Proposition~\ref{thm:generic-framing}. For such $f$, ordinary Eynard--Orantin recursion applies, exactly two simple cold ramification points enter the neck, and every other ramification point and every puncture stays in a fixed hot region (Proposition~\ref{thm:generic-framing}), locally uniformly in $s$.

Equip the smooth fibers with the conifold-adapted Torelli marking fixed before Proposition~\ref{prop:regime-atlas}, and let $\gamma$ be the oriented vanishing cycle fixed there, defined in both topologies. If the neck is nonseparating, $\gamma$ is the first $A$-cycle; if the neck is separating, the componentwise limiting $A$-markings are used, with the relative generalized cycle $\delta_{\rm neck}$ of Proposition~\ref{prop:regime-atlas} crossing the neck. Let $B_t$ be the resulting normalized fundamental bidifferential. For a fixed scalar $c_{\rm per}\in\C^*$, define
\begin{equation}\label{eq:main-period}
t=c_{\rm per}\int_\gamma\omega_{0,1}.
\end{equation}
Then $(t,s)$ is a holomorphic coordinate system near $(0,s_0)$ in which the discriminant is $\{t=0\}$; in particular $t$ is single-valued and holomorphic. In a projection-compatible neck (Proposition~\ref{thm:normal-form}), on the quadratic cover,
\[
v^2=u^2-\eps^2,\qquad
t=\eps^2c_{\rm neck}(\eps^2,s),
\]
where $c_{\rm neck}$ is holomorphic and bounded away from zero, locally uniformly in $s$.

For every fixed integer $g\geq2$, define the conifold-polarized Eynard--Orantin free energy by
\begin{equation}\label{eq:main-free-energy}
F_g^c(s,t)=
\frac{1}{2-2g}
\sum_{a\in\Ram(x_f)}
\operatorname*{Res}_{p=a}
\Phi_a(p)\,\omega_{g,1;t,s}^c(p),
\qquad
d\Phi_a=\omega_{0,1;t,s}.
\end{equation}
Here $\omega_{g,1;t,s}^c$ are the one-point correlators of Eynard--Orantin recursion with initial data $(\omega_{0,1},B_t)$, and $\Phi_a$ is a local primitive near $a$; by Theorem~\ref{thm:deck-descent}, $F_g^c$ is jointly holomorphic and single-valued on the punctured product of a spectator neighborhood with the $t$-disc. After shrinking to a product $U_{\rm sp}\times D_t$, where $U_{\rm sp}$ may have complex dimension zero, there are a unique constant $C_g\in\C$, independent of $s$, and a unique jointly holomorphic function
\[
H_g^{\rm gap}\in\Ocal(U_{\rm sp}\times D_t)
\]
such that, for $t\neq0$,
\begin{equation}\label{eq:main-gap}
F_g^c(s,t)=C_gt^{2-2g}+H_g^{\rm gap}(s,t).
\end{equation}
In particular, the Laurent expansion of $F_g^c$ in $t$ has no logarithmic term and no negative power other than $t^{2-2g}$. The shrinking and the holomorphic extension are uniform for $s$ in compact subsets of the smooth conifold divisor. The conclusion holds for each fixed $g$ and covers both nonseparating and separating nodes.
\end{thm}

The constant $C_g$ is evaluated by comparison with the Gaussian Hermitian one-matrix model, whose spectral curve is the Gaussian curve
\[
G=\{w^2=z^2-1\},\qquad
z=\frac{\zeta+\zeta^{-1}}{2},\qquad
w=\frac{\zeta-\zeta^{-1}}{2}
\]
of Section~\ref{sec:normal-form}, with initial form $-\tfrac12w\,dz$; Section~\ref{sec:calibration} rescales this form to the exact-period normalization. Let $\gamma_G=\{|\zeta|=1\}$ be the counterclockwise unit circle, whose image under $\Phi_\eps$ is the oriented vanishing cycle $\gamma$, and put $\kappa=\int_{\gamma_G}w\,dz$, which is nonzero by Proposition~\ref{thm:period-split}. The Bernoulli numbers are defined by $\frac{x}{e^x-1}=\sum_{m\geq0}B_m\frac{x^m}{m!}$. The following corollary is proved in Section~\ref{sec:calibration}.

\begin{cor}[Calibrated universal coefficient]\label{cor:main-calibrated}
In the setting of Theorem~\ref{thm:main-gap}, for every integer $g\geq2$, the constant $C_g$ of \eqref{eq:main-gap} satisfies
\begin{equation}\label{eq:main-calibrated}
C_g=
\frac{B_{2g}}{2g(2g-2)}
(2c_{\rm per}\kappa)^{2g-2}.
\end{equation}
The coefficient $C_g$ is nonzero, depends only on the period normalization and the recursion and free-energy conventions, and is independent of $\mathfrak X$, of $\Delta$, and of the point $s$ of the smooth conifold divisor. Reversing the orientation of the vanishing cycle replaces $\kappa$ by $-\kappa$ and $t$ by $-t$; since the exponent $2g-2$ is even, both $C_g$ and the polar term $t^{2-2g}$ are unchanged. For the counterclockwise orientation fixed above, $\kappa=-\pi i$; if $c_{\rm per}=1/(2\pi i)$, then
\[
C_g=\frac{B_{2g}}{2g(2g-2)}.
\]
\end{cor}

The remaining step is complex-analytic and involves no toric input. We state it for an abstract family: the two exact identities below are the shape produced by Theorems~\ref{thm:physical-special-geometry} and~\ref{thm:euler-defect}, with a general bounded residual allowed. The letters $F$, $B$, $H$, $I$, $R$ are local to the lemma and are matched with the objects of the paper only in the proof of Theorem~\ref{thm:main-gap}.

\begin{lem}[Parametric Laurent assembly]\label{lem:assembly}
Let $d$ be a positive integer, let $U\subset\C^d$ be open, let $\delta>0$, and put
\[
D=\{t\in\C:|t|<\delta\},
\qquad
D^*=\{t\in\C:0<|t|<\delta\}.
\]
Let $g\geq2$ be an integer, and let
\[
F,\ B,\ H,\ I,\ R
\]
be holomorphic functions on $U\times D^*$. Assume that $B$, $H$, and $R$ are $O(1)$ as $t\to0$, locally uniformly in $u$, and assume the exact identities
\begin{equation}\label{eq:assembly-identities}
(2-2g)F=B+H+tI,
\qquad
\partial_tF=I+R.
\end{equation}
Define
\[
E=(2-2g)F-t\partial_tF.
\]
Then
\[
E=B+H-tR,
\]
so $E$ is bounded as $t\to0$, locally uniformly in $u$.

If, in addition, a constant $C\in\C$ satisfies
\[
t^{2g-2}F(u,t)\longrightarrow C
\]
locally uniformly in $u$ as $t\to0$, then there is a unique holomorphic function
\[
W\in\Ocal(U\times D)
\]
such that
\[
F(u,t)=Ct^{2-2g}+W(u,t)
\]
on $U\times D^*$. In particular the Laurent expansion of $F$ in $t$ carries no logarithm, because $F$ is single-valued on the full punctured disc, and no negative power other than $t^{2-2g}$, by the boundedness and limit hypotheses.
\end{lem}

\begin{proof}
\emph{Step 1: exact cancellation of $tI$.} Subtract $t$ times the second identity in \eqref{eq:assembly-identities} from the first:
\[
E
=
B+H+tI-t(I+R)
=
B+H-tR.
\]
The two occurrences of $tI$ cancel exactly, so the bound rests on $B$, $H$, and $R$ alone.

\emph{Step 2: a locally uniform bound on $E$.} Fix a compact $K\subset U$. There are constants $M>0$ and $0<\delta_K<\delta$ such that
\[
|B(u,t)|+|H(u,t)|+|R(u,t)|
\leq M
\]
for $u\in K$ and $0<|t|<\delta_K$. Hence
\[
|E(u,t)|\leq M+\delta_KM
\]
on the same set.

\emph{Step 3: holomorphic extension of $E$ across $t=0$.} This boundedness gives a joint holomorphic extension of $E$. Fix $u_0\in U$ and choose a compact closed polydisc $P\subset U$ whose interior contains $u_0$. Apply the preceding bound with $K=P$, obtaining $\delta_P$, and choose $0<r<\delta_P$. For $u$ in the interior of $P$ and $|t|<r$, define
\[
E^{\rm ext}(u,t)
=
\frac{1}{2\pi i}
\int_{|\tau|=r}
\frac{E(u,\tau)}{\tau-t}\,d\tau.
\]
This fixed-contour integral is jointly holomorphic in $(u,t)$. If
$0<|t|<r$, orient both circles counterclockwise and apply the residue theorem on
\[
\rho<|\tau|<r,
\qquad 0<\rho<|t|.
\]
The integrand has its only pole in this annulus at $\tau=t$, so the outer integral equals $E(u,t)$ plus the inner-circle integral. Since $E$ is bounded, the absolute value of the inner-circle integral is at most a constant multiple of
\[
\frac{\rho}{|t|-\rho},
\]
which tends to zero as $\rho\to0$. Thus $E^{\rm ext}=E$ on the punctured product. The local extensions agree on overlaps by the identity theorem. Hence $E$ extends to a jointly holomorphic function on $U\times D$.

\emph{Step 4: the Laurent coefficients.} Fix $u\in U$ and a circle $|\tau|=r$ contained in $D^*$. For every integer $k$, define
\[
a_k(u)
=
\frac{1}{2\pi i}
\int_{|\tau|=r}
F(u,\tau)\tau^{-k-1}\,d\tau,
\]
and define $e_k(u)$ by the same formula with $E$ in place of $F$. Since $F$ and $E$ are holomorphic in $t$ on $D^*$, these integrals do not depend on $r$; they are the Laurent coefficients of $F$ and $E$. Since $E$ extends holomorphically to $t=0$,
\[
e_k(u)=0
\qquad(k<0).
\]

Using
\[
E=(2-2g)F-t\partial_tF
\]
and integrating by parts around the circle gives
\[
e_k(u)=(2-2g)a_k(u)-ka_k(u)=(2-2g-k)a_k(u).
\]
Indeed,
\[
d\bigl(F\tau^{-k}\bigr)
=
\bigl(\tau\partial_\tau F-kF\bigr)\tau^{-k-1}\,d\tau
\]
integrates to zero over the closed circle, so the Laurent coefficient of
$\tau\partial_\tau F$ at exponent $k$ is $ka_k(u)$. Hence $a_k(u)=0$ for every negative integer $k\neq2-2g$. The Laurent series of $F$ has at most one negative-power term, at exponent $2-2g$.

The absence of every other negative coefficient gives
\[
\lim_{t\to0}t^{2g-2}F(u,t)=a_{2-2g}(u).
\]
The assumed limit gives
\[
a_{2-2g}(u)=C.
\]

\emph{Step 5: construction of the remainder.} Put
\[
W^\circ=F-Ct^{2-2g}
\]
on $U\times D^*$. Its Laurent series has no negative powers. On the interior of $P$ times $|t|<r$, with $P$ and $r$ as in the extension step, define
\[
W(u,t)
=
\frac{1}{2\pi i}
\int_{|\tau|=r}
\frac{W^\circ(u,\tau)}{\tau-t}\,d\tau.
\]
The absence of negative Laurent coefficients implies that this equals $W^\circ$ for $0<|t|<r$. It is jointly holomorphic at $t=0$. The local definitions agree on overlaps because they agree on the punctured overlaps, so they patch to a holomorphic function on $U\times D$. Uniqueness follows from the identity theorem.
\end{proof}

\begin{rem}[Zero-dimensional spectator space]\label{rem:assembly-zero-dim}
The same proof applies when $d=0$, that is, when $U$ is a single point: the polydisc $P$ and the parameter $u$ drop out, and the Cauchy integrals in $t$ are taken without parameters. This is the case $\dim U_{\rm sp}=0$ of Theorem~\ref{thm:main-gap}.
\end{rem}

\begin{proof}[Proof of Theorem~\ref{thm:main-gap}]
Every hypothesis of Lemma~\ref{lem:assembly} is supplied by a result of Sections~\ref{sec:setup}--\ref{sec:defect}; we collect the inputs and apply the lemma. By Proposition~\ref{thm:generic-framing}, only finitely many integer framings are excluded; for a good framing, ordinary recursion applies, exactly two simple cold ramification points enter the neck, and every other ramification point and every puncture lies in the fixed hot region, locally uniformly in $s$.

By Proposition~\ref{thm:period-split}, \eqref{eq:main-period} is a single-valued holomorphic transverse coordinate, and Proposition~\ref{prop:regime-atlas} gives
\[
t=\eps^2c_{\rm neck}(\eps^2,s)
\]
with $c_{\rm neck}$ holomorphic and bounded away from zero, together with the compatible orientation, marking, plumbing-circle, cross-cut, and primitive conventions.

By Theorem~\ref{thm:deck-descent}, after a common shrinking,
\[
F_g^c\in
\Ocal\bigl(
U_{\rm sp}\times(D_t\setminus\{0\})
\bigr)
\]
is jointly holomorphic and single-valued on the full punctured $t$-disc in both one-neck topologies. The descent accounts for the permutation of the two cold points, for the $A$-polarization of the conifold-adapted marking, and for changes of logarithm branches.

Recall the Euler defect of Section~\ref{sec:euler-descent},
\[
\mathcal E_g
=
(2-2g)F_g^c-t\partial_tF_g^c,
\]
for which Theorem~\ref{thm:euler-defect} gives
\[
\mathcal E_g=O(1)
\]
locally uniformly in $s$ and on the full punctured disc.

Apply Lemma~\ref{lem:assembly} with $U=U_{\rm sp}$, $u=s$, and $D=D_t$, and with
\[
F=F_g^c,\qquad
B=0,\qquad
H=\mathcal E_g,\qquad
I=\partial_tF_g^c,\qquad
R=0;
\]
with this assignment the lemma's $E$ is the Euler defect $\mathcal E_g$.
The first identity is the definition of $\mathcal E_g$:
\[
(2-2g)F_g^c
=
\mathcal E_g+t\partial_tF_g^c.
\]
The second holds because $R=0$:
\[
\partial_tF_g^c
=
\partial_tF_g^c+0.
\]
The functions $B$ and $R$ vanish, and $H=\mathcal E_g$ is $O(1)$ as $t\to0$, locally uniformly in $s$, by Theorem~\ref{thm:euler-defect}; the boundedness hypotheses therefore hold.

Corollary~\ref{cor:gaussian-limit-disc} supplies the remaining hypothesis:
\[
\lim_{t\to0}
t^{2g-2}F_g^c(s,t)
=
F_g^G
\]
on the full punctured disc, locally uniformly in $s$. The value $F_g^G$ is defined by the unit exact-period Gaussian data of Section~\ref{sec:gaussian-limit} alone, hence is independent of the spectator moduli and of all global toric data. Lemma~\ref{lem:assembly}, including its zero-dimensional variant (Remark~\ref{rem:assembly-zero-dim}), gives the unique jointly holomorphic $H_g^{\rm gap}$ and
\[
F_g^c(s,t)
=
F_g^Gt^{2-2g}+H_g^{\rm gap}(s,t).
\]
Thus $C_g=F_g^G$, which completes the proof of Theorem~\ref{thm:main-gap}.
\end{proof}

\subsection{Calibration of the universal constant}\label{sec:calibration}

This subsection evaluates $F_g^G$ in closed form, proves Corollary~\ref{cor:main-calibrated}, and deduces Theorem~\ref{thm:main-gap-intro}. Recall from Section~\ref{sec:normal-form} the Gaussian curve $G$ with Zhukovsky coordinate $\zeta$, the counterclockwise circle $\gamma_G$, and $\kappa=\int_{\gamma_G}w\,dz$; recall from Sections~\ref{sec:bergman} and~\ref{sec:gaussian-limit} the bidifferential $B_G$, the unit exact-period form $\overline\omega_{0,1}^{\,G}$, its correlators $\omega_{g,n}^G$, the local primitives $\Phi_G$ vanishing at $\zeta=\pm1$, and the free energy $F_g^G$ of \eqref{eq:gaussian-free-energy}; the Bernoulli numbers are those of Section~\ref{sec:assembly}. Throughout, the recursion kernel is \eqref{eq:EO-kernel}, the correlators obey \eqref{eq:EO-recursion}, and free energies follow \eqref{eq:main-free-energy}. In the proof below, a subscript or superscript $P$ marks the objects of \cite{ACPPRS12}, whose letters $p$, $q$, and $T$ are local to this subsection, and a hat marks the objects produced by the conventions of this paper from the initial form $-\tfrac12w\,dz$ and the bidifferential $B_G$.

\begin{thm}[Bernoulli value of the Gaussian constant]\label{thm:gaussian-calibration}
On the Gaussian curve $G$, the unit exact-period data of Section~\ref{sec:gaussian-limit} are
\[
\overline\omega_{0,1}^{\,G}
=
\frac{1}{c_{\rm per}\kappa}w\,dz,
\qquad
B_G(\zeta_1,\zeta_2)
=
\frac{d\zeta_1\,d\zeta_2}
{(\zeta_1-\zeta_2)^2},
\]
where $c_{\rm per}\in\C^*$ is the fixed period normalization of Section~\ref{sec:normal-form} and $\kappa$ is nonzero by Proposition~\ref{thm:period-split}. For every integer $g\geq2$,
\begin{equation}\label{eq:gaussian-calibration}
F_g^G
=
\frac{B_{2g}}{2g(2g-2)}
(2c_{\rm per}\kappa)^{2g-2}.
\end{equation}
For the counterclockwise $\gamma_G$ of Section~\ref{sec:normal-form}, Proposition~\ref{thm:period-split} gives
\[
\kappa=-\pi i.
\]
If in addition $c_{\rm per}=1/(2\pi i)$, then
\[
F_g^G=\frac{B_{2g}}{2g(2g-2)}.
\]
\end{thm}

\begin{proof}
\emph{Step 1: the data of \cite{ACPPRS12}.} The Gaussian calculation
\cite[Examples~3.1 and~4.1 and Equations~(35), (39), (42)--(44), and~(48)]{ACPPRS12}
uses
\[
x_P(p)=-\sqrt T\,(p+p^{-1}),
\qquad
y_P(p)=\frac{\sqrt T}{2}(p-p^{-1}),
\qquad
B_P(p_1,p_2)=
\frac{dp_1\,dp_2}{(p_1-p_2)^2}.
\]
The cut endpoints are
\[
a=-2\sqrt T,\qquad b=2\sqrt T.
\]
The recursion kernel of \cite{ACPPRS12} is
\[
K_P(q,p)
=
\frac{8q^3\,dp}
{(a-b)^2(q^2-1)(p-q)(qp-1)\,dq}.
\]
The stable correlators $W_{g,n}^P$ are generated from the data $(y_P\,dx_P,B_P)$ by the recursion of \cite{ACPPRS12}, which is \eqref{eq:EO-recursion} with $K_P$ in place of $K_a$. For $g\geq2$, the free-energy convention of \cite{ACPPRS12} is
\[
F_g^P
=
\frac{1}{2g-2}
\sum_{q_*=\pm1}
\operatorname*{Res}_{q=q_*}
\phi_P(q)W_{g,1}^P(q),
\qquad
d\phi_P=y_P\,dx_P,
\]
the sum running over the ramification points $p=\pm1$ of $x_P$, at which $\phi_P$ is a local primitive of $y_P\,dx_P$; the cited evaluation of these free energies is
\begin{equation}\label{eq:published-gaussian}
F_g^P
=
\frac{B_{2g}}
{2g(2g-2)T^{2g-2}}.
\end{equation}

\emph{Step 2: matching the curves and the kernels.} Specialize the cited data to $T=\tfrac14$ with $\sqrt T=\tfrac12$, and identify the uniformizations by $p=-\zeta$; then $a=-1$, $b=1$, and
\[
x_P=z,\qquad
y_P=-\frac12w,\qquad
y_P\,dx_P=-\frac12w\,dz,\qquad
B_P=B_G.
\]
Writing $q=-\zeta$ and $p=-\zeta_0$ in the kernel and using $(a-b)^2=4$ gives
\[
K_P(\zeta,\zeta_0)
=
\frac{2\zeta^3\,d\zeta_0}
{(\zeta^2-1)(\zeta_0-\zeta)(\zeta_0\zeta-1)\,d\zeta}.
\]
Here $\zeta$ is the residue variable and $\zeta_0$ the external one; the cited kernel lists them in this order, the reverse of \eqref{eq:EO-kernel}. The recursion kernel \eqref{eq:EO-kernel} of this paper, formed on $G$ from the initial form $-\tfrac12w\,dz$ and the bidifferential $B_G$ at either ramification point $\zeta=\pm1$, is, in the same argument order,
\begin{equation}\label{eq:kernel-comparison}
\widehat K^G(\zeta,\zeta_0)
=
-\frac{2\zeta^3\,d\zeta_0}
{(\zeta^2-1)(\zeta_0-\zeta)(\zeta_0\zeta-1)\,d\zeta}
=
-K_P(\zeta,\zeta_0).
\end{equation}

\emph{Step 3: the correlator sign.} Let $k=2g-2+n$ be the complexity of the stable type $(g,n)$. For the common data $(-\tfrac12w\,dz,B_G)$, with the sum in \eqref{eq:EO-recursion} over $\Ram(x_f)$ read as the two ramification points $\zeta=\pm1$ and with $\sigma(\zeta)=\zeta^{-1}$, one has, for every stable type $(g,n)$ and for $(g,n)=(0,2)$, where both sides equal $B_G$,
\begin{equation}\label{eq:correlator-sign}
\widehat\omega_{g,n}^G
=
(-1)^kW_{g,n}^P.
\end{equation}
To prove this, run the sign induction of \eqref{eq:physical-eo-sign} on $k$ in the recursion formula \eqref{eq:EO-recursion}: the base is $k=0$, where $\widehat\omega_{0,2}^G=B_G=W_{0,2}^P$; the genus-reduction term, when present, has complexity $k-1$, in every splitting term the complexities of the factors sum to $k-1$, and the prime in \eqref{eq:EO-recursion} excludes $(0,1)$ factors, so each term of the recursion source has sign $(-1)^{k-1}$, and the kernel sign \eqref{eq:kernel-comparison} gives $(-1)^k$.

\emph{Step 4: the free energies at $n=1$.} Since $d\phi_P=y_P\,dx_P=-\tfrac12w\,dz$, the primitives $\phi_P$ are local primitives of the initial form, as \eqref{eq:main-free-energy} requires. For $n=1$ the complexity is $k=2g-1$, which is odd, so \eqref{eq:correlator-sign} reads $\widehat\omega_{g,1}^G=-W_{g,1}^P$. Combining this with the two free-energy prefactors $1/(2-2g)$ and $1/(2g-2)$, which differ by a sign, gives
\begin{align}
\widehat F_g^G
&=
\frac{1}{2-2g}
\sum_{\alpha\in\{1,-1\}}
\operatorname*{Res}_{\zeta=\alpha}
\phi_P\,\widehat\omega_{g,1}^G
\nonumber\\
&=
-\frac{1}{2g-2}
\sum_{\alpha\in\{1,-1\}}
\operatorname*{Res}_{\zeta=\alpha}
\phi_P(-W_{g,1}^P)
\nonumber\\
&=
F_g^P\big|_{T=1/4}.
\label{eq:free-energy-comparison}
\end{align}

\emph{Step 5: exact-period rescaling.} For initial data $(\omega_{0,1},B)$ on $G$, write $\omega_{g,n}[\omega_{0,1},B]$ and $F_g[\omega_{0,1},B]$ for the correlators and free energies of this paper's conventions. If the initial one-form is multiplied by a nonzero scalar $\mu$, with $B$ and the projection, hence the ramification points, unchanged, the recursion kernel is multiplied by $\mu^{-1}$. The same complexity induction gives
\[
\omega_{g,n}[\mu\omega_{0,1},B]
=
\mu^{-(2g-2+n)}
\omega_{g,n}[\omega_{0,1},B].
\]
This is the scaling used in Step~3 of the proof of Proposition~\ref{prop:mixed-scaled}. The primitive in the free-energy residue is multiplied by $\mu$, while $\omega_{g,1}$ is multiplied by $\mu^{-(2g-1)}$, so
\begin{equation}\label{eq:free-energy-homogeneity}
F_g[\mu\omega_{0,1},B]
=
\mu^{2-2g}F_g[\omega_{0,1},B].
\end{equation}
The scalar carrying $-\tfrac12w\,dz$ to $\overline\omega_{0,1}^{\,G}$, that is, the solution of $\mu\bigl(-\tfrac12w\,dz\bigr)=(c_{\rm per}\kappa)^{-1}w\,dz$, is
\[
\mu=-\frac{2}{c_{\rm per}\kappa}.
\]
Using $F_g^G=F_g[\overline\omega_{0,1}^{\,G},B_G]$, then \eqref{eq:free-energy-homogeneity} with $\mu=-2/(c_{\rm per}\kappa)$, then \eqref{eq:free-energy-comparison}, and finally \eqref{eq:published-gaussian} at $T=1/4$,
\begin{align*}
F_g^G
&=
\left(
-\frac{2}{c_{\rm per}\kappa}
\right)^{2-2g}
\frac{B_{2g}}{2g(2g-2)}
4^{2g-2}\\
&=
\frac{B_{2g}}{2g(2g-2)}
(2c_{\rm per}\kappa)^{2g-2},
\end{align*}
because $2-2g$ is even, so the sign in $\bigl(-2/(c_{\rm per}\kappa)\bigr)^{2-2g}$ drops out. This proves \eqref{eq:gaussian-calibration}.

\emph{Step 6: evaluation of $\kappa$ and the normalized value.} Proposition~\ref{thm:period-split} gives $\kappa=-\pi i$ for the counterclockwise $\gamma_G$. If
\[
c_{\rm per}=\frac{1}{2\pi i},
\]
then
\[
2c_{\rm per}\kappa=-1,
\]
whose $(2g-2)$-nd power is one, so \eqref{eq:gaussian-calibration} becomes $F_g^G=B_{2g}/(2g(2g-2))$.
\end{proof}

We can now prove the corollary stated in Section~\ref{sec:assembly}.

\begin{proof}[Proof of Corollary~\ref{cor:main-calibrated}]
The proof of Theorem~\ref{thm:main-gap} identifies $C_g=F_g^G$, and substituting \eqref{eq:gaussian-calibration} into this identity gives \eqref{eq:main-calibrated}. Euler's evaluation
\[
\sum_{m\geq1}m^{-2g}=\frac{(2\pi)^{2g}|B_{2g}|}{2\,(2g)!},
\]
whose left-hand side is positive, gives $B_{2g}\neq0$; since $c_{\rm per}\neq0$, $\kappa\neq0$, and $2g(2g-2)\neq0$, the coefficient $C_g$ is nonzero. The value $F_g^G$ is determined by the unit exact-period Gaussian data alone, so $C_g$ depends only on $c_{\rm per}$ and the recursion and free-energy conventions, and is independent of $\mathfrak X$, of $\Delta$, and of the point $s$ of the smooth conifold divisor. Reversing the orientation of $\gamma$ reverses that of $\gamma_G$, since the two are identified by $\Phi_\eps$ (Section~\ref{sec:assembly}); this replaces $\kappa$ by $-\kappa$ and, by \eqref{eq:main-period}, $t$ by $-t$, and the even exponents $2g-2$ in \eqref{eq:main-calibrated} and $2-2g$ in \eqref{eq:main-gap} leave $C_g$ and the polar term unchanged. The normalized value at $c_{\rm per}=1/(2\pi i)$ is the final display of Theorem~\ref{thm:gaussian-calibration}.
\end{proof}

\begin{proof}[Proof of Theorem~\ref{thm:main-gap-intro}]
Specialize Theorem~\ref{thm:main-gap} to $c_{\rm per}=1/(2\pi i)$ and evaluate the constant by Corollary~\ref{cor:main-calibrated}: this turns \eqref{eq:main-gap} into \eqref{eq:intro-gap}. The corollary's orientation-reversal clause makes the constant independent of the orientation of $\gamma$, which Theorem~\ref{thm:main-gap-intro} leaves unspecified; every remaining assertion is inherited from Theorem~\ref{thm:main-gap}.
\end{proof}


\begin{thebibliography}{99}

\bibitem[ABDKS23]{ABDKS23}
Alexander Alexandrov, Boris Bychkov, Petr Dunin-Barkowski, Maxim Kazarian, and Sergey Shadrin,
\emph{Log topological recursion through the prism of $x$-$y$ swap},
Int. Math. Res. Not. IMRN 2024, no. 21, 13461--13487,
arXiv:2312.16950,
doi:10.1093/imrn/rnae213.

\bibitem[Ali25]{Ali25}
Murad Alim,
\emph{Intrinsic non-perturbative topological strings},
Adv. Theor. Math. Phys. 29 (2025), no. 5, 1365--1406,
arXiv:2102.07776,
doi:10.4310/atmp.251023013812.

\bibitem[ACPPRS12]{ACPPRS12}
J{\o}rgen Ellegaard Andersen, Leonid O. Chekhov, R. C. Penner, Christian M. Reidys, and Piotr Su{\l}kowski,
\emph{Topological recursion for chord diagrams, RNA complexes, and cells in moduli spaces},
Nuclear Phys. B 866 (2013), no. 3, 414--443,
arXiv:1205.0658,
doi:10.1016/j.nuclphysb.2012.09.012.

\bibitem[BKMP09]{BKMP09}
Vincent Bouchard, Albrecht Klemm, Marcos Mari\~no, and Sara Pasquetti,
\emph{Remodeling the B-model},
Comm. Math. Phys. 287 (2009), no. 1, 117--178,
arXiv:0709.1453.

\bibitem[BKMP10]{BKMP10}
Vincent Bouchard, Albrecht Klemm, Marcos Mari\~no, and Sara Pasquetti,
\emph{Topological open strings on orbifolds},
Comm. Math. Phys. 296 (2010), no. 3, 589--623,
arXiv:0807.0597,
doi:10.1007/s00220-010-1020-0.

\bibitem[Bri25]{Bri25}
Andrea Brini,
\emph{Conifold gap and all-genus mirror symmetry for local $\mathbb{P}^2$},
arXiv:2509.19298, 2025.

\bibitem[CEO06]{CEO06}
Leonid Chekhov, Bertrand Eynard, and Nicolas Orantin,
\emph{Free energy topological expansion for the 2-matrix model},
J. High Energy Phys. 2006, no. 12, 053,
arXiv:math-ph/0603003,
doi:10.1088/1126-6708/2006/12/053.

\bibitem[EO07]{EO07}
B.~Eynard and N.~Orantin,
\emph{Invariants of algebraic curves and topological expansion},
Commun. Number Theory Phys. \textbf{1} (2007), no.~2, 347--452,
doi:10.4310/CNTP.2007.v1.n2.a4,
arXiv:math-ph/0702045.
Equation and footnote numbering follows the arXiv version.

\bibitem[EO15]{EO15}
Bertrand Eynard and Nicolas Orantin,
\emph{Computation of open Gromov--Witten invariants for toric Calabi--Yau 3-folds by topological recursion, a proof of the BKMP conjecture},
Comm. Math. Phys. 337 (2015), no. 2, 483--567,
arXiv:1205.1103.

\bibitem[FLZ20]{FLZ20}
Bohan Fang, Chiu-Chu Melissa Liu, and Zhengyu Zong,
\emph{On the remodeling conjecture for toric Calabi--Yau 3-orbifolds},
J. Amer. Math. Soc. 33 (2020), no. 1, 135--222,
arXiv:1604.07123.

\bibitem[FP00]{FP00}
Carel Faber and Rahul Pandharipande,
\emph{Hodge integrals and Gromov--Witten theory},
Invent. Math. 139 (2000), no. 1, 173--199,
arXiv:math/9810173.

\bibitem[GhV95]{GhV95}
Debashis Ghoshal and Cumrun Vafa,
\emph{$c=1$ string as the topological theory of the conifold},
Nuclear Phys. B 453 (1995), no. 1-2, 121--128,
arXiv:hep-th/9506122.

\bibitem[GoV98]{GoV98}
Rajesh Gopakumar and Cumrun Vafa,
\emph{M-theory and topological strings II},
arXiv:hep-th/9812127, 1998.

\bibitem[GKN19]{GKN19}
Samuel Grushevsky, Igor Krichever, and Chaya Norton,
\emph{Real-normalized differentials: limits on stable curves},
Russian Math. Surveys 74 (2019), no. 2, 265--324,
doi:10.1070/RM9877.

\bibitem[HK07]{HK07}
Min-xin Huang and Albrecht Klemm,
\emph{Holomorphic anomaly in gauge theories and matrix models},
J. High Energy Phys. 2007, no. 9, 054,
arXiv:hep-th/0605195.

\bibitem[HKQ09]{HKQ09}
Min-xin Huang, Albrecht Klemm, and Seth Quackenbush,
\emph{Topological string theory on compact Calabi--Yau: modularity and boundary conditions},
Homological mirror symmetry, Lecture Notes in Phys. 757, Springer, Berlin, 2009, pp. 45--102,
arXiv:hep-th/0612125.

\bibitem[HMO26]{HMO26}
Alexander Hock, Olivier Marchal, and Nicolas Orantin,
\emph{Geometry of Logarithmic Topological Recursion: Dilaton Equations, Free Energies and Variational Formulas},
arXiv:2604.25622 [math-ph], 2026.

\bibitem[HN20]{HN20}
Xuntao Hu and Chaya Norton,
\emph{General variational formulas for abelian differentials},
Int. Math. Res. Not. IMRN 2020 (2020), no. 12, 3540--3581,
doi:10.1093/imrn/rny106.

\bibitem[HZ86]{HZ86}
John Harer and Don Zagier,
\emph{The Euler characteristic of the moduli space of curves},
Invent. Math. 85 (1986), no. 3, 457--485.

\bibitem[IILZ25]{IILZ25}
Nikolai Iorgov, Kohei Iwaki, Oleg Lisovyy, and Yurii Zhuravlov,
\emph{Many-faced Painlev\'e I: irregular conformal blocks, topological recursion, and holomorphic anomaly approaches},
arXiv:2505.16803, 2025.

\bibitem[IK22]{IK22}
Kohei Iwaki and Omar Kidwai,
\emph{Topological recursion and uncoupled BPS structures I: BPS spectrum and free energies},
Adv. Math. 398 (2022), Paper No. 108191,
arXiv:2010.05596.

\bibitem[IKoT19]{IKoT}
Kohei Iwaki, Tatsuya Koike, and Yumiko Takei,
\emph{Voros coefficients for the hypergeometric differential equations and Eynard--Orantin's topological recursion, Part II: for the confluent family of hypergeometric equations},
J. Integrable Syst. 4 (2019), no. 1, xyz004,
arXiv:1810.02946.

\bibitem[JGJ26]{JGJ26}
Haocheng Ju, Guoxiong Gao, Jiedong Jiang, Bin Wu, Zeming Sun, Shurui Liu, Leheng Chen, Yutong Wang, Yuefeng Wang, Zichen Wang, Wanyi He, Peihao Wu, Liang Xiao, Ruochuan Liu, Bryan Dai, and Bin Dong,
\emph{Automated conjecture resolution with formal verification},
arXiv:2604.03789, 2026.

\bibitem[LGS26]{LGS26}
Jihao Liu, Guoxiong Gao, Zeming Sun, Bin Wu, Shurui Liu, Jiedong Jiang, Haocheng Ju, Leheng Chen, Ronnie Cheng, Xiping Zhang, and Bin Dong,
\emph{Danus: orchestrating mathematical reasoning agents with fact-graph memory},
arXiv:2607.06447, 2026.

\bibitem[Mar08]{Marino08}
Marcos Mari\~no,
\emph{Open string amplitudes and large order behavior in topological string theory},
J. High Energy Phys. 2008, no. 3, 060, 34 pp.,
arXiv:hep-th/0612127,
doi:10.1088/1126-6708/2008/03/060.

\bibitem[PS10]{PS10}
Sara Pasquetti and Ricardo Schiappa,
\emph{Borel and Stokes nonperturbative phenomena in topological string theory and $c=1$ matrix models},
Ann. Henri Poincar\'e 11 (2010), 351--431,
arXiv:0907.4082,
doi:10.1007/s00023-010-0044-5.

\bibitem[Str95]{Str95}
Andrew Strominger,
\emph{Massless black holes and conifolds in string theory},
Nuclear Phys. B 451 (1995), no. 1-2, 96--108,
arXiv:hep-th/9504090.

\bibitem[Yam80]{Yam80}
Akira Yamada,
\emph{Precise variational formulas for abelian differentials},
Kodai Math. J. 3 (1980), no. 1, 114--143,
doi:10.2996/kmj/1138036124.

\end{thebibliography}
\end{document}